%% file: main.tex
\PassOptionsToPackage{dvipsnames}{xcolor}
\documentclass[leqno]{siamart251104}

\usepackage{mydef}
\usepackage[english]{babel}

\usepackage{amsmath, amsfonts}
\usepackage{mathrsfs}
\usepackage{bbm}
\usepackage{xcolor}
\usepackage{graphicx}
\usepackage{tikz}
\usepackage{algorithm}
\usepackage{algpseudocode}
\usepackage{subcaption}
\usepackage[
  mode = math,
  inter-unit-product = \ensuremath{\,},
]{siunitx}
\usetikzlibrary{arrows.meta}

\usepackage{tikz}
\usetikzlibrary{arrows,decorations,shapes, patterns}

\usepackage{pgfplots}
\pgfplotsset{compat=1.18}

\pgfmathsetmacro{\sieL}{7.3}
\pgfmathsetmacro{\sieR}{0.1117}
\pgfmathsetmacro{\vL}{3.9}
\pgfmathsetmacro{\vR}{0.5}
\pgfmathsetmacro{\cv}{0.000389}
\pgfmathsetmacro{\Tnot}{150}
\pgfmathdeclarefunction{atan_interp}{4}{%
  \pgfmathparse{0.5 * (#1 + #2) + 0.5 * (#2 - #1) * atan(#3 * #4) / atan(#3)}
}

\usepackage{enumitem}
\setlist[enumerate]{leftmargin=.5in}
\setlist[itemize]{leftmargin=.5in}

\usepackage{physics}
\usepackage[numbers]{natbib}

\makeatletter
\renewcommand*\env@matrix[1][*\c@MaxMatrixCols c]{%
  \hskip -\arraycolsep
  \let\@ifnextchar\new@ifnextchar
  \array{#1}}
\makeatother

\begin{document}

\newcommand\footnotemarkfromtitle[1]{%
	\renewcommand{\thefootnote}{\fnsymbol{footnote}}%
	\footnotemark[#1]%
	\renewcommand{\thefootnote}{\arabic{footnote}}}

\newcommand{\TheTitle}{Analysis, thermodynamics, and a numerical solver for a pressure-temperature equilibrium closure of the four-equation model}
\newcommand{\TheAuthors}{B. Clayton, J. McConnell, C. Solomon}

\headers{Thermodynamics of PTE}{\TheAuthors}

\title{{\TheTitle}\thanks{Draft version, \today \funding{
			BC, JM, and CS, acknowledge the support of the Eulerian Applications Project (within the Advanced Simulation and Computing program) at LANL.
            LANL is operated by Triad National Security, LLC, for the National Nuclear Security Administration of the U.S. Department of Energy (Contract No. 89233218CNA000001).}}}

\author{
	Bennett Clayton\footnotemark[2]
	\and 
	Joshua McConnell\footnotemark[2]
    \and 
    Clell Solomon\footnotemark[2]
}

\maketitle

\renewcommand{\thefootnote}{\fnsymbol{footnote}}

\footnotetext[2]{%
  X Computational Physics Division, Los Alamos National Laboratory, P.O. Box 1663, Los Alamos, NM, 87545, USA.}

\renewcommand{\thefootnote}{\arabic{footnote}}

\begin{abstract}
    We analyze an often used closure model for multi-material hydrodynamics where pressure-temperature equilibrium (PTE) is assumed for every state; emphasis is placed on tabular equations of state.
    This multi-material model is often referred to as the four-equation model.
    The identification of the admissible set is presented and is proven to be convex, setting the foundation for development of invariant-domain preserving methods for this model.
    A novel numerical method is presented for solving the highly nonlinear system for the equilibrated pressure and temperature with an arbitrary number of materials.
    This new method is compared with some traditional iterative solvers through a collection of different tests.
    Additionally, we provide a detailed analysis of the thermodynamics of the mixture model for general equations of state and prove existence and uniqueness of the pressure-temperature equilibrium solution under some thermodynamic assumptions.
\end{abstract}

\begin{keywords}
  thermodynamics, tabulated equation of state, mixture model, pressure temperature equilibrium, iterative methods, Newton methods, four-equation model, admissible set.
\end{keywords}
  
\input{introduction}
\input{model_problem}
\input{pte_solver}
\input{cyclic_method}
\input{eos}
\input{numerical_results}
\input{conclusion}

\appendix

\input{appendix}

\section*{Acknowledgments}
The authors thank their colleagues from Los Alamos National Laboratory: Tariq Aslam, Eduardo Lozano, and Jeffery Peterson for their insightful discussions on thermodynamics and equations of state.

\bibliography{sample}

\end{document}

%% file: introduction.tex
\section{Introduction}
The simulation of multi-material flows is a complex physical process; wherein, many choices of models are available to describe this process. 
These often depend on the material phases, length scales, and the application of interest.
A common choice for modeling approximately inviscid compressible flows is the compressible Euler equations.
As is well known, the single material compressible Euler equations are underdetermined and must be closed with an equation of state (EOS).
Similarly, for multiple materials (multi-component compressible Euler), a closure model must be prescribed for how materials interact.
There have been a variety of models and assumptions derived from the Euler equations to model multi-material interactions.
In the literature, they are often referred to as the four-, five-, six-, and seven-equation models, indicative of the number of equations to solve when only two materials are present.
For example, the five-equation model, \citet[Sec.~2.3]{shyue1998efficient} and \citet{allaire2002five}, assumes that the pressure is equilibrated between the materials, and has an additional equation to be solved for the volume fraction. 
The six-equation model developed in \citet{saurel2009simple} drops the pressure equilibrium assumption from \citep{allaire2002five} which results in two energy equations to solve.
Lastly, the seven-equation models of \citet{saurel1999multiphase} and \citet{baer1986two} have all of the previous assumptions but now assumes that material velocities are not in equilibrium.
In \citet[Sec.~4]{flaatten2011relaxation} a hierarchy of some of these models (referred to as ``relaxation systems'') is presented; see also \citet{flaatten2010wave} and \citet{lund2012hierarchy}.
For more discussion about these different models (among others) see \citet{saurel2008modelling, saurel2009simple}.
Another multi-material model which does not fit into the previously mentioned modeling framework is the work of \citet[Sec.~2.2]{shyue1998efficient} which introduces a $\gamma$-based model for mixtures of stiffened gas EOS.

The focus of this paper is on the four-equation model which assumes pressure, temperature, and velocity equilibrium of the mixtures. Specifically, we focus on the analysis, solvability, and numerical algorithms for the pressure-temperature equilibrium (PTE) system.
This model is often coupled with other chemical and physical effects; however, the underlying assumption is that only two mass equations, one momentum equation and one energy equation are solved.
There are also many applications of this model in several areas, see \citep{lund2011depressurization, de2017homogeneous, saurel_2024, peden2026large}; however, the specific details for solving for the equilibrated pressure and temperature are often omitted.
Furthermore, we assume that the material mass fractions remain unchanged during the pressure-temperature equilibrium closure of the four-equation model; i.e. chemical equilibrium and mass transfer are not imposed.
This is seen in our assumption on the Gibbs free energy of the mixture; that is, the mixture Gibbs free energy is defined to be the mass fraction average of each of the component Gibbs free energies.
This is unlike the commonly used \textit{homogeneous equilibrium model} (HEM) which assumes equality of each material Gibbs free energy (see \citet[Sec.~4.3]{stewart1984two}, \citet{clerc2000numerical}, \citet{foll2019use} and references therein).
The general four-equation model was popularized in \citet{saurel2016general} under a relaxation of the Gibbs free energy towards equilibrium; earlier investigations into this model are seen in \citet{flaatten2010wave}.
In many cases, the EOS used are typically an ideal gas mixture or an ideal and stiffened gas mixture; see \citep{gouasmi2020formulation, renac2021entropy, wong2022positivity, sementilli2023scalable, trojak2024positivity, menshov2025sharp, CLAYTON2026107159,  collis2026robust}.
For more general equations of state, \citet{aslam_2021} proposed a method for obtaining the equilibrated pressure and temperature for a mixture of Mie-Gr{\"u}neisen equations of state.
Further analysis of this model is provided in \citet{rai2022evaluation} which also includes alternative closure models to PTE.
The general thermodynamics of this mixture model is discussed in more detail in \citet{grove2010pressure, grove2019some}.

One of the difficult challenges with solving the four-equation model is to rapidly, robustly, and accurately calculate the equilibrated pressure and temperature state of the mixture as the pressure is necessary for evolution of the conserved variables.
In the chemical engineering field, minimization methods are typically applied to the Gibbs free energy to compute the phase equilibrium state.
Many examples can be found in the literature, for example, \citep{du1987phase, teh2002study, sofyan2003multiphase, zhang2011review, xu2023precipitation}; however, the model assumptions vary greatly, and are not directly related to the four-equation model.
The PTE model that we are concerned with is simpler and we place emphasis on equations of state which are thermodynamically stable.
Therefore we target the nonlinear problem with direct iterative solvers rather than minimization techniques.

\subsection{Novel contributions}
To the authors best knowledge, little has been published regarding the mixture of arbitrary or tabular equations of state under PTE.
The purpose of this paper is twofold --
one, to present a rigorous analysis of the PTE closure and the constraints for which a solution exists,
and, two, provide a robust numerical algorithm for solving the nonlinear system for pressure and temperature.
The existence of the PTE solution is directly related to the admissible set for the four-equation model which can be used for development of positivity-preserving or invariant-domain preserving numerical methods.
The numerical algorithm we introduce for solving PTE employs the \textit{cyclic rule} in order to construct tangent line approximations in the $P$-$T$ solution space, from which a simple line intersection calculation provides an approximation of the root.
This is combined with simple one-dimensional Newton steps, to construct a robust iterative solver.
While this method is often more expensive than typical Newton solvers, it is more robust to the initial starting location based on our numerical experiments.
This novel method is not specific to thermodynamics and could be applied to any nonlinear system from $\Real^2$ to $\Real^2$ (assuming the necessary derivatives and tangent lines exist).
In summary the novel contributions are:
\begin{itemize}
    \item Identification of the admissible set for the four-equation model with arbitrary equations of state.
    \item Proof of existence and uniqueness for solutions to the PTE system.
    \item Development of a novel iterative method for solving nonlinear equations from $\Real^2$ to $\Real^2$ which utilizes the cyclic rule.
    \item Application of this method to the PTE system.
    \item A numerical study of different iterative methods for solving the PTE system with tabulated equations of state.
\end{itemize}

\subsection{Organization}
In Section~\ref{sec:pte_model} we describe the four-equation model and provide a brief review of necessary thermodynamics for single material and mixtures.
In Section~\ref{sec:admissible_set} the admissible set is determined which outlines a necessary condition for existence of solutions to the nonlinear PTE system.
The main theoretical result of this paper is provided in Theorem~\ref{thm:existence} which outlines the necessary assumptions for existence of a unique PTE solution.
In Section~\ref{sec:tabular_approximation} we introduce a tabular equation of state approximation and discuss some of the necessary details and complications with using tabular EOS.
In Section~\ref{sec:pte_solving} we introduce the numerical details for solving PTE as well as some common methods for solving nonlinear systems.
Then in Section~\ref{sec:the_cyclic_method}, the novel numerical algorithm is introduced in a general form with application to the PTE system.
The details of several equations of state that we use for numerical illustrations are outlined in Section~\ref{sec:eos} and we close with a variety of numerical tests presented in Section~\ref{sec:numeric_results}.

%% file: model_problem.tex
\section{The four-equation model} \label{sec:pte_model}%
The four-equation model, which we also refer to as the multi-component Euler equations, is described by the following equations
\begin{subequations}
\begin{align}
    \partial_t \big((\alpha_m \rho_m)(\bx, t)\big) + \nabla \cdot \Big( (\alpha_m \rho_m)(\bx, t) \frac{\bsfm(\bx, t)}{\rho(\bx, t)} \Big) &= 0, \label{eq:mass_cons} \\ 
    \partial_t \bsfm(\bx, t) + \nabla \cdot \Big(\frac{\bsfm(\bx, t)}{\rho(\bx, t)} \otimes \bsfm(\bx, t) + P(\bx, t) \polI_d\Big) &= \mathbf{0}, \label{eq:momentum_cons} \\ 
    \partial_t E(\bx, t) + \nabla \cdot \Big( \frac{\bsfm(\bx, t)}{\rho(\bx, t)} (E(\bx,t) + P(\bx,t))\Big) &= 0, \label{eq:total_energy_cons}
\end{align}
\end{subequations}
for $m \in \intset{1}{M}$ the index for each material with $M$ denoting the number of materials (also referred to as species or components), $\bx \in \Real^d$ is the spatial coordinate where $d$ is the spatial dimension, and $t > 0$ is the time.
Without loss of generality, we assume that all $M$ materials are present, since if some materials are absent, we can simply re-index the materials over the set $\intset{1}{M'}$ for $M' < M$.
We denote $\polI_d$ to be the $d\times d$ identity matrix.
The conservative quantities are: the partial density (or conserved density), $\alpha_m\rho_m$, for material $m$, the momentum, $\bsfm$, and, $E$, the total energy.
The equilibrated pressure is denoted by $P$ and will sometimes be referred to as a generalized pressure the details are discussed in Section~\ref{sec:pte_closure}.
The total energy is the sum of the internal and kinetic energies; that is, $E = \rho e + \frac12 \rho \Vert \bv \Vert^2_{\ell^2}$, where $\rho \eqq \sum_{m=1}^M \alpha_m \rho_m$ is the total density, $e$ is the specific internal energy of the mixture (see Section~\ref{sec:pte_closure}), $\bv \eqq \bsfm/\rho$ is the velocity and $\Vert \cdot \Vert^2_{\ell^2}$ is the standard Euclidean norm ($\ell^2$-norm).
The definition of $\rho$ is justified in Assumption~\ref{ass:mix_gibbs}.
The specific internal energy can be computed from the conservative variables by $e = \sfe(\bu) \eqq \rho^{-1} (E - \frac{1}{2\rho} \Vert \bsfm \Vert^2_{\ell^2})$.

Each material density and specific internal energy can be computed through its respective equation of state.
In particular, $\rho_m = \rho_m(P, T)$ and $e_m = e_m(P,T)$ where $T$ is the equilibrated temperature.
From this, each material volume fraction is computed by $\alpha_m \eqq \frac{\alpha_m \rho_m}{\rho_m(P, T)}$.
The material mass fraction is defined by, $Y_m \eqq \frac{\alpha_m\rho_m}{\rho}$.
It is also convenient to define the vector form of these quantities; that is, $\balpha \eqq (\alpha_1, \ldots, \alpha_M)^\sfT$ and $\bY \eqq (Y_1, \ldots, Y_M)^\sfT$ with, ${}^\sfT$, denoting the transpose.
We also introduce the simplex set, $\Delta_M \eqq \{\by \in \Real^M : \sum_{m=1}^M y_m = 1\}$ since $\bY, \balpha \in \Delta_M$.
It is advantageous to work with the specific volume, which we define by, $\tau \eqq \rho^{-1}$ and similarly for each material specific volume, $\tau_m \eqq \rho_m^{-1}$.
Note the important relation 
\begin{equation}
    \rho Y_m = \alpha_m \rho_m  \quad \Leftrightarrow \quad Y_m \tau_m = \alpha_m \tau,
\end{equation}
hence $\tau = \sum_{m=1}^M Y_m \tau_m$.
Lastly, it is useful to define the ``active material'' set,
\begin{equation} \label{eq:active_material_index_set}
    I(\bY) \eqq \{ m \in \intset{1}{M} : Y_m > 0\}.
\end{equation}
Note, the indices are not re-indexed, just that a subset of indices have been selected from $\intset{1}{M}$ based on the present materials.

\begin{remark}[Density well-posedness] \label{rem:well_posed_density}
    For many equations of state; in particular cubic equations of state (see \citet{valderrama2003state}), $\rho_m(P,T)$, may be multi-valued.
    This occurs when thermodynamic stability (see Definition~\ref{def:thermo_stability}) fails.
    For the scope of this paper, we only concern ourselves with EOS for which $\rho_m(P,T)$ is well-defined.
    It is also possible to restrict ourselves to regions of the phase space where $\rho_m(P,T)$ is well-defined; however, this introduces additionally complexity beyond our current focus.
\end{remark}

\subsection{Single species thermodynamics} \label{sec:thermo}
We now proceed with a review of some necessary thermodynamics for a single species.
We drop the material, ${}_m$, subscript to simplify notation.
Much of what is discussed here can be found in: \citet{callen1998thermodynamics}, \citet{planck2013treatise}, \citet{fermi2012thermodynamics} and \citet{menikoff1989riemann}.
In order to define a complete equation of state, one can first begin with a thermodynamic potential like the Gibbs or Helmholtz free energy (assuming the underlying Legendre transform is well-defined) from which all thermodynamic information can be derived.
Similarly, if $e = e(\tau, s)$ is known, where $s$ is the specific entropy, then all thermodynamic information can be derived by using the Gibbs' identity: $T \diff s = \diff e + P \diff \tau$.
That is, the pressure and temperature are defined by $T \eqq \big( \pdv{e}{s}\big)_{\tau}$ and $P \eqq -\big( \pdv{e}{\tau}\big)_{s}$, respectively.

As we shall see in Section~\ref{sec:pte_closure}, the Gibbs free energy becomes the natural starting point when working with the pressure-temperature equilibrium closure.
Recall,
\begin{equation}
    G(P,T) = \inf_{\tau,s} \big(e(\tau,s) - Ts + P\tau\big).
\end{equation}
Given $(P,T)$, we assume the infimum exists whence $e = e(\tau(P,T), s(P,T))$ is well-defined.
We assume that $T > 0$ and $P$ belongs to some unbounded interval, $\calP$ (see Section~\ref{sec:asymp_thermo}).
We also reference the specific internal energy as $e = e(P,T)$ rather than $e = e(\tau(P,T), s(P,T))$,
since $e(P,T)$ is more natural under the PTE closure as seen in Section~\ref{sec:pte_closure}.
Then we have the following identities: $\tau(P,T) \eqq \big( \pdv{G}{P}\big)_T$ and $s(P,T) \eqq -\big( \pdv{G}{T}\big)_P$.
The \textit{coefficient of thermal expansion} is defined by $\beta \eqq \frac{1}{\tau} \big(\pdv{\tau}{T} \big)_P$.
The \textit{specific heat capacity at constant pressure} is provided by $c_p \eqq T \big( \pdv{s}{T} \big)_P$ and the \textit{isothermal compressibility factor} is $K_T \eqq -\frac{1}{\tau} \big( \pdv{\tau}{P} \big)_T$.
Assuming $K_T > 0$ and $c_p > 0$, the specific volume and specific entropy can be inverted; that is, $P = P(\tau, T)$ and $T = T(s, P)$.
We are now able to define the \textit{specific heat capacity at constant volume} by, $c_v \eqq T \partial_T s(P(\tau, T), T)$.
Proceeding, we wish to show that $P$ and $T$ can both be written as functions of $\tau$ and $s$.
As we see, the pressure is defined through the implicit equation, $P = P(\tau, T(s, P))$.
Writing $g(\tau, s, P) \eqq P - P(\tau, T(s,P))$ we shall show that $g(\tau, s, P) = 0$ can be solved for $P$; this provides the pressure in the form, $P = p(\tau, s)$.
Existence is guaranteed since $\tau$ and $s$ are assumed to be defined on the range of $\tau(P,T)$ and $s(P,T)$.
Differentiating $g$ with respect to $P$, we have, $\partial_P g = 1 - \big( \pdv{P}{T} \big)_\tau \big( \pdv{T}{P} \big)_s$.
Using the cyclic rule (Corollary~\ref{cor:cyclic_rule}), we have that $\big(\pdv{P}{T}\big)_\tau = - \big( \pdv{\tau}{T}\big)_P / \big( \pdv{\tau}{P} \big)_T = \frac{\beta}{K_T}$ and $\big( \pdv{T}{P} \big)_s = - \big( \pdv{s}{P} \big)_T \big( \pdv{T}{s} \big)_P = \frac{\tau \beta T}{c_p}$ where we have used Theorem~\ref{thm:clairaut_shwarz} on the Gibbs free energy to find $\big(\pdv{s}{P}\big)_T = -\big(\pdv{\tau}{T}\big)_P = -\tau \beta$.
Therefore, $\partial_P g = 1 - \frac{\tau \beta^2 T}{c_p K_T} > 0$ which is guaranteed by thermodynamic stability (Definition~\ref{def:thermo_stability}).
Hence, the pressure is defined uniquely as $P = p(\tau, s)$ since $g$ is a strictly monotonic function of $P$; similarly, we als have that $T = \theta(\tau, s)$ (we use $p$ and $\theta$ to avoid abuse of notation).
From this, the \textit{isentropic bulk modulus} is given by $B_s \eqq -\tau \big(\pdv{p}{\tau}\big)_s$.
Furthermore the \textit{isentropic compressibility} and \textit{isothermal bulk modulus} are given by $K_s \eqq 1 / B_s$ and $B_T \eqq 1 / K_T$, respectively.
We now note that the (isentropic) sound speed is defined by $c = \sqrt{B_s / \rho}$, hence for the Euler equations to be hyperbolic, we require that $B_s > 0$.
This ties directly into the notion of thermodynamic stability.
\begin{definition}[Thermodynamic stability] \label{def:thermo_stability}
    An equation of state is said to be thermodynamically stable (\citet[Sec II. C.]{menikoff1989riemann}) if 
    \begin{equation}
        c_v^{-1} \geq c_p^{-1} \geq 0 \quad \text{ and } \quad K_s^{-1} \geq K_T^{-1} \geq 0.
    \end{equation}
    This is equivalent to $e(v,s)$ being jointly convex.
\end{definition}

\begin{remark}[Unstable EOS]
For many equations of state, there exists a region in phase space where $B_T < 0$; e.g. cubic equations of state like the van der Waals EOS (see \citet{valderrama2003state}).
This non-physical region is often attributed to phase change of the material (e.g. vaporization of water; see \citet[Part 1.~Chpt 1.]{planck2013treatise}) but can also occur when a solid expands to a vacuum state (see~\citet{fecht1990thermodynamic}).
To avoid such complexities with phase change and unstable EOS, we restrict ourselves to EOS which satisfy the stability criteria in Definition~\ref{def:thermo_stability}.
As described in Remark~\ref{rem:well_posed_density}, this assumption guarantees that $\rho(P,T)$ is well-defined.
\end{remark}

\subsection{Thermodynamic asymptotics} \label{sec:asymp_thermo}
We require some information regarding the asymptotic behavior of our equation of state.
These asymptotic properties are typical of many EOS.
We make the following assumptions:
\begin{equation} \label{eq:asymptotics}
    \lim_{\tau \to 0^+} P(\tau, T) = +\infty \quad \text{ and } \quad  \lim_{T \to \infty} e(P(\tau, T), T) = +\infty.
\end{equation}
That is, the pressure tends to infinity under infinite isothermal compression and the energy tends to infinity under isochoric heating.

We also require the notion of a minimum pressure.
Let $\calP$ be an interval denoting the admissible range of pressure values.
The infimium of the pressure is given by,
\begin{equation}
    \inf(\calP) = \inf_{\substack{T > 0 \\ \tau > 0 }} P(\tau, T).
\end{equation}
In general, real equations of state satisfy $|\inf(\calP)| < \infty$.
For an ideal gas, $\inf(\calP) = 0$, and for a stiffened gas, $\inf(\calP) = -P_\infty$.
Furthermore, under the assumption of thermodynamic stability (Definition~\ref{def:thermo_stability}), the minimum pressure must occur as $\tau \to \infty$.

\subsection{Mixture thermodynamics} \label{sec:pte_closure}
We are now in a position to discuss the closure of the system \eqref{eq:mass_cons}--\eqref{eq:total_energy_cons}.
To close this system, a pressure-temperature equilibrium assumption is imposed (see \citet[Sec.~3]{grove2010pressure}).
Specifically, one assumes that there exists some $P \in \Real$ and $T \in \Real_+$, such that
\begin{equation} \label{eq:pte_assumption}
    P = P_m(\tau_m, T_m) \quad \text{ and } \quad T = T_m(s_m, P_m),
\end{equation}
for all $m \in \intset{1}{M}$.
(The subscript $m$ is used to denote a specific material quantity.)
In fact, under appropriate assumptions, the existence and uniqueness of such a solution will be proven in Theorem~\ref{thm:existence}.
\begin{remark}[Single Material PTE]
    If $\alpha_k \rho_k = \rho$ for some $k \in \intset{1}{M}$; i.e., $\alpha_i\rho_i = 0$ for all $i \in \intset{1}{M} \setminus\{k\}$, then the PTE system holds trivially.
    The pressure and temperature are then defined by their respective EOS.
\end{remark}
Note however, that the equilibrated pressure, $P$, must be valid for each equation of state in the mixture; that is, we require $P \in \calP_m$ for each material, $m$, present in the mixture, where $\calP_m$ is the interval of admissible pressure values for material $m$ as described in Section~\ref{sec:asymp_thermo}.
In particular, we require,
\begin{equation} \label{eq:pressure_domain}
    P \in \calP(\bY) \eqq \bigcap_{m\in I(\bY)} \calP_m.
\end{equation}
for any mixture with representative mass fractions, $\bY$.

\begin{remark}[Solid-gas mixture]
    For solids, it is possible for the pressure to be negative; that is, the material enters into \textit{tension}.
    However, gases, by their very nature, cannot undergo tension.
    So when an equation of state for a solid is mixed with an equation of state for a gas, the PTE model necessarily requires that $P > 0$.
\end{remark}

In order to complete the thermodynamic picture, we must make an additional assumption regarding mixtures.
This assumption is the definition of the mixture Gibbs free energy; see \citet[Sec.~3.2]{grove2010pressure}.
\begin{assumption}[mixture Gibbs free energy] \label{ass:mix_gibbs}
    We assume that the mixture is a linear mixture rule without mixing contributions; that is, components of the mixture do not interact. 
    In particular, we assume that the mixture Gibbs free energy is ``mass averaged''; that is,
    \begin{equation} \label{eq:mixture_gibbs}
        G(P, T, \bY) \eqq \sum_{m=1}^M Y_m G_m(P,T),
    \end{equation}
    where $G_m(P,T) \eqq e_m(P,T) + P\tau_m(P,T) - T s_m(P,T)$.
\end{assumption}

Under Assumption~\ref{ass:mix_gibbs} we can derive the remaining thermodynamics of this mixture model. 
We have the known relations $\tau_m \eqq \big(\pdv{G_m}{P}\big)_T$ and $s_m \eqq -\big(\pdv{G_m}{T}\big)_P$.
Therefore, the mixture specific entropy is $s \eqq -\big( \pdv{G}{T}\big)_P = \sum_{m=1}^M Y_m s_m$ and the mixture specific volume is $\tau \eqq \big( \pdv{G}{P} \big)_T = \sum_{m=1}^M Y_m \tau_m$.
Lastly, the mixture specific internal energy is provided by $e = G - P\tau + Ts = \sum_{m=1}^M Y_m e_m$.
\begin{remark}[Entropy of mixing]
    For real gas mixtures, the entropy increases because of the mixing.
    This process is referred to as the ``entropy of mixing''.
    In particular, there exists some function $s_{\text{mix}}(\bY)$ such that $s \eqq \sum_{m=1}^M Y_m s_m(P,T) + s_{\text{mix}}(\bY)$.
    We have no such additional entropy in this model and therefore it can be thought of as a mixture of immiscible materials; e.g. sand and gravel or steel and air.
\end{remark}
We now provide some mixture identities for the thermodynamic derivatives.
\begin{lemma}[Thermodynamic mixture derivatives]
    Under the PTE assumption \eqref{eq:pte_assumption}, the following mixture derivative identities hold: $c_p = \sum_{m=1}^M Y_m c_{p,m}$,  $\beta = \sum_{m=1}^M \alpha_m \beta_m$ and $K_T = \sum_{m=1}^M \alpha_m K_{T,m}$.
\end{lemma}

\begin{proof}
    For the specific heat capacity at constant pressure, the result follows immediately since, $c_p = T \big( \pdv{s}{T} \big)_{P, \bY} = \sum_{m=1}^M Y_m T \big(\pdv{s_m}{T}\big)_{P} = \sum_{m=1}^M Y_m c_{p,m}$.
    Next consider, $K_T = -\rho \big(\pdv{\tau}{P}\big)_{T, \bY} = -\sum_{m=1}^M \rho Y_m \big(\pdv{\tau_m}{P}\big)_T = \sum_{m=1}^M -\frac{\alpha_m}{\tau_m} \big(\pdv{\tau_m}{P}\big)_T = \sum_{m=1}^M \alpha_m K_{T,m}$.
    The same reasoning can be applied for $\beta = \sum_{m=1}^M \alpha_m \beta_m$.
\end{proof}
The relations for the mixture isentropic bulk modulus, $B_s$, and specific heat capacity at constant volume, $c_v$, are not easily expressed in terms of a simple additive mixture.
Instead, one applies the following thermodynamic identity to recover them:
\begin{equation} \label{eq:important_thermo_iden}
    \frac{K_s}{K_T} = 1 - \frac{\beta^2 \tau T}{c_p K_T} = \frac{c_v}{c_p},
\end{equation}
(see \cite[Eq.~A14]{menikoff1989riemann}).
In this way, we see that $K_T$, $c_p$, and $\beta$ must first be computed before $c_v$ and $K_s$ (or $B_s$) can be computed.
We also provide the mixture Gr{\"u}neisen coefficient (\cite[Eq.~A16]{menikoff1989riemann}) for reference:
\begin{equation}
    \Gamma = \frac{\beta \tau}{c_v K_T}.
\end{equation}
\begin{remark}[Negative Gr{\"u}neisen coefficient] \label{rem:neg_gru_coeff}
    Unlike the positivity we demand in Definition~\ref{def:thermo_stability}, $\Gamma$ (and equivalently $\beta$) need not be positive.
    This often leads to strange behavior in the material; the most common example is water near 4${}^{\circ}$C, which contracts upon heating (see \citet{bethe1998theory}).
    There are other more exotic materials that can have a negative thermal expansion as well, see \citep{gava2012first, zwanziger2007phonon}.
\end{remark}
\begin{remark}[Maximum entropy principle]
    Fundamentally, the pressure temperature equilibrium can be deduced from the notion of the \textit{maximum entropy principle} when constrained under an energy and volume constraint. 
    The maximum entropy principle states that entropy must be maximized in an equilibrated system under the imposed constraints.
    In particular, this involves solving the constrained optimization problem:
    \begin{equation}
        s = \sup_{\substack{\sum Y_m\tau_m = \tau \\ \sum Y_m e_m = e }} \sum_{m=1}^M Y_m s_m(\tau_m, e_m).
    \end{equation}
    Using Lagrange multipliers, one finds that pressure and temperature must be in equilibrium.
    A similar result can be shown for the \textit{minimum energy principle}, and a similar idea is used in the proof of Proposition~\ref{prop:concave_energy_constraint}.
    For more information regarding these principles see \citet[Chapters~1 \& 2]{callen1998thermodynamics} as well as \citet[Chapter IV]{fermi2012thermodynamics}.
\end{remark}
\begin{remark}[Confusion with the HEM]
    The Euler system \eqref{eq:mass_cons}--\eqref{eq:total_energy_cons} with the assumption of pressure-temperature equilibrium and \eqref{eq:mixture_gibbs} should not be confused with the \textit{homogeneous equilibrium model} (HEM) often used in multiphase flows.
    Many papers often do not report on the Gibbs free energy; whether material Gibbs free energies are in equilibrium or if a mixture Gibbs free energy is used.
    Without a clear assumption or statement, the complete thermodynamic picture becomes unclear.
\end{remark}

With the thermodynamic mixture quantities in place, we establish the system of equations to solve for the equilibrated pressure and temperature.
Since we are interested in solving the four-equation model, one can readily compute $\bY$, $\tau$, and $e$ from the conservative state, $\bu$, through an algebraic calculation.
We then propose to solve the following system for $P$ and $T$:
\begin{equation} \label{eq:mass_frac_pte_system}
    \sum_{m=1}^M Y_m \tau_m(P,T) = \tau \quad \text{ and } \quad \sum_{m=1}^M Y_m e_m(P,T) = e.
\end{equation}
In contrast, solving \eqref{eq:pte_assumption} requires $2(M-1)$ equations, whereas \eqref{eq:mass_frac_pte_system} only requires solving two equations.
We close this section with several remarks and lemmas regarding the mixture thermodynamics.

\begin{lemma}[Generalized pressure] \label{lem:generalized_pressure}
    Assume each equation of state satisfies Definition~\ref{def:thermo_stability}, and that the following asymptotic properties hold for the mixture: $\lim_{P\to \inf(\calP)^+} \sum_{m=1}^M Y_m \tau_m(P,T) = \infty$ and $\lim_{P\to\infty} \sum_{m=1}^M Y_m \tau_m(P,T) = 0$ for all $T > 0$.
    Then the generalized pressure, $P$, is defined to be the root the equation $\sum_{m=1}^M Y_m\tau_m(P,T) = \tau$ and is well-defined for all $T > 0$ and $\tau > 0$.
\end{lemma}

\begin{proof}
    Differentiating $\sum_{m=1}^M Y_m \tau_m(P, T)$ with respect to $P$; we find,
    \begin{equation}
        \sum_{m=1}^M Y_m \partial_P \tau_m(P,T) = \sum_{m=1}^M -Y_m \tau_m(P,T) K_{T, m}(P,T) = -\tau K_T < 0,
    \end{equation}
    for all $T > 0$ and $\tau > 0$.
    Therefore, the specific volume equation can be inverted to find the generalized pressure, $P = P(\tau, T, \bY)$, since $\tau$ is in the range of $\sum_{m=1}^M Y_m \tau_m(P,T)$ from the asymptotic assumptions.
\end{proof}

Note that finding the generalized temperature is not as straightforward as in Lemma~\ref{lem:generalized_pressure}, but is established in the proof of Theorem~\ref{thm:existence}.

\begin{lemma}[The Gibbs' identity] \label{lem:gibbs_identity}
Assuming the Gibbs' identity holds for each individual material, then the mixture Gibbs' identity is given by
\begin{equation} \label{eq:mix_gibbs_identity}
    T \diff s = \diff e + P \diff \tau - \sum_{m=1}^M G_m \diff Y_m.
\end{equation}
\end{lemma}
\begin{proof}
    We start by taking the mass fraction average of each single material Gibbs' identity,
    \begin{equation*}
        \sum_{m=1}^M T Y_m \diff s_m = \sum_{m=1}^M \big(Y_m \diff e_m + P \diff \tau_m \big).
    \end{equation*}
    Then using the identity, $\sum_{m=1}^M X_m \diff y_m = \diff \big(\sum_{m=1}^M X_m y_m \big) - \sum_{m=1}^M y_m \diff X_m$ for some arbitrary quantities $\{X_m\}$ and $\{y_m\}$, we arrive at the result $T\diff s = \diff e + P \diff \tau - \sum_{m=1}^M (e_m + P\tau_m - T s_m) \diff Y_m$.
    This completes the proof.
\end{proof}

\begin{remark}[Reduction to a single EOS]
    When the mass fractions are held constant, the mixture EOS can be treated as a single material EOS.
    This can be seen with the mixture Gibbs' identity \eqref{eq:mix_gibbs_identity} in Lemma~\ref{lem:gibbs_identity} which reduces to the single material case with $\diff Y_m = 0$ for each $m \in \intset{1}{M}$.
    Hence the mass fractions can be thought of as simply equation of state parameters.
    This allows us to apply many of the thermodynamic identities provided in \citet[Appendix~A]{menikoff1989riemann}.
\end{remark}

\subsection{The admissible set} \label{sec:admissible_set}
When solving \eqref{eq:mass_cons}--\eqref{eq:total_energy_cons}, the initial state as well the solution (for all $t > 0$), must belong to a so-called \textit{admissible set}.
For the Euler equations, the global admissible set is the collection of states for which $\rho > 0$ and $T > 0$.
As one often works with the specific internal energy in the Euler equations, the form of the admissible set may vary when expressed in terms of $\sfe(\bu)$ and is dependent on the equation of state.
For example, the admissible set for (single material) ideal gas law is $\calA_{\text{ideal}} \eqq \{\bu \in \Real^{d+2} : \rho > 0, \, \sfe(\bu) > 0\}$ since $e = c_v T$.
Whereas for the Nobel-Abel stiffened gas law (NASG) (\citet{le2016noble}), $P(\tau, e) = (\gamma-1) \frac{e-q}{\tau - b} - \gamma p_\infty$, the admissible set is, $\calA_{\text{NASG}} \eqq \{ \bu \in \Real^{d+2} : 0 < \rho^{-1} < b, \, \sfe(\bu) - q > p_\infty(\rho^{-1} - b)\}$ since $e - q - p_\infty (\tau - b) = c_v T$.
Note the NASG EOS has the added constraint of a maximal density, $b^{-1}$, for $b > 0$.
However, we omit any EOS with a maximum density for simplicity.
For an ideal gas \textit{mixture}, much of the complexity with PTE is simplified (see \citep[Sec.~2]{CLAYTON2026107159}).
In which case the admissible set is given by $\calA_{\text{ideal mix}} \eqq \{ \bu \in \Real^{M+d+1} : \rho(\bu) > 0, \alpha_k\rho_k \geq 0, \forall k \in \intset{1}{M}, \sfe(\bu) > 0\}$.

Each of these admissible sets can be characterized under one admissible set by introducing the zero temperature isotherm,
\begin{equation} \label{eq:zero_temp_isotherm}
    e_{\text{zero}}(\tau) \eqq \lim_{T \to 0^+} e(P(\tau, T), T).
\end{equation}
As we shall see, $e_{\text{zero}}(\tau)$ is the minimal curve in the $\tau$-$e$ space which bounds (from below) all admissible specific internal energy states; see Lemma~\ref{lem:zero_temp_admissible_set}.

\begin{remark}[3rd law of thermodynamics] \label{rem:3rd_law}
    The third law of thermodynamics as described by \citet[Chapter VI, \S 282]{planck2013treatise}, says that the entropy must approach a finite value as the temperature approaches zero and is independent of the pressure. 
    This finite value can, without loss of generality, be taken as zero.
    Applying this law to the specific internal energy, we arrive at the definition of the cold curve:
    \begin{equation}
        e_{\textup{cold}}(\tau) \eqq \lim_{s \to 0^+} \frake(\tau,  \theta(\tau, s)) = \lim_{T \to 0^+} \frake(\tau, T),
    \end{equation}
    where $\frake(\tau, T) \eqq e(P(\tau, T), T)$.
    The 3rd law of thermodynamics is sometimes referred to as \textit{Nernst's theorem}; however, Nernst's theorem is described by the vanishing of differences in entropy near absolute zero.
    
    Note, not every equation of state satisfies the 3rd law of thermodynamics; most notably, the ideal and stiffened gas EOS. 
    The failure occurs in the sense that, the zero temperature isotherm requires $s \to -\infty$, see \citet[Remark~2.7 \& 2.9]{clayton2026preserving}.
    In order to remain as general as possible, we work with the zero temperature curve rather than the cold curve.
    However, some cases require special considerations; e.g. Proposition~\ref{prop:concave_energy_constraint}.
\end{remark}

\begin{lemma}[Positive temperature admissible set] \label{lem:zero_temp_admissible_set}
    Define
    \begin{subequations}
    \begin{align}
        \calA_{T} &\eqq \{ \bu \in \Real^{2+d} : \rho > 0, \, \sfT(\bu) > 0 \}, \\
        \calA_{\textup{zero}} &\eqq \{ \bu \in \Real^{2+d} : \rho > 0, \, \sfe(\bu) > e_{\textup{zero}}(\rho^{-1}) \}.
    \end{align}
    \end{subequations}
    If $\sfc_v(\bu) \eqq \big(\pdv{\frake}{T}\big)_\tau = c_v(\rho^{-1}, \sfT(\bu)) > 0$ for all $\bu \in \calA_T \cup \calA_{\textup{zero}}$ where $\sfT(\bu)$ is the temperature as a function of the state variable $\bu$.
    Then $\calA_T = \calA_{\textup{zero}}$.
\end{lemma}

\begin{proof}
    For $\bu_0 \in \Real^{2+d}$, define $\tau_0 \eqq \rho^{-1}(\bu_0)$, $T_0 \eqq \sfT(\bu_0)$, and $e_0 \eqq \sfe(\bu_0)$.
    Now, assume that $\bu_0 \in \calA_{\text{zero}}$.
    From the hypothesis we have that $\partial_T \frake(\tau, T) > 0$, this implies that for any fixed $\tau > 0$, $e$ is an increasing function of $T$.
    Hence $e_{\text{zero}}(\tau_0) = \lim_{T \to 0^+} e(\tau_0, T) = \inf_{T > 0} e(\tau_0, T)$, and therefore, $T_0 > 0$ implies $e_0 > e_{\text{zero}}(\tau_0)$ which implies that $\calA_{\text{zero}} \subset \calA_T$.
    For the reverse direction, let $\bu_0 \in \calA_T$.
    Again, since $\sfc_v(\bu_0) > 0$, 
    we know that $\frake(\tau, T)$ is an increasing function of $T$ and therefore, $e_0 = \frake(\tau_0, T_0) > \lim_{T \to 0^+} \frake(\tau_0, T) = e_{\text{zero}}(\tau_0)$.
    Therefore $e_0 > e_{\text{zero}}(\tau_0)$ and thus $\calA_T = \calA_{\text{zero}}$.
\end{proof}

From Lemma~\ref{lem:zero_temp_admissible_set}, it becomes apparent that that $\calA_T$ is a convex admissible set (assuming $e_{\text{zero}}(\tau)$ is a convex function) owing to \citep[Proposition~2.8]{clayton2026preserving}.
Such a convex admissible set is often referred to as an invariant domain and physically correct approximations of the solution are expected to reside there, (see \citep{chueh1977positively, frid2001maps, guermond2016invariant}).

For the PTE closure, the admissible set is not easily realized and requires a bit more formalization.
We first define the \textit{generalized pressure zero curve}.
\begin{definition}[Generalize pressure zero curve] \label{def:P_zero_curve}
    Define,
    \begin{equation}
        \tau_{\textup{zero}}(P, \bY) \eqq \sum_{m \in I(\bY)} Y_m \tau_{m,\textup{zero}}(P), \quad \text{ for } P \in \calP(\bY).
    \end{equation}
    where $\tau_{m,\textup{zero}}(P) \eqq \lim_{T \to 0^+} \tau_m(P, T)$.
    If $\tau_{\textup{zero}}(P, \bY)$ is a one-to-one function of $P \in \calP(\bY)$, then $P_{\textup{zero}}(\tau, \bY)$ is defined to be the inverse of $\tau_{\textup{zero}}(P, \bY)$ and is referred to as the generalized pressure zero curve.
\end{definition}

\begin{remark}[Multi-valued pressure cold curves] \label{rem:multi_valued_cold}
    As there is a large variety of equations of state, we would like to emphasize the assumption of one-to-one in Definition~\ref{def:P_zero_curve}.
    Consider an ideal gas mixture of two materials.
    The pressure laws are given by, $P\tau_1 = R_1 T$ and $P\tau_2 = R_2 T$.
    Taking the limit as $T \to 0^+$, we see that $P\tau_1 = 0$ and $P\tau_2 = 0$.
    Hence $\tau_{1,\text{zero}}(P) = \tau_{2,\text{zero}}(P) = 0$ for all $P > 0$, and therefore, the inverse of $\tau_{\text{zero}}(P)$, is the set $(0,\infty)$ which is certainly not a function.
\end{remark}

\begin{definition}[Mixture zero curve] \label{def:mix_cold_curve}
    If $P_{\textup{zero}}(\tau, \bY)$ is well-defined as described in Definition~\ref{def:P_zero_curve}, then the mixture zero curve is defined to be,
    \begin{equation} \label{eq:cold_energy_curve}
    \begin{split}
        e_{\textup{zero}}(\tau, \bY) &\eqq \lim_{T\to 0^+} \sum_{m \in I(\bY)} Y_m e_m(P(\tau, T, \bY), T) \\
        &= \sum_{m \in I(\bY)} Y_m e_m(P_{\textup{zero}}(\tau, \bY), 0).
    \end{split}
    \end{equation}
\end{definition}

We would like to emphasize that $e_m(P_{\text{zero}}(\tau, \bY), 0) = e_{m,\text{zero}}(\tau_{m,\text{zero}}(P_{\text{zero}})) \neq e_{m,\text{zero}}(\tau)$.
That is, zero temperature isotherm is evaluated at the \textit{material} specific volume, $\tau_m$, and not the bulk quantity, $\tau$.
Each equation of state has it's own zero temperature curve, and so (depending on $\bY$), the generalized pressure zero curve, $P_{\text{zero}}(\tau, \bY)$, is also generally never the same as a specific material's zero temperature pressure curve, $P_{m,\text{zero}}(\tau_m) \eqq \lim_{T\to 0^+} P_m(\tau_m, T)$.
We are now ready to introduce the admissible set for the four-equation model under pressure-temperature equilibrium and ideal mixing (Assumption~\ref{ass:mix_gibbs}).
\begin{definition}[PTE admissible set] \label{def:pte_admissible_set}
Assume $P_{\textup{zero}}(\tau, \bY)$ is well-defined by Definition~\ref{def:P_zero_curve} for all $\bY \in \Delta_M$.
Then the admissible set is defined by,
\begin{equation} \label{eq:pte_admissible_set}
\begin{split}
    \calA_{\textup{PTE}} \eqq \Big\{\bu \in \Real^{M+d+1} : \rho(\bu) > &0, \, \alpha_k\rho_k \geq 0, \, \forall k \in \intset{1}{M}, \, \\
    &\qquad\qquad \sfe(\bu) > e_{\textup{zero}}(\tau, \bY) \Big\},
\end{split}
\end{equation}
where $e_{\textup{zero}}(\tau, \bY)$ is defined in Definition~\ref{def:mix_cold_curve}.
\end{definition}
As one may notice, the energy constraint in \eqref{eq:pte_admissible_set} is not easy to conceptualize but it is vital to guarantee a unique solution to the PTE system \eqref{eq:mass_frac_pte_system} (see Theorem~\ref{thm:existence}).
We also see from Proposition~\ref{prop:concave_energy_constraint} that $\calA_{\text{PTE}}$ is a convex set when each EOS satisfies the 3rd law of thermodynamics.
Under further thermodynamic assumptions, by Corollary~\ref{cor:concave_energy}, we have that $\calA_{\text{PTE}}$ is still convex even if all the EOS do not satisfy the 3rd law of thermodynamics.
\begin{remark}[Admissible set with weakened assumptions]
    One my wonder what the PTE admissible set is if the one-to-one assumption fails in Definition~\ref{def:P_zero_curve}.
    In the case of an ideal gas, we note that $e = c_v T$, hence $e$ is independent of $P$.
    From \eqref{eq:cold_energy_curve}, we would simply have that $e_m(P_{\text{zero}}, 0) = 0$ for $m$ being the index of the ideal gas law.
    In general, one would need to analyze the specific formulation of $e_m(P,T)$ to make a conclusive statement.
\end{remark}

\begin{theorem}[Existence of unique PTE solution] \label{thm:existence}
    Let $\bu_0 \in \calA_{\textup{PTE}}$ with $\tau_0 \eqq \rho_0^{-1}$ and $\bY_0 \eqq (\balpha\brho)_0 / \rho_0$.
    Assume each equation of state is thermodynamically stable by Definition~\ref{def:thermo_stability} and that all $M$ materials are present; that is, $Y_{m,0} > 0$ for all $m \in \intset{1}{M}$.
    Furthermore, let $P(\tau_0, T, \bY_0)$ be the generalized pressure from Lemma~\ref{lem:generalized_pressure} and assume that the following asymptotic properties hold: 
    \begin{enumerate}
        \item There exists $m \in \intset{1}{M}$ such that $\lim_{P \to \inf(\calP)^+} \tau_m(P,T) = \infty$ for all $T > 0$.
        \item $\lim_{P \to \infty} \sum_{m=1}^M Y_m \tau_m(P,T) = 0$ for all $T > 0$.
        \item $\lim_{T\to\infty} \sum_{m=1}^M Y_m e_m(P(\tau_0, T, \bY_0), T) = \infty$.
    \end{enumerate}
    Then a unique solution exists to the PTE system \eqref{eq:mass_frac_pte_system}.
\end{theorem}

\begin{proof}
    From thermodynamic stability and assumptions 1 and 2, we know that the generalized pressure function $P(\tau_0, T, \bY_0)$ exists and is well-defined for all $T > 0$ (Lemma~\ref{lem:generalized_pressure}).
    So one only needs to show that $e_0 \eqq \sfe(\bu_0)$ is in the range of the following function:
    \begin{equation}
        g_0(T) \eqq \sum_{m=1}^M Y_{m,0} e_m(P(\tau_0, T, \bY_0), T).
    \end{equation}
    First, we note that $g_0(T)$ is a continuous function of $T$ since each $e_m$ is continuous and $P(\tau_0, T, \bY_0)$ is the inverse of a continuous function of $T$; namely, $\sum_{m=1}^M Y_{m,0} \tau_m(P, T)$.
    Then from the hypothesis, $\bu_0 \in \calA_{\text{PTE}}$, this implies $e_0 > \lim_{T\to 0^+} g_0(T)$ (see Definition~\ref{def:mix_cold_curve}).
    Lastly, by the asymptotic assumption 3, $e_0 \in \text{Range}(g_0)$. 
    Hence a solution exists.
    We now show that the solution is unique.
    
    We proceed by showing that $g_0(T)$ is a strictly increasing function of $T$.
    Computing the derivative with respect to $T$, we have,
    \begin{align*}
        g_0'(T) &= 
        \sum_{m=1}^M Y_m \frac{\partial}{\partial T} \big( e_m(P(\tau_0,T,\bY_0),T) \big) \\
        &= \sum_{m=1}^M Y_m \Big[ \Big(\pdv{e_m}{P}\Big)_T \Big(\pdv{P}{T} \Big)_{\tau, \bY} + \Big( \pdv{e_m}{T} \Big)_P \Big]
    \end{align*}
    Note that $\big(\pdv{e_m}{P}\big)_T = P \tau_m K_{T,m} - \beta_m \tau_m T$ and $\big(\pdv{e_m}{T}\big)_P = (c_{p,m} - P\tau_m \beta_m)$.
    To determine $\big(\pdv{P}{T} \big)_{\tau, \bY}$, we differentiate $\sum_{m=1}^M Y_m \tau_m(P(\tau_0,T,\bY_0), T) = \tau_0$ with respect to $T$, to find, $\sum_{m=1}^M Y_m \big[\big(\pdv{\tau_m}{P}\big)_T \big(\pdv{P}{T} \big)_{\tau, \bY} + \big(\pdv{\tau_m}{T}\big)_P\big] = 0$.
    Solving, we find that $\big(\pdv{P}{T} \big)_{\tau, \bY} = -\sum_{m=1}^M Y_m \big(\pdv{\tau_m}{T}\big)_P / \sum_{m=1}^M Y_m \big( \pdv{\tau_m}{P} \big)_T = \beta / K_T$.
    Therefore,
    \begin{align*}
        g_0'(T)
        &= \sum_{m=1}^M Y_m \Big[ (P\tau_m K_{T,m} - \beta_m \tau_m T) \frac{\beta}{K_T} + (c_{p,m} - P\tau_m\beta_m) \Big] \\
        &= \tau (P K_T - \beta T) \frac{\beta}{K_T} + (c_p - P\tau \beta) \\
        &= c_p - \frac{\beta^2\tau T}{K_T} = c_v.
    \end{align*}
    As proven in \citet[Sec.~3.2]{grove2010pressure}, $c_v > 0$ since each EOS is thermodynamically consistent by Definition~\ref{def:thermo_stability}.
    Thus, $\sum_{m=1}^M Y_m e_m(P(\tau_0, T, \bY_0),T)$ is strictly increasing, with respect to $T$.
    Hence a unique $T_0 > 0$ exists such that $\sum_{m=1}^M Y_m e_m(P_0, T_0) = e_0$ where $P_0 \eqq P(\tau_0, T_0, \bY_0)$.
    Thus $(P_0, T_0)$ is the unique solution to the PTE system \eqref{eq:mass_frac_pte_system}.
\end{proof}

The importance of Theorem~\ref{thm:existence} is not only that a PTE solution can be found (assuming the state belongs to the admissible set), but that that the generalized pressure and temperature are uniquely defined by $\tau$, $e$, and $\bY$.
That is, $P = P(\tau, e, \bY)$ and $T = T(\tau, e, \bY)$ are well-defined functions of $\tau$, $e$, and $\bY$ (with a slight abuse of notation), thus completing the thermodynamic picture.

We now proceed by establishing the convexity of the PTE admissible set, $\calA_{\text{PTE}}$; however, we must make one more technical assumption regarding the convergence of $\frake(\tau, \theta(\tau, s))$ to the cold curve (in the sense of $s \to 0^+$).

\begin{assumption}[Convergence to the cold curve] \label{ass:cold_curve_convergence}
    We assume that equations of state which satisfy the 3rd law of thermodynamics converge in the following sense:
    \begin{equation}
        \frake(\tau, \theta(\tau,s)) \to e_{\text{cold}}(\tau) \quad \text{ in } \, C^1_{\text{loc}}(\Real_+) \, \text{ as } s \to 0^+.
    \end{equation}
    That is, the function and its partial derivative with respect to $\tau$ converge uniformly on any compact subset of $\Real_+$.
\end{assumption}
This assumption is required to identify the following thermodynamic derivative identity in the limit of zero entropy.
\begin{lemma}[Cold curve consistency] \label{lem:cold_curve_consistency}
    Under Assumption~\ref{ass:cold_curve_convergence}, the following identity holds:
    \begin{equation}
        e'_{\textup{cold}}(\tau) = - P_{\textup{cold}}(\tau)
    \end{equation}
    where $P_{\textup{cold}}(\tau) = \lim_{s\to 0^+} p(\tau, s)$.
\end{lemma}

\begin{proof}
    From the $C^1_{\text{loc}}$ convergence, we can pass the derivative through the limit; in particular,
    \begin{equation}
        e_{\text{cold}}'(\tau) = \lim_{s \to 0^+} \Big( \pdv{e}{\tau} \Big)_s = -\lim_{s \to 0^+} p(\tau, s) = P_{\text{cold}}(\tau).
    \end{equation}
    This completes the proof.
\end{proof}

\begin{proposition}[Concavity of admissibility constraint] \label{prop:concave_energy_constraint}
    Assume each equation of state is thermodynamically stable (Definition~\ref{def:thermo_stability}) and satisfies the 3rd law of thermodynamics (Remark~\ref{rem:3rd_law}), then the following functional is concave:
    \begin{equation}
        \Psi_{\textup{cold}}(\bu) = \rho \sfe(\bu) - \sum_{m \in I(\bY)} \alpha_m\rho_m e_m(P_{\textup{cold}}(\rho^{-1}, \bY(\bu)), 0).
    \end{equation}
\end{proposition}

\begin{proof}
    To start, let $\bw \eqq (\rho , \bbm, E)$ and define $U(\rho, \bsfm, E) \eqq E - \frac{\Vert \bsfm \Vert^2_{\ell^2}}{2\rho}$. 
    Note the classical result that $U(\bw)$ is a concave functional.
    Note the mapping $\bu \mapsto \bw(\bu) = (\sum_{m=1}^M \alpha_m \rho_m, \bbm, E)$ is a linear map and $\rho \sfe(\bu) = U(\bw(\bu))$.
    It follows that $\rho \sfe$ is a concave function of $\bu$.
    More explicitly, for $\lambda \in [0,1]$, and $\bu_1, \bu_2 \in \calA_{\text{PTE}}$,
    \begin{equation*}
    \begin{split}
        \rho\sfe(\lambda \bu_1 + (1 - \lambda) \bu_2) &= 
        U(\bw(\lambda\bu_1 + (1 - \lambda) \bu_2))\\
        &=U(\lambda \bw(\bu_1) + (1 - \lambda) \bw(\bu_2)) \\
        &\geq \lambda U(\bw(\bu_1)) + (1 - \lambda) U(\bw(\bu_2)) \\
        &= \lambda \rho \sfe(\bu_1) + (1 - \lambda) \rho \sfe(\bu_2).
    \end{split}
    \end{equation*}

    Now, we must show that the cold curve portion of the function,
    \begin{equation} \label{proof:cold_mix}
        \sum_{m \in I(\bY)} \alpha_m \rho_m e_{m,\text{cold}}(\tau_{m,\text{cold}}(P_{\text{cold}}(\rho^{-1}, \bY))),
    \end{equation}
    is a convex function of $\{\alpha_m \rho_m\}_{m=1}^M$.
    To simplify notation we set $x_m \eqq \alpha_m \rho_m$ and $P_c \eqq P_{\text{cold}}(\rho^{-1}, \bY)$.
    (Note the identity, $e_m(P_c, 0) = e_{m,\text{cold}}(\tau_{m,\text{cold}}(P_c))$.)
    Since $P_c$ is an implicitly defined function, directly proving convexity by computing the Hessian is not easily achieved.
    We instead show that \eqref{proof:cold_mix} can be written in a simpler form for which convexity can be analyzed.
    
    We define
    \begin{equation} \label{proof:energy_inf_functional}
        \calE(x_1, \ldots, x_M) \eqq \inf_{\substack{V_m \geq 0 \\ \sum_{m} V_m = 1}} \sum_{m=1}^M \Phi_m(x_m, V_m),
    \end{equation}
    where $V_m \eqq x_m \tau_m$ and
    the $\Phi_m$ are the \textit{extended-valued perspective} for each material, defined by,
    \begin{equation}
        \Phi_m(x, V) = \begin{cases}
            x e_{m,\text{cold}}(V / x), \quad &x > 0, \, V > 0, \\
            0, &x = 0, \, V = 0, \\
            +\infty, &\text{otherwise.}
        \end{cases}
    \end{equation}
    Note, $\Phi_m$ is used to handle cases when a material vanishes.
    In addition, the $+\infty$ value is only defined for cases that should not occur since $x = 0$ if and only if $V = 0$.
    So we have that,
    \begin{equation}
        \inf_{\substack{V_m \geq 0 \\ \sum_{m} V_m = 1}} \sum_{m=1}^M \Phi_m(x_m, V_m) = \inf_{\substack{V_m \geq 0 \\ \sum_{m} V_m = 1}} \sum_{m \in I(\bY)} x_m e_{m,\text{cold}}\Big(\frac{V_m}{x_m}\Big).
    \end{equation}
    Note, we are not invoking pressure temperature equilibrium, but simply the material equation of state definition of the cold curve.
    We claim that the following identity holds:
    \begin{equation} \label{proof:infimum_identity}
        \calE(x_1, \ldots, x_M) = \sum_{m \in I(\bY)} x_m e_{m,\text{cold}}(\tau_{m,\text{cold}}(P_c)).
    \end{equation}
 
    To prove this, define $\overline{V}_m \eqq x_m \overline{\tau}_m$ where $\overline{\tau}_m \eqq \tau_{m,\text{cold}}(P_c)$ for each $m \in I(\bY)$.
    Then we have the following,
    \begin{equation}
        \sum_{m \in I(\bY)} \overline{V}_m = \sum_{m \in I(\bY)} x_m \overline{\tau}_m = \rho \sum_{m \in I(\bY)} Y_m \overline{\tau}_m = \rho \tau = 1.
    \end{equation}
    Note we have used $\sum_{m \in I(\bY)} Y_m \overline{\tau}_m = \tau$ by Definition~\ref{def:P_zero_curve} since $P_c$ is the cold pressure associated with $\tau = 1/\rho$.

    We now proceed to show that $\{\overline{V}_m\}$ is the minimizer of \eqref{proof:energy_inf_functional}.
    Let $\{V_m\}$ be an arbitrary candidate for the minimizer.
    Since $e_{m,\text{cold}}$ is convex and $e_{m,\text{cold}}'(\tau_{m,\text{cold}}(P_c)) = - P_c$ (by Lemma~\ref{lem:cold_curve_consistency}), a Taylor series expansion of $e_{m,\text{cold}}(\tau_m)$ at $\tau_m = \overline{\tau}_m$ provides us the following inequality,
    \begin{equation*}
        e_{m,\text{cold}}(\tau_m) = e_{m,\text{cold}}\Big( \frac{V_m}{x_m} \Big) \geq e_{m,\text{cold}}(\overline{\tau}_m) - \overline{P} \Big(\frac{V_m}{x_m} - \overline{\tau}_m \Big).
    \end{equation*}
    Multiplying by $x_m$ and summing over all $m \in I(\bY)$, we have,
    \begin{equation*}
        \sum_{m \in I(\bY)} x_m e_{m,\text{cold}}\Big( \frac{V_m}{x_m} \Big) \geq \sum_{m \in I(\bY)} x_m e_{m,\text{cold}}(\overline{\tau}_m) - \overline{P} \sum_{m \in I(\bY)} (V_m - \overline{V}_m).
    \end{equation*}
    Since $\sum_{m\in I(\bY)} V_m = \sum_{m \in I(\bY)} \overline{V}_m = 1$, we see that 
    \begin{equation*}
        \sum_{m \in I(\bY)} x_m e_{m,\text{cold}}\Big( \frac{V_m}{x_m} \Big) \geq \sum_{m \in I(\bY)} x_m e_{m,\text{cold}}\Big( \frac{\overline{V}_m}{x_m} \Big)
    \end{equation*}
    for all potential minimizers, $\{V_m\}$.
    Thus, $\{\overline{V}_m\}$ is the minimizer with the associated cold curve pressure, $P_{\text{cold}}(\tau, \bY)$.
    Then, noting the following identities, $e_{m,\text{cold}}(\overline{\tau}_m) = e_{m,\text{cold}}(\overline{V}_m / x_m) = e_{m,\text{cold}}(\tau_{m,\text{cold}}(P_{\text{cold}}(\tau, \bY)))$, we see that the identity \eqref{proof:infimum_identity} holds.

    We now only need to show that $\calE(x_1, \ldots, x_M)$ is a convex function.
    Note that each $e_{m,\text{cold}}(\tau_m)$ is a convex function of $\tau_m$ (from thermodynamic stability), therefore, the \textit{perspective} of each $e_{m,\text{cold}}$ is convex; that is, $\sfp_m(x_m, V_m) \eqq x_m e_{m,\text{cold}}(V_m / x_m)$ is convex on $\Real_+ \times \Real_+$.
    It then follows that each $\Phi_m(x_m, V_m)$ is convex on $[0, \infty)^2$.
    Then, since the sum of convex functions is a convex function and taking the infimum with respect to the $\{V_m\}_{m=1}^M$ over the convex set $\{\bV \in (\Real_+)^M : \sum_{m=1}^M V_m = 1 \}$ produces a convex function, we have that $\calE$ is convex.
    Therefore, \eqref{proof:cold_mix} is a convex function of $\alpha_1\rho_1, \ldots, \alpha_M\rho_M$ which completes the proof.
\end{proof}

One may wonder if the assumption of the 3rd law of thermodynamics in Proposition~\ref{prop:concave_energy_constraint} can be dropped.
It is possible; however, the main caveat in the proof is that each material must satisfy $\partial_\tau (\lim_{T \to 0^+} \frake(\tau, T) ) = -P_{\text{zero}}$.
The next Lemma provides the conditions that we require.
\begin{lemma}[Zero temperature identity] \label{lem:cold_deriv_equal_p}
    For a (single) thermodynamically stable equation of state (Definition~\ref{def:thermo_stability}),
    assume that
    \begin{equation}
        \frake(\tau, T) \to e_{\textup{zero}}(\tau) \quad \text{ in } \, C^1_{\textup{loc}}(\Real_+) \, \text{ as } T \to 0^+
    \end{equation}
    and that $T \frac{\beta(P(\tau, T), T)}{K_T(P(\tau, T), T)} \to 0$ pointwise as $T \to 0^+$ (not necessarily uniformly).
    Then the following identity holds:
    \begin{equation}
        \lim_{T \to 0^+} \partial_\tau \frake(\tau, T) = e_{\textup{zero}}'(\tau) = -P_{\textup{zero}}(\tau).
    \end{equation}
\end{lemma}

\begin{proof}
    Computing the derivative with respect to $\tau$, we have,
    \begin{equation*}
        \partial_\tau \frake(\tau, T) = \partial_\tau e(P(\tau, T), T) = \Big( \pdv{e}{P} \Big)_T \Big( \pdv{P}{\tau} \Big)_T = T\frac{\beta(P(\tau, T), T)}{K_T(P(\tau, T), T)} - P(\tau, T).
    \end{equation*}
    where we have used the identity $\big( \pdv{e}{P} \big)_T = P \tau K_T - \beta \tau T$.
    From this we see that $\lim_{T\to 0^+} \partial_\tau \frake(\tau, T) = - P_{\text{zero}}(\tau)$.
    The result follows just as in Lemma~\ref{lem:cold_curve_consistency}.
    This completes the proof.
\end{proof}

\begin{example}[Stiffened-gas]
    To see an application of Lemma~\ref{lem:cold_deriv_equal_p}, consider the stiffened-gas law, $\frake(\tau, T) = c_v T + P_\infty \tau + q$.
    We see that $\partial_\tau \frake(\tau, T) = P_\infty$ which converges uniformly to $P_\infty$ as $T \to 0^+$ (trivially). 
    For the second assumption, we have, $\lim_{T\to 0^+} T \beta / K_T = \lim_{T\to 0^+} (P(\tau, T) + P_\infty) = \lim_{T\to 0^+} \frac{RT}{\tau} = 0$.
    Note, the convergence of this limit is \textit{not} uniform.
    Then we see that the result of Lemma~\ref{lem:cold_deriv_equal_p} holds by checking the following identities, $\lim_{T \to 0^+} \frake(\tau, T) = P_\infty \tau + q$, $\partial_\tau \frake(\tau, T) = P_\infty$, and $\lim_{T \to 0^+} P(\tau, T) = -P_\infty$.
\end{example}

\begin{corollary} \label{cor:concave_energy}
    If the conditions in Lemma~\ref{lem:cold_deriv_equal_p} are satisfied for each equation of state, all $e_{m,\textup{zero}}(\tau)$ are convex and $P_{\textup{zero}}(\tau, \bY)$ is well defined by Definition~\ref{def:P_zero_curve}, then 
    \begin{equation}
        \Psi_{\textup{zero}}(\bu) \eqq \rho \sfe(\bu) - \sum_{m \in I(\bY)} \alpha_m \rho_m e_{m,\textup{zero}}(P_{\textup{zero}}(\rho^{-1}, \bY))
    \end{equation}
    is concave.
\end{corollary}

\begin{proof}
    The proof is the same as the proof for Proposition~\ref{prop:concave_energy_constraint}, except that one invokes Lemma~\ref{lem:cold_deriv_equal_p} for the derivative identity used in the Taylor series expansion to show that the minimizer occurs at the equilibrated pressure, $P_{\text{zero}}$.
\end{proof}

\begin{theorem}[Convex PTE admissible set] \label{thm:pte_convex_set}
    The PTE admissible, $\calA_{\textup{PTE}}$ defined in Definition~\ref{def:pte_admissible_set} is convex if all EOS satisfy the 3rd law of thermodynamics or the assumptions in Lemma~\ref{lem:cold_deriv_equal_p} and Corollary~\ref{cor:concave_energy} hold for all EOS which do not satisfy the 3rd law of thermodynamics.
\end{theorem}

\begin{proof}
    Since $\calA_{\text{PTE}}$ is characterized by concave functionals owing to Proposition~\ref{prop:concave_energy_constraint} and Corollary~\ref{cor:concave_energy}, then $\calA_{\text{PTE}}$ is convex.
\end{proof}

\begin{remark}[Admissible state check] \label{rem:admissible_stat_check}
    During a numerical simulation, one would like to know whether the updated solution resides in $\calA_{\text{PTE}}$. 
    Let $\bu_0 \in \Real^{M+d+1}$ be a given state.
    We wish to determine if $\bu_0 \in \calA_{\text{PTE}}$.
    Using Definition~\ref{def:P_zero_curve}, we first solve
    \begin{equation*}
        \sum_{m \in I(\bY_0)} Y_{m,0} \tau_{m,\text{zero}}(P) = \tau_0,
    \end{equation*}
    for the root $P_0 \in \calP(\bY_0)$ using a numerical algorithm, e.g. Newton-Raphson, bisection, etc.
    Then we simply assess if $e_0 \eqq \sfe(\bu_0) > e_{\text{cold}}(\tau_0, \bY_0) = \sum_{m=1}^M Y_{m,0} e_m(P_0, 0)$.
\end{remark}


\begin{remark}[Example PTE systems] \label{rem:example_pte_systems}
    In general, finding an analytic solution to the PTE system is not possible; however, there are few special cases that can be analyzed in more detail.
    The simplest case is ideal gas mixtures; see \citep{renac2021entropy, trojak2024positivity, CLAYTON2026107159}.
    In slightly more generality one can find an exact solution for PTE when there is one stiffened-gas EOS and $M-1$ ideal gases; see \citet[Sec.~4.2]{wong2022positivity} and generalizing more with the Nobel-Abel stiffened gas law in \citet{collis2026robust}.
    See also \citet{borisov2018exact} wherein they use a mixture equation of state similar to a stiffened-gas law.
    However, for multiple stiffened-gas laws, iterative methods must be used to get the PTE solution.
    Another notable case is provided in \citet{aslam_2021}, where they analyze a PTE mixture for two different Mie-Gr{\"u}neisen equations of state and derive a simple form to apply iterative methods.
\end{remark}

%% file: pte_solver.tex
\section{Tabular approximation} \label{sec:tabular_approximation}
We now present a method for constructing a tabular equation of state from some discrete set of data.
The construction process is not the sole focus of the paper; however, we illustrate one possible approximation method and discuss several issues with tabular approximations.

The tabular equation of state that we shall use consists of a set of discrete data for the density and specific internal energy of each material.
We then provide an interpolation method between the discrete data points.
There are many ways to construct tabular approximations and each method has advantages and disadvantages.
Some early work for surface approximation of an EOS can be seen in \citet{carlson1985monotone}; however, it assumes monotonicity of the data.
\citet{dilts2006consistent} covers this topic in great detail and provides a novel interpolation method which guarantees thermodynamic consistency and stability.
Alternatively, a bi-quintic interpolation method is proposed for the Helmholtz free energy in \citet{swesty1996thermodynamically}, this method guarantees thermodynamic consistency in the sense that the energy and pressure are continuous and that the Maxwell relations are preserved; see also \citet[Sec.~2]{timmes2000accuracy}.
We also refer to \citet[Sec.~3]{foll2019use} which utilizes an adaptive mesh refinement procedure to improve the tabular approximation.

We now proceed with a single simple approximation method, noting the deficiencies along the way.

\subsection{Tabular data} \label{sec:tabular_data}
In order to simplify the tabular approximation and resulting PTE solves, we tabulate each EOS with respect to pressure on the pressure domain, $\calP = (0, \infty)$.
This avoids some complexity with tabulating across different pressure domains while losing some of the generality.
However, in general, one can tabulate each EOS in its own unique way.

Let $N_p \in \polN$ denote the number of grid points for the pressure domain. 
Define $\frakP \eqq \{P_i\}_{i=0}^{N_p}$, with each $P_i\in (0, \infty)$, be an ordered collection of unique pressure values where $P_{\min} \eqq P_0 \eqq \epsilon_P$ and $P_{\max} \eqq P_{N_p}$ where $\epsilon_P > 0$. 
We refer to $\frakP$ as the \textit{pressure partition}.
Similarly, we let $N_t \in \polN$ denote the number of temperature values.
We define $\frakT \eqq \{T_j\}_{j=0}^{N_t}$ to be an ordered collection of temperature values where each $T_j > 0$.
We refer to $\frakT$ as the \textit{temperature partition} and denote $T_{\min} \eqq T_0 = \epsilon_T > 0$ to be the minimum temperature and $T_{\max} \eqq T_{N_t}$.
This rectangular domain is denoted by 
\begin{equation}
    \sfR \eqq [P_{\min}, P_{\max}] \times [T_{\min}, T_{\max}].
\end{equation}
Details of the constructions of these partitions are outlined in Section~\ref{sec:tab_approx}.
Numerically, we choose $0 < T_{\min} \ll 1$ since the 3rd law of thermodynamics indicates we can never reach absolute zero.
Note that some EOS are not well-defined at $T = 0$ as mentioned in Remark~\ref{rem:multi_valued_cold}.

We then retrieve (or construct) a tabulated set of data for each material density and material specific internal energy, $\rho_m^{i,j} \eqq \rho_m(P_i, T_j)$ and $e^{i,j}_m \eqq e_m(P_i, T_j)$, respectively, for all $m \in \intset{1}{M}$, $i \in \intset{0}{N_p}$, and $j\in\intset{0}{N_t}$.
We shall also use the specific volume quantities, $\tau^{i,j}_m \eqq 1 / \rho^{i,j}_m$.
All interpolation of thermodynamic information will be performed on these density (or specific volume) and specific internal energy tables.
It is possible that the solution to the PTE system, \eqref{eq:mass_frac_pte_system}, exists outside of the bounded domain, $[P_{\min}, P_{\max}] \times [T_{\min}, T_{\max}]$. 
In general, one must choose $\sfR$ in an \textit{a posteriori} way, to ensure that any numerical experiments stay within the prescribed domain.
That is, we do not wish to perform \textit{extrapolation} from the data set.
However, extrapolating an EOS beyond some $T_{\max}$ for which the material may behave similarly to an ideal gas, is worth investigating, but this is beyond the scope of this paper.
The domain, $\sfR$, is provided for each test problem in Section~\ref{sec:numeric_results}.

\subsection{Tabular approximation} \label{sec:tab_approx}
In order to construct our tabular equation of state, we now must define how interpolation is performed.
First, we establish some notation. 
We define the ``rectangular box'' by, $\sfR^i_j \eqq [P_i,P_{i+1}) \times [T_j, T_{j+1})$, so then $\sfR = \bigcup_{i,j} \overline{\sfR^i_j}$.
We use the standard Newton divided difference notation with respect to $P$ and $T$ for some quantity, $\sfX$.
That is, $[\sfX^{i,j}, \sfX^{i+1,j}] = \frac{\sfX^{i+1,j} - \sfX^{i,j}}{P_{i+1} - P_i}$ and $[\sfX^{i,j}, \sfX^{i,j+1}] = \frac{\sfX^{i,j+1} - \sfX^{i,j}}{T_{j+1} - T_j}$.
We also define $\Delta P_i \eqq P_{i+1} - P_i$ and $\Delta T_j \eqq T_{j+1} - T_j$.

The simplest partition one can construct is the uniform partition.
That is, $\Delta P_{i-1} = \Delta P_i$ for all $i\in\intset{1}{N_p}$.
However, EOS tend to have certain asymptotic features that could pose problems for this partition.
For example, EOS used for modeling solids have incompressible behavior and therefore $\big(\pdv{\rho}{P}\big)_T \approx 0$; especially in the more compressed regions.
Mathematically, we must have $\big(\pdv{\rho}{P}\big)_T > 0$, but finite decimal approximation in computer systems may result in $[\rho^{i,j}, \rho^{i+1,j}] = 0$ if $P_{i+1} - P_i$ is not large enough.
We must be careful in constructing the tabular approximation to avoid such issues.

An alternative pressure and temperature partition space we construct is a logarithmic one.
Since pressure values can range from extremely large to extremely small (near zero) and even to extremely negative values, a logarithmic approach can provide better approximation to the EOS for a smaller cost.
In particular,
\begin{align}
    \Delta_{\text{log}} \sfY &\eqq \frac{\log_b(\sfY_{\max}) - \log_b(\sfY_{\min})}{N_\sfY} \\
    \sfY_i &\eqq b^{\log_b(\sfY_{\min}) + i \Delta_{\text{log}}\sfY}
\end{align}
for $\sfY$ either $P$ or $T$ and where $b$ is some base; we shall use $b = 10$.
Then $\Delta \sfY_i \eqq \sfY_{i+1} - \sfY_i$.
Note however that this definition requires $\sfY_{\min} > 0$.

The tabular approximation we construct is what we describe as a, \textit{rational form approximation}.
Generically, the form is given by $f(x,y) = \frac{a + by}{c + dx}$ where $a$, $b$, $c$, and $d$ are selected to interpolate the points of interest.
We define the tabular approximation of the specific volume by,
\begin{subequations} \label{eq:rational_tau_approx}
\begin{align}
    \ttau_m(P,T) &\eqq \sum_{i,j} \ttau_m^{i,j}(P,T) \chi_{\sfR^i_j}(P,T), \\
    \ttau_m^{i,j}(P,T) &\eqq \frac{(T_{j+1} - T) / \Delta T_j}{\rho_m^{i,j} + [\rho_m^{i,j}, \rho_m^{i+1,j}](P - P_i)} + \frac{(T - T_j) / \Delta T_j}{\rho_m^{i,j+1} + [\rho_m^{i,j+1}, \rho_m^{i+1,j+1}](P - P_i)} \label{eq:tau_m_approx}
\end{align}
\end{subequations}
where $\chi_{\sfR^i_j}$ is the characteristic function.
In this case we see that each $\ttau_m^{i,j}(P,T)$ is a linear function of $T$ and a rational function of $P$.
This is motivated from the ideal gas law where $\tau = \frac{RT}{P}$.

\begin{lemma}[Approximation properties]
    The function $\ttau_m$ defined in \eqref{eq:rational_tau_approx} is $H^1$-conforming on $[P_{\min}, P_{\max}] \times [T_{\min}, T_{\max}]$.
\end{lemma}

\begin{proof}
    Since $\ttau^{i,j}_m$ is a rational function, it is immediate that $\ttau^{i,j}_m \in C^\infty(\sfR^i_j)$ for each $i$, $j$.
    In order to show that $\ttau_m$ is continuous, we first note that $\ttau_m(P_i,T_j) = \tau_m^{i,j}$ for all $i \in \intset{0}{N_p-1}$ and $j \in \intset{0}{N_t-1}$.
    To see that $\ttau_m$ is continuous, we need only check that $\ttau_m$ is continuous at the cell interfaces, $\overline{\sfR^i_j} \cap \overline{\sfR^{i+1}_j}$ and $\overline{\sfR^i_j} \cap \overline{\sfR^i_{j+1}}$.
    Consider $\ttau_m^{i,j}(P_{i+1}, T)$ and $\ttau_m^{i+1,j}(P_{i+1}, T)$.
    Since both $\ttau_m^{i,j}$ and $\ttau^{i+1,j}_m$ are linear in $T$ and there exists exactly one linear function which interpolates two points, we must have that $\ttau_m^{i,j}(P_{i+1}, T) = \ttau_m^{i+1,j}(P_{i+1}, T)$.
    Similarly, for $\overline{\sfR^i_j} \cap \overline{\sfR^i_{j+1}}$, by construction in \eqref{eq:tau_m_approx}, we have that $\ttau_m^{i,j}(P, T_{j+1}) = \ttau_m^{i,j+1}(P, T_{j+1})$.
    Lastly, the derivatives are all continuous on the interior of each cell, $\sfR^i_j$.
    Therefore, $\ttau_m \in H^1(\sfR)$.
\end{proof}

For completeness, we also provide the partial derivatives which are required for the numerical algorithms,
\begin{equation}
\begin{split}
    \Big( \pdv{\ttau_m^{i,j}}{P} \Big)_T &= -\frac{[\rho_m^{i,j}, \rho_m^{i+1,j}] (T_{j+1} - T) / \Delta T_j}{\big(\rho_m^{i,j} + [\rho_m^{i,j}, \rho_m^{i+1,j}](P - P_i)\big)^2} \\
    &\qquad - \frac{[\rho_m^{i,j+1}, \rho_m^{i+1,j+1}](T - T_j) / \Delta T_j}{\big(\rho_m^{i,j+1} + [\rho_m^{i,j+1}, \rho_m^{i+1,j+1}](P - P_i)\big)^2}
\end{split}
\end{equation}
and
\begin{equation}
\begin{split}
    \Big( \pdv{\ttau_m^{i,j}}{T} \Big)_P &= \frac{-1 / \Delta T_j}{\rho_m^{i,j} + [\rho_m^{i,j}, \rho_m^{i+1,j}](P - P_i)} \\
    &\qquad + \frac{1 / \Delta T_j}{\rho_m^{i,j+1} + [\rho_m^{i,j+1}, \rho_m^{i+1,j+1}](P - P_i)}.
\end{split}
\end{equation}
For the interpolation of the specific internal energy, we use a form which depends on the approximation $\ttau_m^{i,j}$; the reason for this is explained in Remark~\ref{rem:tab_inversion}.
The approximation is given by,
\begin{equation} \label{eq:sie_rational_approx}
\begin{split}
    \te_m^{i,j}(P,T) &= \frac{T_{j+1} - T}{\Delta T_j} \Big( e^{i,j}_m + \frac{e_m^{i+1,j} - e_m^{i,j}}{\tau^{i+1,j}_m - \tau^{i,j}_m} (\ttau_m(P, T_j) - \tau^{i,j}_m) \Big) \\
    &\qquad + \frac{T - T_j}{\Delta T_j} \Big( e^{i,j+1}_m + \frac{e_m^{i+1,j+1} - e_m^{i,j+1}}{\tau^{i+1,j+1}_m - \tau^{i,j+1}_m} (\ttau_m(P, T_{j+1}) - \tau^{i,j+1}_m) \Big),
\end{split}
\end{equation}
and the corresponding derivatives are,
\begin{equation}
\begin{split}
    \Big(\pdv{\te_m^{i,j}}{P}\Big)_T &= \frac{T_{j+1} - T}{\Delta T_j} \cdot \frac{e_m^{i+1,j} - e_m^{i,j}}{\tau^{i+1,j}_m - \tau^{i,j}_m} \Big( \pdv{\ttau_m}{P} \Big)_T(P, T_j) \\ 
    &\qquad + \frac{T - T_j}{\Delta T_j} \frac{e_m^{i+1,j+1} - e_m^{i,j+1}}{\tau^{i+1,j+1}_m - \tau^{i,j+1}_m} \Big( \pdv{\ttau_m}{P} \Big)_T(P, T_{j+1})
\end{split}
\end{equation}
and
\begin{equation}
\begin{split}
    \Big( \pdv{\te_m}{T} \Big)_P = \frac{e^{i,j+1}_m - e^{i,j}_m}{\Delta T_j} 
    &+ \frac{1}{\Delta T_j} \Big( \frac{e^{i+1,j+1}_m - e^{i,j+1}}{\tau^{i+1,j+1}_m - \tau^{i,j+1}_m} (\ttau_m(P,T_{j+1}) - \tau^{i,j+1}_m) \\
    &\qquad - \frac{e^{i+1,j}_m - e^{i,j}_m}{\tau^{i+1,j}_m - \tau^{i,j}_m}(\ttau_m(P,T_j) - \tau^{i,j}_m) \Big).
\end{split}
\end{equation}

Finally, the bulk approximate quantities are defined by,
\begin{subequations}
\begin{align}
    \ttau(P, T, \bY) &\eqq \sum_{m=1}^M Y_m \ttau_m(P, T), \\
    \te(P, T, \bY) &\eqq \sum_{m=1}^M Y_m \te_m(P, T).
\end{align}
\end{subequations}

Note, we make no claims about thermodynamic consistency or stability of this tabular approximation, only that the EOS can be properly inverted, see Remark~\ref{rem:tab_inversion}.
In our numerical experiments in Section~\ref{sec:numeric_results}, we have not found any violation of thermodynamic stability.

\begin{remark}[EOS inversion] \label{rem:tab_inversion}
    The explicit use of $\ttau_m$ in \eqref{eq:sie_rational_approx} is necessary to preserve the invertability of the equation of state.
    For example, assume we construct $\ttau_m$ as in \eqref{eq:tau_m_approx} but we instead construct $\widetilde{\te}_m(P,T)$ independently of $\ttau_m$ by some other means.
    Then, if we provide EOS data of the form $(\tau_0, e_0, T_0)$.
    Let $P_0$ be the root of the equation $\ttau_m(P, T_0) = \tau_0$ and $P_0'$ the root of $\widetilde{\te}_m(P, T_0) = e_0$, then in general $P_0 \neq P_0'$.
    Hence the equation of state for pressure is \textit{inconsistent}!
\end{remark}

\begin{remark}[Real data]
    The constructions we have provided are somewhat simple since we define a single local interpolation method from some predefined data. 
    In general, for a global tabular EOS, constructing a highly accurate approximation is incredibly complex.
    The process consists of a variety of techniques, from curve fitting experimental data to applications of density functional theory (DFT) among a host of other methods.
    See \citet{sjostrom2016multiphase} and \citet{rehn2021multiphase} for more details of these constructions.
\end{remark}

\section{Pressure-temperature equilibrium preliminaries} \label{sec:pte_solving}

We now proceed by establishing a few preliminaries as well as some standard numerical methods for solving nonlinear systems of equations.
We drop the $\, \widetilde{} \,$ notation used in Section~\ref{sec:tab_approx} since there is no requirement that the equations of state must be tabulated.
First, recall the system we are interested in solving is:
\begin{subequations}
\begin{align}
    \sum_{m=1}^M Y_{m,0} \tau_m(P,T) - \tau_0 &= 0, \\
    \sum_{m=1}^M Y_{m,0} e_m(P,T) - e_0 &= 0.
\end{align}
\end{subequations}
We shall transform this system into a more useful form.
Since $\lim_{P \to \calP^+} \tau(P,T) \to \infty$, Newton methods may converge slowly when near the asymptotic regime, $\tau \gg 1$.
A better approach is to instead solve the equation,
\begin{equation} \label{eq:density_pte}
    \varrho(P,T) \eqq \frac{1}{\sum_{m=1}^M Y_{m,0} \tau_m(P,T)} - \frac{1}{\tau_0},
\end{equation}
since the density behaves more linearly with respect to $P$.
See Figure~\ref{fig:density_spc_volume_iso} for a comparison of some generic functions.

\begin{figure}[htbp]
\centering
\begin{tikzpicture}
    \begin{axis}[
      width=10cm, height=7cm,
      xlabel={\(P\)},
      title={Density vs. specific volume isotherms},
      legend pos=north east,
      legend cell align=left,
      grid=major,
    ]
        \addplot[domain=0:10,samples=200,thick,dashed,blue] {x^2};
        \addplot[domain=0.1:10,samples=200,thick,dotted,red] {1/x^2};
        \addlegendentry{$\rho$}
        \addlegendentry{$\tau$}
    \end{axis}
\end{tikzpicture}
\caption{Comparison of generic isotherms for density and specific volume.
Note the behavior of $\rho$ is more suitable for Newton solves compared to $\tau$.}
\label{fig:density_spc_volume_iso}
\end{figure}
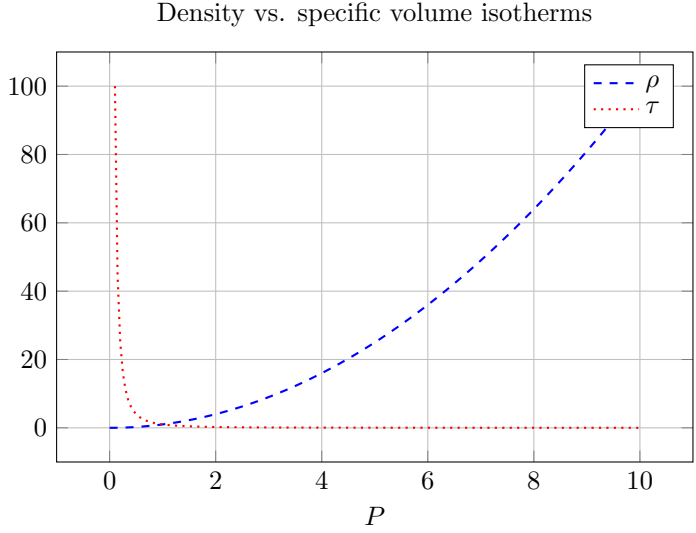

We also have that $\partial_P \varrho = \rho K_T > 0$. 
Furthermore, $\partial_P^2 \varrho = \partial_P(1 / c_T^2) = -\frac{2}{c_T^3} \big(\pdv{c_T}{P}\big)_T$, where $c_T^2 = 1 / (\rho K_T)$ is the isothermal sound speed. 
Hence, if $c_T$ is an increasing function of $P$ on each isotherm, then $\varrho$ is an increasing concave function of $P$.
For $e(P,T)$, there is typically non-monotonic behavior in either of the variables, $e(P,T)$ may be increasing or decreasing with respect to each variable.
Therefore, we instead opt to solve a specific enthalpy equation for the temperature; that is,
\begin{equation} \label{eq:enthalpy_pte}
    \frakh(P,T) \eqq \sum_{m=1}^M Y_{m,0} \big(e_m(P,T) + P \tau_m(P,T)\big) - (e_0 + P \tau_0) = 0.
\end{equation}
Notice that $\partial_T \frakh(P,T) = \partial_T h(P,T) = c_p > 0$.

\subsection{Stopping criteria}
Handling the convergence of the iterative methods requires a bit of care.
The specific internal energy, like the pressure, can range across many different scales including negative values.
The specific enthalpy, $h$, also maintains this same behavior.
We use the following relative error stopping criteria for the iterative methods,
\begin{align}
    |\tau(P, T) - \tau_0| &< \epsilon \tau_0, \label{eq:tau_error_indicator} \\
    |e(P, T) - e_0| &< \epsilon (e_{\text{ref}} + |e_0|), \label{eq:h_error_indicator}
\end{align}
where $\epsilon > 0$ is a unitless parameter typically around $10^{-12}$ and $e_{\text{ref}} > 0$ is some reference specific internal energy.
Since the specific internal energy can be identically zero, we incorporate the $e_\text{ref}$ to avoid this degeneracy.

\subsection{Single-Newton full-bisection method}
Let $(P^{(0)}, T^{(0)}) \in \calP \times \Real_+$ be an initial guess, we first perform a full bisection method on the temperature to find the equilibrated temperature at the initial $P^{(0)}$.
That is, we find $T^{(1)}$ which solves $e(P^{(0)},T^{(1)}) - e_0 = 0$.
We then apply one step of a standard Newton method to the equation $\varrho(P,T^{(1)}) = 0$ to find $P^{(1)}$.
This process is repeated until convergence is met.
The specific algorithm for this method is provided in Algorithm~\ref{alg:split_newton_PTE_solver}.

\begin{algorithm}
\caption{Single-Newton full-bisection method.}
\label{alg:split_newton_PTE_solver}
\begin{algorithmic}
    \Require $e_0$, $\tau_0$, $\bY_0$, $(P^{(0)}, T^{(0)})$
    \State \textbf{set:} $n = 0$
    \While{(\eqref{eq:tau_error_indicator} \textbf{and} \eqref{eq:h_error_indicator} \textbf{false})}
        \State $T^{(n+1)} = \texttt{Bisection}(e(P^{(n)}, T) - e_0)$ \Comment{Run bisection to convergence}
        \State $P^{(n+1)} \gets P^{(n)} - \frac{\varrho(P^{(n)}, T^{(n+1)})}{\partial_P \varrho(P^{(n)}, T^{(n+1)})}$  \Comment{one Newton step}
        \If{$P^{(n+1)} < P_{\min}$ \textbf{or} $P^{(n+1)} > P_{\max}$}
            \State $P^{(n+1)} \gets \max(P_{\min}, P^{(n+1)})$ \Comment{Catch for out of bounds Newton step}
            \State $P^{(n+1)} \gets \min(P_{\max}, P^{(n+1)})$
        \EndIf
        \State $n \gets n + 1$
    \EndWhile  
    \State \Return{$(P^{(n+1)}, T^{(n+1)})$}
\end{algorithmic}
\end{algorithm}

\subsection{2D Newton method}
One of the more common numerical methods for solving a system of nonlinear equations is the 2D Newton method.
The convergence can be very rapid; however, the method requires a good initial guess and is often supplemented with other fail-safes to improve its robustness.
Let,
\begin{equation}
    \bsfE(P,T) \eqq \begin{pmatrix}
        \varrho(P,T) \\
        e(P,T) - e_0
    \end{pmatrix}
\end{equation}
then the Jacobian is,
\begin{equation}
    D\bsfE(P,T) \eqq \begin{bmatrix}
        \rho K_T & -\rho \beta \\
        P\tau K_T - \beta \tau T & c_p - P \tau \beta
    \end{bmatrix} = 
    \begin{bmatrix}
        \big(\pdv{\rho}{P}\big)_T & \big(\pdv{\rho}{T}\big)_P \\
        \big(\pdv{e}{P}\big)_T & \big(\pdv{e}{T}\big)_P
    \end{bmatrix}.
\end{equation}
Then using \eqref{eq:important_thermo_iden} we find that $\det(D\bsfE(P,T)) = \rho c_v K_T > 0$.
The 2D Newton method is described by,
\begin{equation} \label{eq:2dnewton_pte_update}
    \begin{pmatrix}
        P^{(n+1)} \\ T^{(n+1)}
    \end{pmatrix}
    = 
    \begin{pmatrix}
        P^{(n)} \\ T^{(n)}
    \end{pmatrix}
    - \frac{1}{\rho c_v K_T} \begin{bmatrix}
        c_p - P\tau \beta & \rho \beta \\
        \tau (\beta T - P K_T) & \rho K_T
    \end{bmatrix}
    \bsfE(P^{(n)}, T^{(n)}).
\end{equation}
Note however, it is more computationally efficient to use the derivatives of $\rho$ and $e$ rather than the named thermodynamic derivatives: $c_v$, $c_p$, $K_T$, and $\beta$.

\begin{remark}[Global convergence]
    Global convergence of the Newton method is studied in \citet{more1971global}, and sufficient conditions are provided which guarantee convergence.
    However, all of these assumptions require the function, $\bsfE$, to be $C^\infty$ and convex.
    This assumption is certainly not true for this tabular equation of state mixture and as we see in Section~\ref{sec:numeric_results}, the initial guess can have a much more limited domain for convergence.
\end{remark}

\subsubsection{Damped Newton method}
It is common in iterative solvers to apply some sort of ``safeguard'' or ``half measure'' to improve the robustness of the solver.
We have implemented a damped Newton method using an Armijo-type backtracking line search to find the Newton step which decreases the residual, $\Vert \bsfE(P,T) \Vert^2_{\ell^2}$.
The generic form is provided by,
\begin{equation}
    \bx^{(n+1)}_k = \bx^{(n)} - \alpha^k (D\bF(\bx^{(n)}))^{-1} \bF(\bx^{(n)}), \quad k \in \intset{0}{K}.
\end{equation}
where the next iterate is selected as the first $k$ which satisfies 
\begin{equation}
    \Vert \bF(\bx^{(n+1)}_k) \Vert_{\ell^2} < (1 - c\alpha^k) \Vert \bF(\bx^{(n)}) \Vert_{\ell^2}.
\end{equation}
The PTE notation is simply $\bF = \bsfE$ and $\bx^{(n)} = (P^{(n)}, T^{(n)})^{\sfT}$.
(Note the $k$ in $\alpha^k$ signifies an exponent and \textit{not} an index.)
For all numerical results in Section~\ref{sec:numeric_results}, we have used $\alpha = 1/2$, $c = 10^{-4}$, and $K = 4$.
If none of the $k \in \intset{0}{K}$ decreases the residual sufficiently, then we set $\bx^{(n+1)} \eqq \bx^{(n+1)}_0$.
Note, we cap the back-tracking at $K=5$, since the EOS evaluations are generally expensive.

%% file: cyclic_method.tex
\section{The cyclic method} \label{sec:the_cyclic_method}
While the 2D Newton method converges at a quadratic rate, it may fail to converge unless the initial guess is close enough to the root.
Alternatively, the Newton-bisection method typically converges even when the initial guess is far away from the solution; however, the computational time is quite large.
Solving PTE is just one part of solving the four-equation model, therefore the iterative solver should converge rapidly and accurately, since there could be billions of PTE solves at every time step.

There have been many proposed methods to improve the robustness of standard iterative solvers, such as using a bisection method to find a safe starting point for the iterative method; see \citet{moore1977safe}, or the trust region method (see \citet{sorensen1982newton} and \citet{yuan2015recent}) which solves a safer subproblem (typically a minimization problem) to obtain the next iterate.
We propose a novel and efficient method for solving nonlinear systems of equations from $\Real^2$ to $\Real^2$ that converges rapidly and is often more robust to the initial guess than standard Newton methods.
To the best of the author's knowledge, such an algorithm has not appeared in the literature.

\subsection{The general form}
We first present the method in a generic form solving, $\bF(\bx) = \ba$, where $\bF : \Real^2 \to \Real^2$ is some nonlinear function and $\ba = (a,b)^\sfT \in \Real^2$ is the initial datum. 
To be more explicit, we would like to solve the system:
\begin{subequations}
\begin{align}
    f(x,y) = a, \\
    g(x,y) = b.
\end{align}
\end{subequations}
for $f, g : \Real^2 \to \Real$, some nonlinear functions where $\bF(\bx) \eqq (f(\bx), g(\bx))^\sfT$ and $\bx = (x,y)$.
The method we propose, takes advantage of the \textit{cyclic rule} (also known as the \textit{triple product rule}), see Corollary~\ref{cor:cyclic_rule}.

\subsection{The cyclic method algorithm} \label{sec:cyclic_method_general}
We now describe the algorithm for the cyclic method.
For a quick illustration of the procedure see Figure~\ref{fig:cyclic_method_illustration}.
We define the following 1-dimensional solution manifolds in the $(x,y)$-plane:
\begin{align}
    \frakF_0 &\eqq \{ (x, y) : f(x,y) = 0\}, \\
    \frakG_0 &\eqq \{ (x, y) : g(x,y) = 0\}.   
\end{align}
To simplify upcoming notation, we also define the lines in the $x$-$y$ plane by,
\begin{subequations}
\begin{align}
    \sfI_\sfX^y(x_0, y_0) &\eqq \{(x,y) : y = y_0 + \Big(\pdv{y}{x}\Big)_\sfX(x_0, y_0)(x - x_0)\}, \\
    \sfI_\sfX^x(x_0, y_0) &\eqq \{(x,y) : x = x_0 + \Big(\pdv{x}{y}\Big)_\sfX(x_0, y_0)(y - y_0)\}.
\end{align}
\end{subequations}
for $\sfX$ some fixed quantity.
A general outline of the method is now provided.
Let $(x_0, y_0)$ be an initial guess, then:
\begin{enumerate}
    \item Take a Newton step in the $x$-direction for $f(x,y) = 0$ from the point $(x_0, y_0)$ to find a point $(\tx_0, y_0)$.
    \item Construct the tangent line, $\sfI_f^x(\tx_0, y_0)$, to the curve $\frakF_{f(\tx_0, y_0)}$ using the cyclic rule (Corollary~\ref{cor:cyclic_rule}).
    \item Take a Newton step in the $y$-direction for $g(x,y) = 0$ from the point $(\tx_0, y_0)$ to find a point $(\tx_0, \ty_0)$.
    \item Construct the tangent line, $\sfI_g^y(\tx_0, \ty_0)$, to the curve $\frakG_{g(\tx_0, \ty_0)}$ using the cyclic rule (Corollary~\ref{cor:cyclic_rule}).
    \item Find the intersection point, $(x_1, y_1)$, of $\sfI_f^x(\tx_0, y_0)$ and $\sfI_g^y(\tx_0, \ty_0)$.
\end{enumerate}

We now provide more explicit details on the construction of the tangent lines.
After the first Newton step, we obtain the point $(\tx_0, y_0)$.
The slopes of the tangent lines are provided by the cyclic rule, respectively as,
\begin{equation} \label{eq:cyclic_formula}
    \Big(\pdv{x}{y}\Big)_f(\tx_0, y_0) = - \frac{\big(\pdv{f}{y}\big)_x(\tx_0, y_0)}{\big(\pdv{f}{x}\big)_y(\tx_0, y_0)} \quad \text{ and } \quad \Big( \pdv{y}{x} \Big)_g(\tx_0, \ty_0) = -\frac{\big(\pdv{g}{x}\big)_y(\tx_0, \ty_0)}{\big(\pdv{g}{y}\big)_x(\tx_0, \ty_0)}.
\end{equation}
The intersection point of the two tangent lines, $\sfI^x_f(\tx_0, y_0) \cap \sfI^y_g(\tx_0, \ty_0)$, is given by
\begin{subequations}
\label{eq:line_intersec_pts}
\begin{align}
    x_1 &\eqq \frac{
        \tx_0 + \big( \pdv{x}{y} \big)_f(\tx_0, y_0) \big(\ty_0 - y_0 - \tx_0 \big( \pdv{y}{x} \big)_g(\tx_0, \ty_0) \big)
    }{
        1 - \big( \pdv{x}{y}\big)_f(\tx_0, y_0) \big( \pdv{y}{x} \big)_g(\tx_0, \ty_0)
    }, \label{eq:slope_intersection_x} \\
    y_1 &\eqq \ty_0 + \Big( \pdv{y}{x} \Big)_g(\tx_0, \ty_0) (x_1 - \tx_0).
\end{align}
\end{subequations}
If the denominator is zero, this implies that the two tangent lines are parallel, therefore, the application of this method should assess the existence of the intersection point.
This process can be iterated until convergence and the concise description is provided in Algorithm~\ref{alg:cyclic_method}.
\begin{figure}
    \centering
    \includegraphics[width=0.95\linewidth]{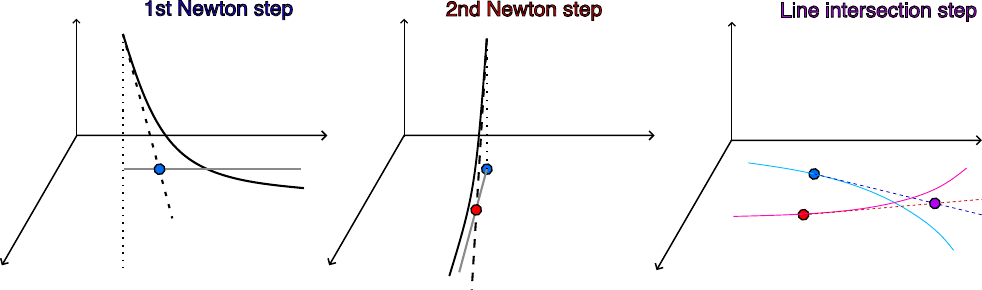}
    \caption{The three steps of the new method.}
    \label{fig:cyclic_method_illustration}
\end{figure}
\begin{remark}[Change of variables]
This method can be generalized even more and tailored to the problem of interest.
In particular, the derivatives in the $x$-$y$ plane do not require $f$ and $g$ to be constant.
We can instead choose some other quantity dependent on $f$ and $g$.
For example, one could define mappings, $p = p(f,g,x,y)$ and $q = q(f,g,x,y)$, and apply standard calculus techniques to derive $\big(\pdv{x}{y}\big)_p$ and $\big( \pdv{y}{x} \big)_q$ for use in the cyclic algorithm (assuming these alternative tangent lines take is towards the solution).
This allows us to avoid possible singularities in the derivatives and provide possible other improvements in convergence rates and robustness.
This change in derivatives is used in the PTE application in Section~\ref{sec:cyclic_pte_app}.
\end{remark}

\begin{algorithm}
\caption{Generic cyclic method.}
\label{alg:cyclic_method}
\begin{algorithmic}
    \Require $x_0$, $y_0$
    \State \textbf{set} $n = 0$
    \While{(\eqref{eq:tau_error_indicator} \textbf{and} \eqref{eq:h_error_indicator} \textbf{false})}
        \State $\tx_{n} = x_n - \frac{f(x_n,y_n)}{\partial_x f(x_n, y_n)}$ \Comment{Newton step}
        \State \textbf{compute} $\big(\pdv{x}{y}\big)_f(\tx_{n}, y_n)$ from \eqref{eq:cyclic_formula}
        \State $\ty_{n} = y_n - \frac{g(\tx_{n}, y_n)}{\partial_y g(\tx_{n}, y_n)}$ \Comment{Newton step}
        \State \textbf{compute} $\big(\pdv{y}{x}\big)_g(\tx_{n}, \ty_{n})$ from \eqref{eq:cyclic_formula}
        \State \textbf{compute} $(x_{n+1}, y_{n+1})$ from \eqref{eq:line_intersec_pts}  \Comment{Intersection of lines}
        \State $n \gets n+1$
    \EndWhile
    \State \Return{$(x_n, y_n)$}
\end{algorithmic}
\end{algorithm}

\subsection{Application of the cyclic method to PTE} \label{sec:cyclic_pte_app}
We now proceed with the application of the method presented in Section~\ref{sec:cyclic_method_general} to the PTE system formulated by \eqref{eq:density_pte} and \eqref{eq:enthalpy_pte}.
All calculations are the same as in Section~\ref{sec:cyclic_method_general} but with $(x,y)$ replaced with $(P,T)$, and $f$ and $g$ replaced with $\varrho$ and $\frakh$.
For reference we have the following thermodynamic identities from the cyclic rule:
\begin{align}
    \Big( \pdv{P}{T} \Big)_\varrho &= -\frac{\big(\pdv{\varrho}{T}\big)_P}{\big(\pdv{\varrho}{P}\big)_T} = \frac{\beta}{K_T}, \label{eq:dpdT_rho} \\
    \Big( \pdv{T}{P} \Big)_\frakh &= -\frac{\big(\pdv{\frakh}{P}\big)_T}{\big(\pdv{\frakh}{T}\big)_P} = \frac{\beta T \tau + \tau_0 - \tau}{c_p}, \label{eq:dTdp_h}
\end{align}
and we shall also use the identity, $\big( \pdv{T}{P} \big)_e = -\big(\pdv{e}{P}\big)_T/\big( \pdv{e}{T} \big)_P$.
In particular, if the line intersection step lands out of bounds, then we swap the tangent line slope, $\big( \pdv{T}{P} \big)_\frakh$ with $\big( \pdv{T}{P} \big)_e$.
We have found that this improves the robustness of the cyclic method.
Note, each of these derivatives are finite real valued numbers since $c_p > 0$ and $K_T > 0$.
There are two possible issues, which is that the tangent lines could be parallel or $\big(\pdv{e}{T}\big)_P = 0$; if this is the case, then we skip the line intersection step and proceed with the one dimensional Newton steps.
Although, in all of our numerical tests, this condition was never triggered.
This algorithm is described in Algorithm~\ref{alg:pte_method}.

Notice also that when the solution is near the root; that is, $\tau \approx \tau_0$ and $e \approx e_0$, we have that $1 - \big(\pdv{P}{T}\big)_\varrho \big( \pdv{T}{P} \big)_\frakh \approx 1 - \frac{\beta^2 T \tau}{K_T c_p} = \frac{c_v}{c_p} > 0$.
Hence we expect that the tangent lines should never be parallel sufficiently close to the root.
\begin{algorithm}
\caption{PTE cyclic method.}
\label{alg:pte_method}
\begin{algorithmic}
    \Require $P^{(0)}$, $T^{(0)}$
    \State \textbf{set} $n = 0$
    \While{(\eqref{eq:tau_error_indicator} \textbf{and} \eqref{eq:h_error_indicator} \textbf{false})}
        \State $\tP^{(n)} = P^{(n)} - \frac{\varrho(P^{(n)},T^{(n)})}{\partial_P \varrho(P^{(n)}, T^{(n)})}$ \Comment{Newton step}
        \State $\tT^{(n)} = T^{(n)} - \frac{\frakh(\tP^{(n)}, T^{(n)})}{\partial_T \frakh(\tP^{(n)}, T^{(n)})}$ \Comment{Newton step}
            \State $(P^{(n+1)}, T^{(n+1)}) \gets \sfI_\varrho^P(\tP^{(n)}, T^{(n)}) \cap \sfI_\frakh^T(\tP^{(n)}, \tT^{(n)})$ \Comment{Line intersection}
        \If{$(P^{(n+1)}, T^{(n+1)}) \not \in \sfR$ \textbf{or} \texttt{isnan}($P^{(n+1)}$)}
            \State $(P^{(n+1)}, T^{(n+1)}) \gets \sfI_\varrho^P(\tP^{(n)}, T^{(n)}) \cap \sfI_e^T(\tP^{(n)}, \tT^{(n)})$ \Comment{Swap tangent line}
            \State $P^{(n+1)} \gets \min(\max(P^{(n+1)}, P_{\min}), P_{\max})$
            \State $T^{(n+1)} \gets \min(\max(T^{(n+1)}, T_{\min}), T_{\max})$
        \EndIf
        \If{\texttt{isnan}($P^{(n+1)}$)}
            \State $(P^{(n+1)}, T^{(n+1)}) \gets (\tP^{(n)}, \tT^{(n)})$ \Comment{Safety fallback}
        \EndIf
        \State $n \gets n+1$
    \EndWhile
    \State \Return{$(P^{(n)}, T^{(n)})$}
\end{algorithmic}
\end{algorithm}
%
%

%% file: eos.tex
\section{Equation of state} \label{sec:eos}
We describe several equations of state which will be used in the testing of the PTE solver.
We provide only the necessary details in order to perform the PTE solve.

\subsection{Ideal and stiffened gas}
The ideal and stiffened gas laws are provided by,
\begin{align}
    \rho_{\text{ideal}}(P,T) = \frac{P}{RT}, \quad &\text{ and } \quad e_{\text{ideal}}(P,T) = c_v T \\
    \rho_{\text{stiff}}(P,T) = \frac{P+P_\infty}{RT}, \quad &\text{ and } \quad e_{\text{stiff}}(P,T) = \frac{c_v T (P + \gamma P_\infty)}{P + P_\infty} + q
\end{align}
where $P_\infty > 0$, $q \in \Real$, $\gamma = c_p / c_v$ and $R = c_p - c_v$ for $c_p > c_v > 0$.
The parameters we use for ideal and stiffened gas are provided in Table~\ref{tab:ideal_stiff_gas_params}.
\begin{table}[]
    \centering
    \begin{tabular}{|c|c|c|c|c|}
    \hline
     & $c_p$ $(\SI{}{\kilo\joule\per\gram\per\kelvin})$ & $c_v$ $(\SI{}{\kilo\joule\per\gram\per\kelvin})$ & $P_\infty$ $(\SI{}{\giga\pascal})$ & $q$ $(\SI{}{\kilo\joule\per\gram})$ \\
    \hline
    Ideal gas & 0.0014 & 0.001 & - & -  \\
    \hline
    Stiffened gas & 0.0085 & 0.0012 & 0.01 & 0 \\
    \hline
    \end{tabular}
    \caption{Parameters for the ideal and stiffened gas EOS.}
    \label{tab:ideal_stiff_gas_params}
\end{table}

\subsection{Reactant and product Davis EOS}
PTE is often assumed among products and reactants in the context of high explosives detonation due to it's relative simplicity, despite the PTE assumption not being generally accurate on the detonation timescale.
This is discussed in \citet{saurel_2024} and references therein.
This motivates testing with equations of state used in HE products and reactants.
In particular, we shall use the reactant Davis and product Davis equations of state.
We follow the descriptions provided in \citet{velizhanin2023notes} and use the parameters from \citet[Tables~I. and II.]{jadrich2023uncertainty}.

The reactant Davis pressure and specific internal energy are given by,
\begin{align}
    e(\tau, T) &= e_s(\tau) + \frac{c_v^0 T_s(\tau)}{1 + \alpha_{ST}} \Big[ \Big( \frac{T}{T_s(\tau)} \Big)^{1 + \alpha_{ST}} - 1 \Big], \\
    P(\tau, T) &= P_s(\tau) + \frac{\Gamma(\tau)}{\tau} \frac{c_v^0 T_s(\tau)}{1 + \alpha_{ST}} \Big[ \Big( \frac{T}{T_s(\tau)} \Big)^{1 + \alpha_{ST}} - 1 \Big]
\end{align}
where 
\begin{subequations}
\begin{align}
    T_s(\tau) &= T_0 \begin{cases}
        \big( \frac{\tau}{\tau_0} \big)^{-\Gamma_0}, &\quad \text{ if } \tau > \tau_0, \\
        \big( \frac{\tau}{\tau_0} \big)^{-(\Gamma_0 + Z)} \exp\big(-Z \big(1 - \frac{\tau}{\tau_0} \big) \big), &\quad \text{ if } \tau \leq \tau_0,
    \end{cases} \\
    e_s(\tau) &= e_0 + \frac{A^2}{16B^2} \begin{cases}
        \exp(4By) - 4By - 1, &\, \text{ if } \tau > \tau_0, \\
        \sum_{n=2}^4 \frac{(4By)^n}{n!} + C \frac{(4By)^5}{5!} + \frac{4By^3}{3(1 - y)^3}, &\, \text{ if } \tau \leq \tau_0,
    \end{cases} \label{eq:davis_eos_e_ref} \\
    P_s(\tau) &= \frac{A^2}{4B\tau_0} \begin{cases}
        \exp(4By) - 1, &\, \text{ if } \tau > \tau_0, \\
        \sum_{n=1}^3 \frac{(4By)^n}{n!} + C \frac{(4By)^4}{4!} + \frac{y^2}{(1 - y)^4}, &\, \text{ if } \tau \leq \tau_0,
    \end{cases} \label{eq:davis_eos_p_ref} \\
    \Gamma(\tau) &= \begin{cases}
        \Gamma_0, &\, \text{ if } \tau > \tau_0, \\
        \Gamma_0 + Zy, &\,  \text{ if } \tau \leq \tau_0,
    \end{cases} \label{eq:davis_eos_gruneisen}
\end{align}
\end{subequations}
with $y \eqq 1 - \frac{\tau}{\tau_0}$.

The product Davis EOS is provided by,
\begin{subequations}
\begin{align}
    P(\tau, e) &= (\gamma - 1 + F(\tau)) \frac{e}{\tau} - \frac{bF(\tau)}{\tau} (e - e_s(\tau)), \\
    F(\tau) &= \frac{2a}{\big( \frac{\tau}{\tau_c} \big)^{2n} + 1}, \\
    e_s(\tau) &= e_c \Big( \frac{\tau}{\tau_c} \Big)^{-(\gamma-1+2a)} \Big( \frac12 \Big( \frac{\tau}{\tau_c} \Big)^{2n} + \frac12 \Big)^{a/n},
\end{align}
\end{subequations}
where $\gamma > 1$ and $\tau_c, a, b, n > 0$ are some parameters for modeling the products.
For the thermal part we have,
\begin{subequations}
\begin{align}
    e(\tau, T) &= c_v (T - T_s(\tau)) + e_s(\tau), \\
    T_s(\tau) &= T_c \Big( \frac{\tau}{\tau_c} \Big)^{-(\gamma-1+2a(1-b))} \Big( \frac12 \Big( \frac{\tau}{\tau_c} \Big)^{2n} + \frac12 \Big)^{a(1-b)/n},
\end{align}
\end{subequations}
where $c_v > 0$ is a constant specific heat capacity (at constant volume) and $T_c > 0$.
For both the reactant and product Davis EOS, we compute the density by solving the nonlinear equation $P(\tau, e(\tau,T_j)) = P_i$ for $\tau_{ij}$.
Then the energy is computed by $e_{ij} \eqq e(v_{ij}, T_j)$.

\subsection{Simple MACAW}
Lastly, we use the Simple MACAW EOS found in \citet{aslam2024simple} with the parameters for copper described in \citep[Table 1]{aslam2024simple}. 
This EOS is a simplified equation of state of the one presented in \citet{lozano2023analytic}.
We omit the details as they are presented concisely in \citep{aslam2024simple}.

%% file: numerical_results.tex
\section{Numerical results} \label{sec:numeric_results}
We deploy a series of tests to determine the efficiency and accuracy of the methods for a variety of different EOS.
All tests were performed on a MacBook Pro using the Apple M2 Max with 32GB of RAM.

\subsection{Preliminaries}
The numerical tests were capped at 100 steps for each method, where ``steps'' are defined in the following way for each method.
A step for the cyclic method consists of the two one-dimensional Newton steps with tangent line intersection.
A step for the Newton-bisection method consists of the full bisection algorithm in temperature with the single one-dimensional Newton step.
Lastly, a step for the 2D Newton method is simply the one update provided by \eqref{eq:2dnewton_pte_update}.
All EOS are approximated with 350 grid points in the pressure and temperature partitions and are distributed logarithmically; see Section~\ref{sec:tab_approx}.
Lastly, we use $e_{\text{ref}} = \SI{0.2}{\kilo\joule\per\gram}$ for all tests.

\subsection{Random data trials}
In order to demonstrate the robustness and efficiency of the new method, we generate random initial data $(P_0, T_0, \bY_0) \in \sfR \times \Delta$ and compute $e_0$ and $v_0$ from $(P_0, T_0, \bY_0)$. 
Then we generate a random initial guess $(P_\text{guess}, T_\text{guess}) \in \sfR$.
All random data is generated using the standard C++ library implementation of the 19937 Mersenne Twister (see \citet{matsumoto1998mersenne}), \texttt{std::mt19937}, with a logarithmic distribution.
If an iterative method exceeds this threshold then the test is flagged as a failure.
All tests were run with 10,000 trials each.
The pressure and temperature residuals were calculated by $|P - P_0| / |P_0|$ and $|T - T_0| / T_0$, respectively.

\subsubsection{Test 1} For the first test, we use the ideal gas and stiffened gas EOS.
The pressure-temperature domain is $\sfR = [\SI{e-9}{\giga\pascal}, \SI{e3}{\giga\pascal}] \times [\SI{0.001}{\kelvin}, \SI{e5}{\kelvin}]$.kThe statistics are reported in Table~\ref{tab:ideal_stiff_random_test}.
We note that on average, the cyclic method is approximately 2.5 times slower than the 2D Newton; however, it is about as fast the damped Newton method while being more robust with 0\% failure rate.
The accuracy in this test problem is more accurate for the cyclic method.

\begin{table}[]
    \centering
    \resizebox{\textwidth}{!}{
    \begin{tabular}{|c|c|c|c|c|}
    \hline
    \multicolumn{5}{|c|}{\textbf{Random trial \#1 (Ideal/Stiffened)}} \\
    \hline
    \multicolumn{5}{|c|}{\textbf{Bulk results}} \\
    \hline
     & Cyclic & 2D Newton & Damped Newton & Newton+Bisect \\
    \hline
    Total CPU time & \SI{0.232}{\second} & \SI{0.193}{\second} & \SI{0.417}{\second} & \SI{24.4}{\second} \\
    \hline
    Failure rate & 0.34\% & 13.94\% & 6.51\% & 0.32\% \\
    \hline
    \multicolumn{5}{|c|}{\textbf{Statistics for 8519 overlapping successful runs}} \\
    \hline
    Avg. CPU time & \SI{2.4e-5}{\second} & \SI{1.0e-5}{\second} & \SI{2.5e-5}{\second} & \SI{2.5e-3}{\second} \\
    \hline
    Avg. EOS evals & 19.6 & 8.0 & 24.5 & 547 \\
    \hline
    Avg. $P$ res. & \SI{1.05e-13}{} & \SI{1.14e-13}{} & \SI{3.92e-13}{} & \SI{3.16e-9}{} \\
    \hline
    Max. $P$ res. & \SI{2.70e-11}{} & \SI{8.23e-11}{} & \SI{1.51e-9}{} & \SI{5.21e-7}{} \\
    \hline
    Avg. $T$ res. & \SI{3.35e-14}{} & \SI{3.73e-14}{} & \SI{3.19e-13}{} & \SI{2.23e-9}{} \\
    \hline
    Max. $T$ res. & \SI{5.24e-12}{} & \SI{5.93e-11}{} & \SI{1.51e-9}{} & \SI{1.47e-7}{} \\
    \hline
    \end{tabular}
    }
    \caption{Results of the random trials for ideal and stiffened gas mixtures.}
    \label{tab:ideal_stiff_random_test}
\end{table}

\subsubsection{Test 2}
For the second test, we use the ideal gas , stiffened-gas, and the simple MACAW on the domain $\sfR = [\SI{e-12}{\giga\pascal}, \SI{e3}{\giga\pascal}] \times [\SI{e-3}{\kelvin}, \SI{e5}{\kelvin}]$.
The statistics are reported in Table~\ref{tab:ideal_stiff_macaw}.
The cyclic method is approximately two times slower than the 2D Newton method.
The failure rate is lower for all methods on this problem, while the cyclic still maintains 0\% failure rate.

\begin{table}[]
    \centering
    \resizebox{\textwidth}{!}{
    \begin{tabular}{|c|c|c|c|c|}
    \hline
    \multicolumn{5}{|c|}{\textbf{Random trial \#2 (Ideal/Stiffened/Simple MACAW)}} \\
    \hline
    \multicolumn{5}{|c|}{\textbf{Bulk results}} \\
    \hline
     & Cyclic & 2D Newton & Damped Newton & Newton+Bisect \\
    \hline
    Total CPU time & \SI{0.305}{\second} & \SI{0.219}{\second} & \SI{0.787}{\second} & \SI{59.4}{\second} \\
    \hline
    Failure rate & 0.01\% & 16.63\% & 10.14\% & 13.09\% \\
    \hline
    \multicolumn{5}{|c|}{\textbf{Statistics for 7545 overlapping successful runs}} \\
    \hline
    Avg. CPU time & \SI{2.4e-5}{\second} & \SI{1.1e-5}{\second} & \SI{2.8e-5}{\second} & \SI{3.0e-3}{\second} \\
    \hline
    Avg. EOS evals & 16.8 & 7.9 & 24.7 & 452 \\
    \hline
    Avg. $P$ res. & \SI{1.15e-10}{} & \SI{6.30e-13}{} & \SI{9.84e-13}{} & \SI{4.00e-9}{} \\
    \hline
    Max. $P$ res. & \SI{1.68e-7}{} & \SI{2.33e-9}{} & \SI{2.92e-9}{} & \SI{1.45e-6}{} \\
    \hline
    Avg. $T$ res. & \SI{1.44e-10}{} & \SI{7.21e-12}{} & \SI{9.33e-12}{} & \SI{3.12e-9}{} \\
    \hline
    Max. $T$ res. & \SI{1.36e-7}{} & \SI{4.86e-8}{} & \SI{4.86e-8}{} & \SI{1.19e-6}{} \\
    \hline
    \end{tabular}
    }
    \caption{Results of the random trials for ideal, stiffened gas, and simple MACAW mixture. The average CPU time and number of steps are only computed over successful runs.}
    \label{tab:ideal_stiff_macaw}
\end{table}

\subsubsection{Test 3} \label{sec:stat_test_3}
We now consider a much more challenging test (for all methods) and discuss some of the issues.
We consider a mixture of five materials: simple MACAW, product Davis, reactant Davis, stiffened gas, and ideal gas.
The $P$-$T$ domain is given by $\sfR = [\SI{e-9}{\giga\pascal}, \SI{e3}{\giga\pascal}] \times [\SI{e-3}{\kelvin}, \SI{5e4}{\kelvin}]$.
The statistics are presented in Table~\ref{tab:5mat_test3}.

From analyzing some of the problems which failed, we found that these methods would fail by ``ping-ponging'' between two points or ``sailing'' around the solution, unable to get within the convergence tolerance.
We believe that the tabular approximation is potential source of these problems since the approximations are not $C^1$.
Furthermore, the reactant Davis also seems to a difficult EOS to handle as seen in Section~\ref{sec:grid_5mat_test} and Section~\ref{sec:grid_davis_test}.
Refining the tables can sometimes allow a method to converge when previously it would fail.
Next, we slightly modify this test to investigate the causes of the failure in Sections~\ref{sec:stat_test_4} and \ref{sec:stat_test_5}.
\begin{table}[]
    \centering
    \resizebox{\textwidth}{!}{
    \begin{tabular}{|c|c|c|c|c|}
    \hline
    \multicolumn{5}{|c|}{\textbf{Random trial \#3 (Five material)}} \\
    \hline
    \multicolumn{5}{|c|}{\textbf{Bulk results}} \\
    \hline
     & Cyclic & 2D Newton & Damped Newton & Newton+Bisect \\
    \hline
    Total CPU time & \SI{0.730}{\second} & \SI{0.400}{\second} & \SI{1.72}{\second} & \SI{166}{\second} \\
    \hline
    Failure rate & 7.45\% & 42.86\% & 36.87\% & 28.72\% \\
    \hline
    \multicolumn{5}{|c|}{\textbf{Statistics for 4609 overlapping successful runs}} \\
    \hline
    Avg. CPU time & \SI{3.8e-5}{\second} & \SI{1.4e-5}{\second} & \SI{3.7e-5}{\second} & \SI{5.0e-3}{\second} \\
    \hline
    Avg. EOS evals & 23.2 & 9.2 & 28.4 & 455 \\
    \hline
    Avg. $P$ res. & \SI{1.69e-11}{} & \SI{2.88e-13}{} & \SI{1.66e-13}{} & \SI{1.12e-9}{} \\
    \hline
    Max. $P$ res. & \SI{4.83e-9}{} & \SI{3.07e-10}{} & \SI{2.80e-11}{} & \SI{4.21e-7}{} \\
    \hline
    Avg. $T$ res. & \SI{4.37e-10}{} & \SI{1.75e-11}{} & \SI{1.62e-11}{} & \SI{3.72e-9}{} \\
    \hline
    Max. $T$ res. & \SI{1.85e-7}{} & \SI{1.65e-8}{} & \SI{1.65e-8}{} &\SI{8.52e-7}{} \\
    \hline
    \end{tabular}
    }
    \caption{Results of the random trials for Test 3. The EOS used are, ideal, stiffened gas, simple MACAW, reactant Davis and product Davis.}
    \label{tab:5mat_test3}
\end{table}

\subsubsection{Test 4} \label{sec:stat_test_4}
We repeat the same test in Section~\ref{sec:stat_test_3} except we change $T_{\max}$ from $\SI{5e4}{\kelvin}$ to $\SI{5e3}{\kelvin}$.
The results are reported in Table~\ref{tab:5mat_test4}.
We now see that the failure rate has drastically decreased for all methods except the Newton-bisection method.
\begin{table}[]
    \centering
    \resizebox{\textwidth}{!}{
    \begin{tabular}{|c|c|c|c|c|}
    \hline
    \multicolumn{5}{|c|}{\textbf{Random trial \#4 (Five material)}} \\
    \hline
    \multicolumn{5}{|c|}{\textbf{Bulk results}} \\
    \hline
     & Cyclic & 2D Newton & Damped Newton & Newton+Bisect \\
    \hline
    Total CPU time & \SI{0.434}{\second} & \SI{0.287}{\second} & \SI{0.776}{\second} & \SI{135}{\second} \\
    \hline
    Failure rate & 0.79\% & 21.34\% & 12.52\% & 27.35\% \\
    \hline
    \multicolumn{5}{|c|}{\textbf{Statistics for 5965 overlapping successful runs}} \\
    \hline
    Avg. CPU time & \SI{3.4e-5}{\second} & \SI{1.4e-5}{\second} & \SI{3.5e-5}{\second} & \SI{5.2e-3}{\second} \\
    \hline
    Avg. EOS evals & 20.7 & 9.1 & 27.3 & 472 \\
    \hline
    Avg. $P$ res. & \SI{1.47e-11}{} & \SI{1.90e-13}{} & \SI{1.91e-13}{} &\SI{1.15e-9}{} \\
    \hline
    Max. $P$ res. & \SI{5.18e-9}{} & \SI{6.81e-11}{} & \SI{4.00e-11}{} & \SI{3.33e-7}{} \\
    \hline
    Avg. $T$ res. & \SI{4.75e-10}{} & \SI{3.44e-11}{} & \SI{3.19e-11}{} & \SI{3.84e-9}{} \\
    \hline
    Max. $T$ res. & \SI{4.51e-7}{} & \SI{6.70e-8}{} & \SI{6.70e-8}{} &\SI{1.62e-6}{} \\
    \hline
    \end{tabular} 
   }
    \caption{Results of the random trials for Test 4. The EOS used are, ideal, stiffened gas, simple MACAW, reactant Davis and product Davis.}
    \label{tab:5mat_test4}
\end{table}

\subsubsection{Test 5} \label{sec:stat_test_5}
For the last test, we simply repeat the test in Section~\ref{sec:stat_test_3} except we now omit the reactant Davis EOS.
The results are reported in Table~\ref{tab:5mat_test5}.
We see test results similar to those reported in Table~\ref{tab:5mat_test4}.
\begin{table}[]
    \centering
    \resizebox{\textwidth}{!}{
    \begin{tabular}{|c|c|c|c|c|}
    \hline
    \multicolumn{5}{|c|}{\textbf{Random trial \#5 (Four material)}} \\
    \hline
    \multicolumn{5}{|c|}{\textbf{Bulk results}} \\
    \hline
     & Cyclic & 2D Newton & Damped Newton & Newton+Bisect \\
    \hline
    Total CPU time & \SI{0.366}{\second} & \SI{0.218}{\second} & \SI{0.574}{\second} & \SI{103}{\second} \\
    \hline
    Failure rate & 0.08\% & 11.60\% & 4.47\% & 23.34\% \\
    \hline
    \multicolumn{5}{|c|}{\textbf{Statistics for 7249 overlapping successful runs}} \\
    \hline
    Avg. CPU time & \SI{3.0e-5}{\second} & \SI{1.3e-5}{\second} & \SI{3.1e-5}{\second} & \SI{4.4e-3}{\second} \\
    \hline
    Avg. EOS evals & 19.1 & 8.4 & 24.7 & 500 \\
    \hline
    Avg. $P$ res. & \SI{9.46e-12}{} & \SI{1.14e-13}{} & \SI{9.59e-14}{} &\SI{5.60e-10}{} \\
    \hline
    Max. $P$ res. & \SI{2.83e-9}{} & \SI{9.90e-11}{} & \SI{1.46e-11}{} & \SI{1.32e-7}{} \\
    \hline
    Avg. $T$ res. & \SI{5.63e-10}{} & \SI{5.58e-11}{} & \SI{6.45e-11}{} & \SI{2.31e-9}{} \\
    \hline
    Max. $T$ res. & \SI{1.23e-6}{} & \SI{3.34e-7}{} & \SI{3.34e-7}{} &\SI{8.59e-7}{} \\
    \hline
    \end{tabular}
    }
    \caption{Results of the random trials for Test 5. The EOS used are: ideal, stiffened gas, simple MACAW, and product Davis.}
    \label{tab:5mat_test5}
\end{table}

\subsection{Two-state transition} \label{sec:transition_problems}
Often when solving the four-equation model, one uses the pressure and temperature solution of the previous time step. 
Since solving the four-equation model is out of scope for this paper, we construct a series of data that mimics a transition between two materials and use the previous solution as a guess for the next data.
We use an $\arctan$ function to interpolate between the data points given by:
\begin{equation} \label{eq:contact_func}
    C(x; y_0, y_1) = \frac12 (y_0 + y_1) + \frac12 (y_1 - y_0) \frac{\arctan(kx)}{\arctan(k)},
\end{equation}
where $y_L$ and $y_R$ are the left and right thermodynamic states and $k = 100$ is a parameter used increase the sharpness of the transition.
For all problems, we select $(P_Z, T_Z) \in \sfR$ then the datum are computed by $\tau_Z = \tau(P_Z, T_Z, \bY_Z)$ and $e_Z = e(P_Z, T_Z, \bY_Z)$ for $Z \in \{L,R\}$.
The transitional data functions are described by $C(x; \tau_L, \tau_R)$, $C(x; e_L, e_R)$, $C(x; Y_{0,L}, Y_{0,R})$, and $Y_1(x) = 1 - C(x; Y_{0,L}, Y_{0,R})$ for all $x \in (-1, 1)$.
The sampling domain is $D = [-1,1]$.
Setting $N = 200$, the sample points are given by $x_i = -1 + \frac{2i}{N+1}$ for $i \in \intset{1}{N}$.
Plots of the transitional data are provided in Figures~\ref{fig:approx_contact_sie} and \ref{fig:approx_contact_tau}.
For all tests in this section, we use $\sfR = [\SI{e-8}{\giga\pascal}, \SI{100}{\giga\pascal}] \times [\SI{0.001}{\kelvin}, \SI{5e4}{\kelvin}]$.

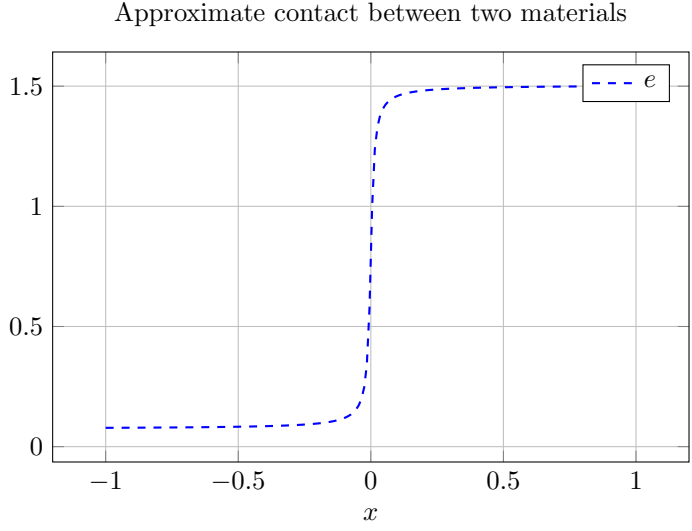
\begin{figure}[htbp]
\centering
\begin{tikzpicture}
    \begin{axis}[
      width=10cm, height=7cm,
      xlabel={\(x\)},
      title={Approximate contact between two materials},
      legend pos=north east,
      legend cell align=left,
      grid=major,
    ]
        \addplot[domain=-1:1,samples=200,thick,dashed,blue] {atan_interp(0.0784, 1.50, 100.0, x)};
        \addlegendentry{$e$}
    \end{axis}
\end{tikzpicture}
\caption{Plots of the approximate contact for $e$ defined $C(x;y_0, y_1)$ in \eqref{eq:contact_func}.}
\label{fig:approx_contact_sie}
\end{figure}

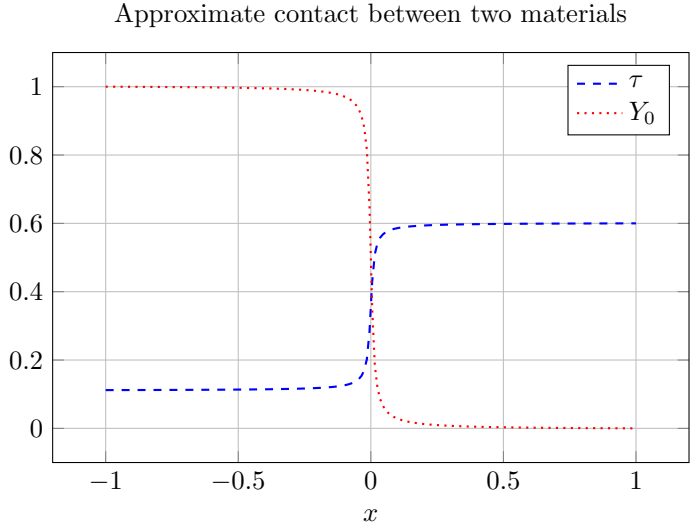
\begin{figure}
\centering
\begin{tikzpicture}
    \begin{axis}[
      width=10cm, height=7cm,
      xlabel={\(x\)},
      title={Approximate contact between two materials},
      legend pos=north east,
      legend cell align=left,
      grid=major,
    ]
        \addplot[domain=-1:1,samples=200,thick,dashed,blue] {atan_interp(0.112, 0.6, 100.0, x)};
        \addlegendentry{$\tau$}
        \addplot[domain=-1:1,samples=200,thick,dotted,red] {atan_interp(1.0, 0.0, 100.0, x)};
        \addlegendentry{$Y_0$}
    \end{axis}
\end{tikzpicture}
\caption{Plots of the approximate transition for $\tau$ and $Y_0$ defined $C(x;y_0, y_1)$ in \eqref{eq:contact_func} for the test problem in Section~\ref{sec:macaw_ideal_transition1}.}
\label{fig:approx_contact_tau}
\end{figure}

\subsubsection{Two-state: Stiffened \& ideal gas contact} \label{sec:stiff_ideal_contact}
For the first test, try a contact between ideal and stiffened gas laws; that is, the pressures are equal on both sides.
The specific data we use is:
\begin{subequations}
\begin{align}
    (P_L, T_L, \bY_L) &= (\SI{1.01e-4}{\giga\pascal}, \SI{1000}{\kelvin}, (1,0)), \\
    (P_R, T_R, \bY_R) &= (\SI{1.01e-4}{\giga\pascal}, \SI{300}{\kelvin}, (0,1)),
\end{align}
\end{subequations}
where the left state is the stiffened gas law and the right state is the ideal gas law.
This smooth mixture transition approximating the contact is similar to what may appear from numerical viscosity in a first order method when solving the four-equation model.
In this case, we see the PTE solutions plotted in Figure~\ref{fig:PTE_trajectories1} throughout the transition.
The statistics are reported in Table~\ref{tab:stiff_ideal_transition1}.
Interestingly the Newton-bisection method struggles for this problem (slow convergence).
\begin{table}[]
    \centering
    \resizebox{\textwidth}{!}{
    \begin{tabular}{|c|c|c|c|c|}
    \hline
    \multicolumn{5}{|c|}{\textbf{Stiffened gas/Ideal gas contact}} \\
    \hline
     & Cyclic & 2D Newton & Damped Newton & Newton+Bisect \\
    \hline
    Failure rate & 0\% & 0\% & 0\% & 1\% \\
    \hline
    Avg. CPU time & \SI{1.0e-5}{\second} & \SI{4.7e-6}{\second} & \SI{7.0e-6}{\second} & \SI{5.0e-3}{\second} \\
    \hline
    Avg. EOS evals & 8.7 & 3.6 & 6.2 & 1161 \\
    \hline
    \end{tabular}
    }
    \caption{Statistics for the contact problem between a stiffened gas and an ideal gas in Section~\ref{sec:stiff_ideal_contact}.
    Note that the failure in the Newton-bisection method was due to slow convergence to the root.
    Increasing the number of iterations allows the Newton-bisection method to converge within the tolerance.}
    \label{tab:stiff_ideal_transition1}
\end{table}

\subsubsection{Two-state: Simple MACAW \& ideal gas} \label{sec:macaw_ideal_transition1}
For the next test, we use the Simple MACAW for the left state and ideal gas for the right state.
The left and right data are provided by,
\begin{subequations}
\begin{align}
    (P_L, T_L, \bY_L) &= (\SI{1.01e-4}{\giga\pascal}, \SI{300}{\kelvin}, (1,0)), \\
    (P_R, T_R, \bY_R) &= (\SI{1}{\giga\pascal}, \SI{1500}{\kelvin}, (0,1)).
\end{align}
\end{subequations}
A plot of the PTE solution along the transition is provided in Figure~\ref{fig:PTE_trajectories1} and the statistics is reported in Table~\ref{tab:macaw_ideal_transition1}.
Each method performs well on this problem.
\begin{table}[]
    \centering
    \resizebox{\textwidth}{!}{
    \begin{tabular}{|c|c|c|c|c|}
    \hline
    \multicolumn{5}{|c|}{\textbf{Simple MACAW/Ideal gas transition}} \\
    \hline
     & Cyclic & 2D Newton & Damped Newton & Newton+Bisect \\
    \hline
    Failure rate & 0\% & 0\% & 0\% & 0\% \\
    \hline
    Avg. CPU time & \SI{8.9e-6}{\second} & \SI{4.5e-6}{\second} & \SI{7.0e-5}{\second} & \SI{6.8e-4}{\second} \\
    \hline
    Avg. EOS evals & 7.4 & 3.5 & 6.0 & 149 \\
    \hline
    \end{tabular}
    }
    \caption{Statistics for the transition between simple MACAW and ideal gas problem in Section~\ref{sec:macaw_ideal_transition1}.
    All methods perform well on this problem.}
    \label{tab:macaw_ideal_transition1}
\end{table}
\begin{figure}
    \centering
    \includegraphics[width=0.49\linewidth]{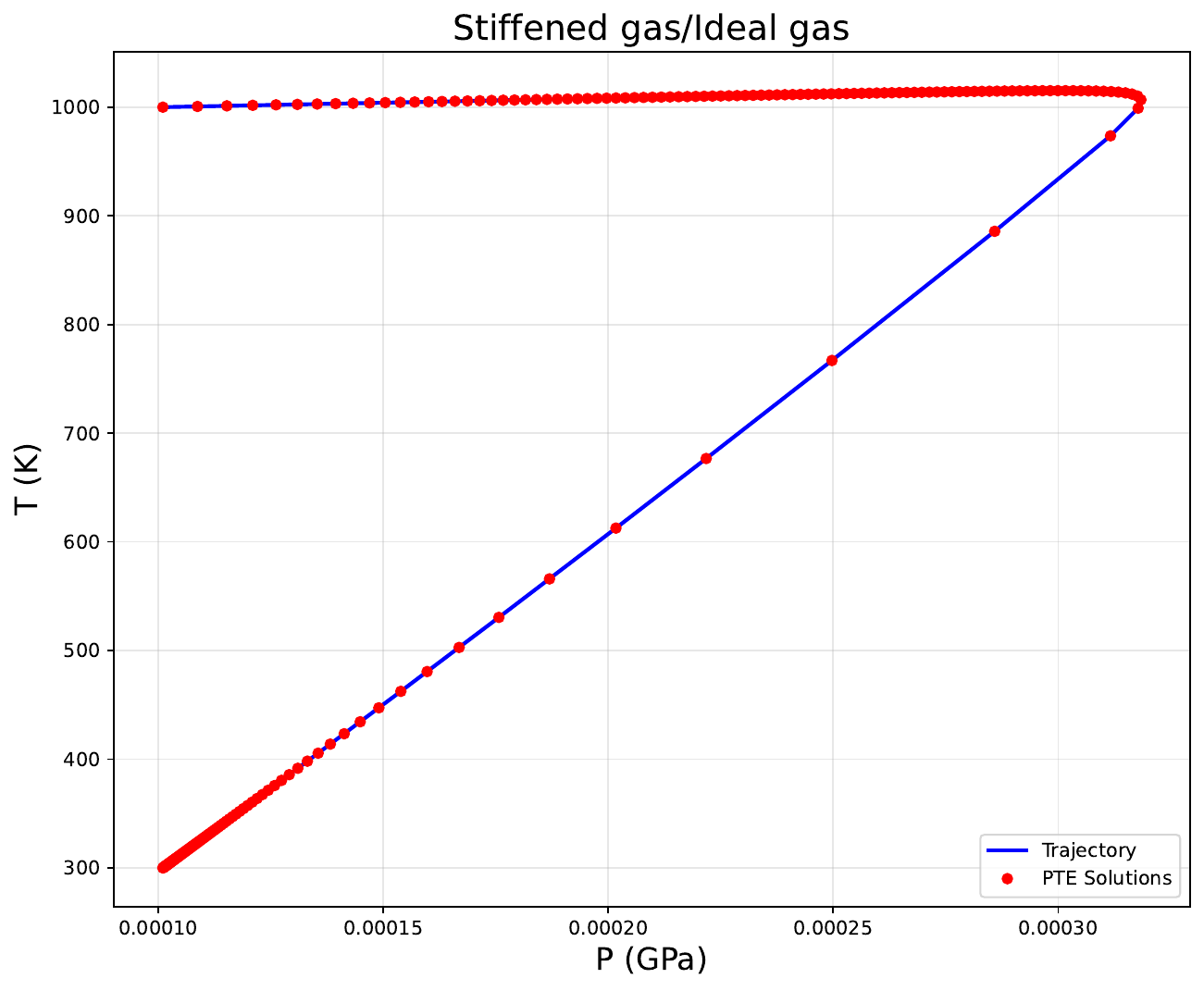}
    \includegraphics[width=0.49\linewidth]{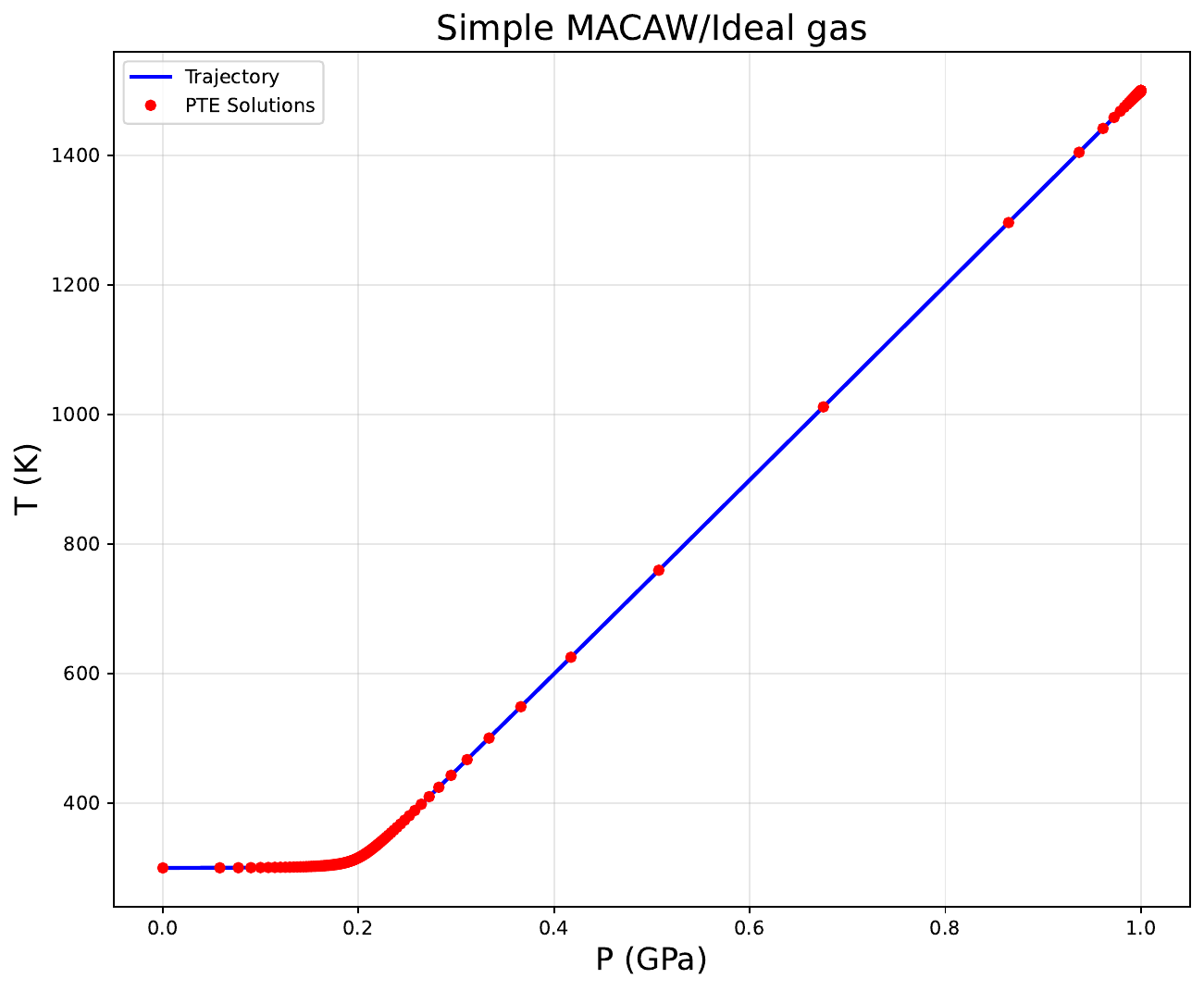}
    \includegraphics[width=0.49\linewidth]{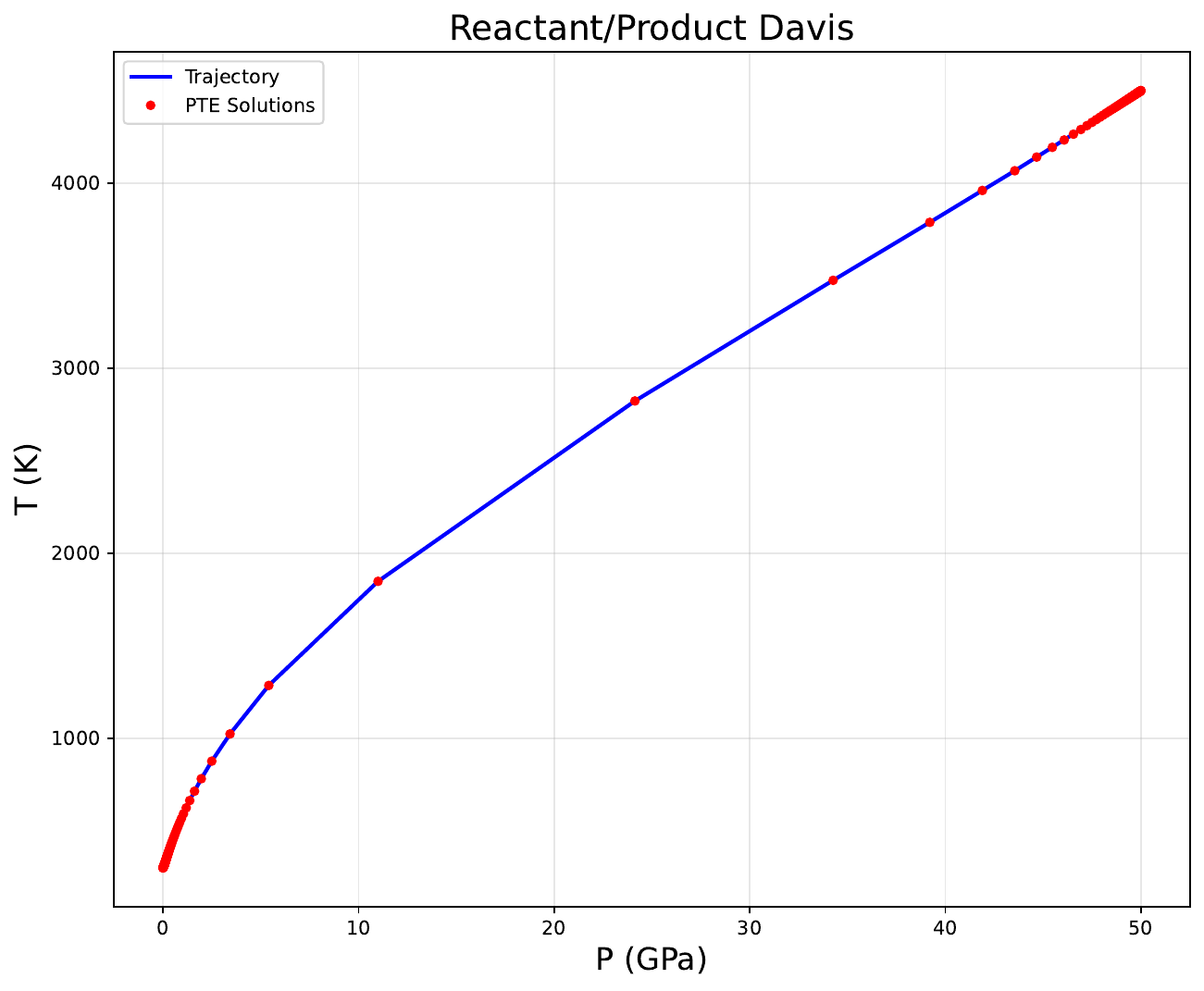}
    \caption{The PTE solutions for the problems in Section~\ref{sec:stiff_ideal_contact} (top left), Section~\ref{sec:macaw_ideal_transition1} (top right), and Section~\ref{sec:davis_transition1} (bottom).}
    \label{fig:PTE_trajectories1}
\end{figure}

\subsubsection{Two-state: Reactant \& product Davis} \label{sec:davis_transition1}
Next, we consider a transition between Davis reactants to Davis products.
The left and right data are:
\begin{subequations}
\begin{align}
    (P_L, T_L, \bY_L) &= (\SI{1.01e-4}{\giga\pascal}, \SI{300}{\kelvin}, (1,0)), \\
    (P_R, T_R, \bY_R) &= (\SI{50}{\giga\pascal}, \SI{4500}{\kelvin}, (0,1)).
\end{align}
\end{subequations}
The statistics are reported in Table~\ref{tab:davis_transition1} and the plot of the PTE solutions are provided in Figure~\ref{fig:PTE_trajectories1}.
\begin{table}[]
    \centering
    \resizebox{\textwidth}{!}{
    \begin{tabular}{|c|c|c|c|c|}
    \hline
    \multicolumn{5}{|c|}{\textbf{Reactant/product Davis transition}} \\
    \hline
     & Cyclic & 2D Newton & Damped Newton & Newton+Bisection \\
    \hline
    Failure rate & 0\% & 0\% & 0\% & 11\% \\
    \hline
    Avg. CPU time & \SI{1.4e-5}{\second} & \SI{4.3e-6}{\second} & \SI{6.6e-6}{\second} & \SI{1.7e-3}{\second} \\
    \hline
    Avg. EOS evals & 12 & 3.3 & 5.5 & 467 \\
    \hline
    \end{tabular}
    }
    \caption{Statistics for the transition between reactant and product Davis EOS problem in Section~\ref{sec:davis_transition1}.}
    \label{tab:davis_transition1}
\end{table}
%
\subsection{Grid test}
We now illustrate the robustness of the nonlinear solves with a ``grid test''.
We start with the true pressure temperature solution, $(P_0, T_0)$, as well as the mass fractions, $\bY_0$, then compute the corresponding data: $\tau_0 = \tau(P_0, T_0, \bY_0)$ and $e_0 = e(P_0, T_0, \bY_0)$.
We run each numerical method with the initial guess provided at each grid point, $(P_i, T_j)$, in our $P$-$T$ domain and then count the number of iterations it took for convergence to be achieved.
If the method exceeds 100 steps then we flag it as a failure.

\subsubsection{Test 1} \label{sec:grid_ideal_stiff}
We use the ideal gas ($Y_0 = 0.5$) and the stiffened gas ($Y_1 = 0.5$); the true solution is $P_0 = \SI{45}{\giga\pascal}$ and $T_0 = \SI{0.15}{\kelvin}$.
The pressure-temperature domain is $\sfR = [\SI{e-9}{\giga\pascal}, \SI{e3}{\giga\pascal}] \times [\SI{0.001}{\kelvin}, \SI{e5}{\kelvin}]$.
Plots of the iteration counts are shown in Figure~\ref{fig:grid_test_macaw_product}.
Note the ``striping'' effect of the Newton-bisection method; this is reproduced in all test problems as the method fundamentally reduces the 2D problem to a 1D problem.
\begin{figure}
    \centering
    \includegraphics[width=0.48\linewidth]{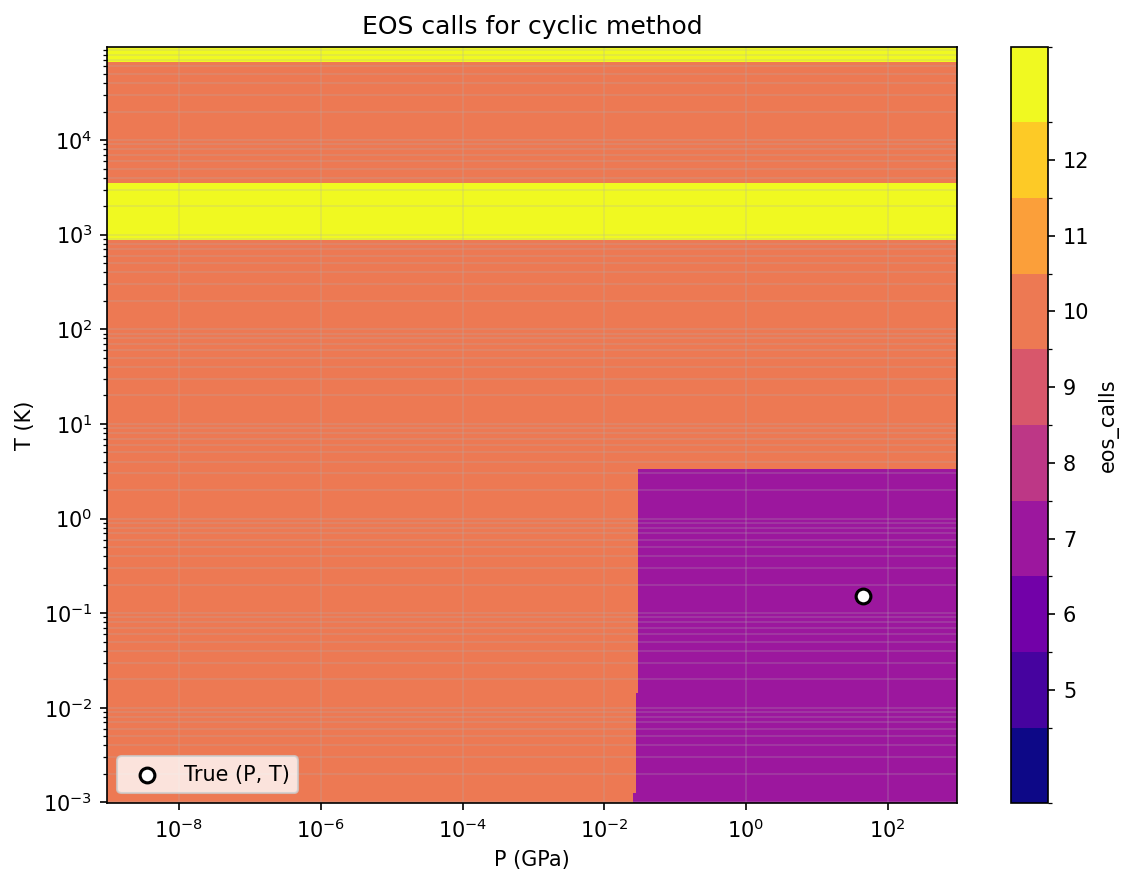}
    \includegraphics[width=0.48\linewidth]{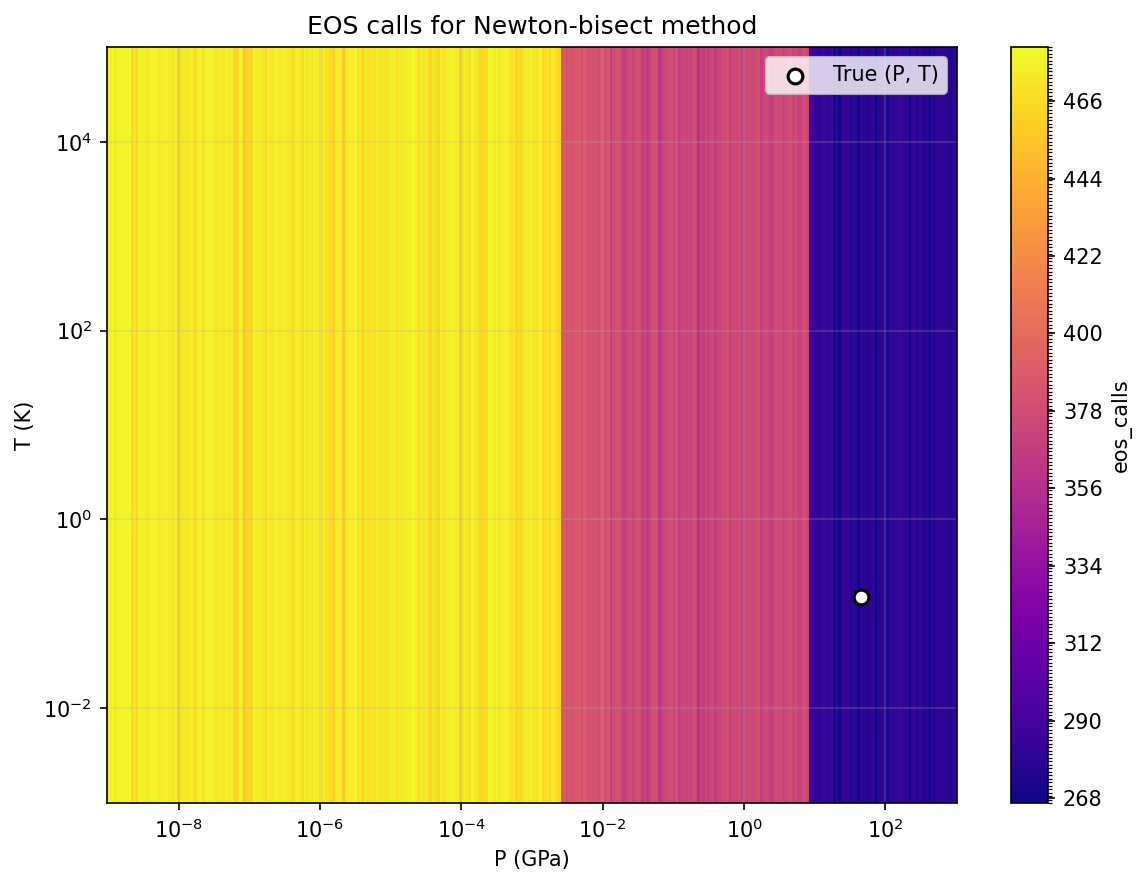}
    \includegraphics[width=0.48\linewidth]{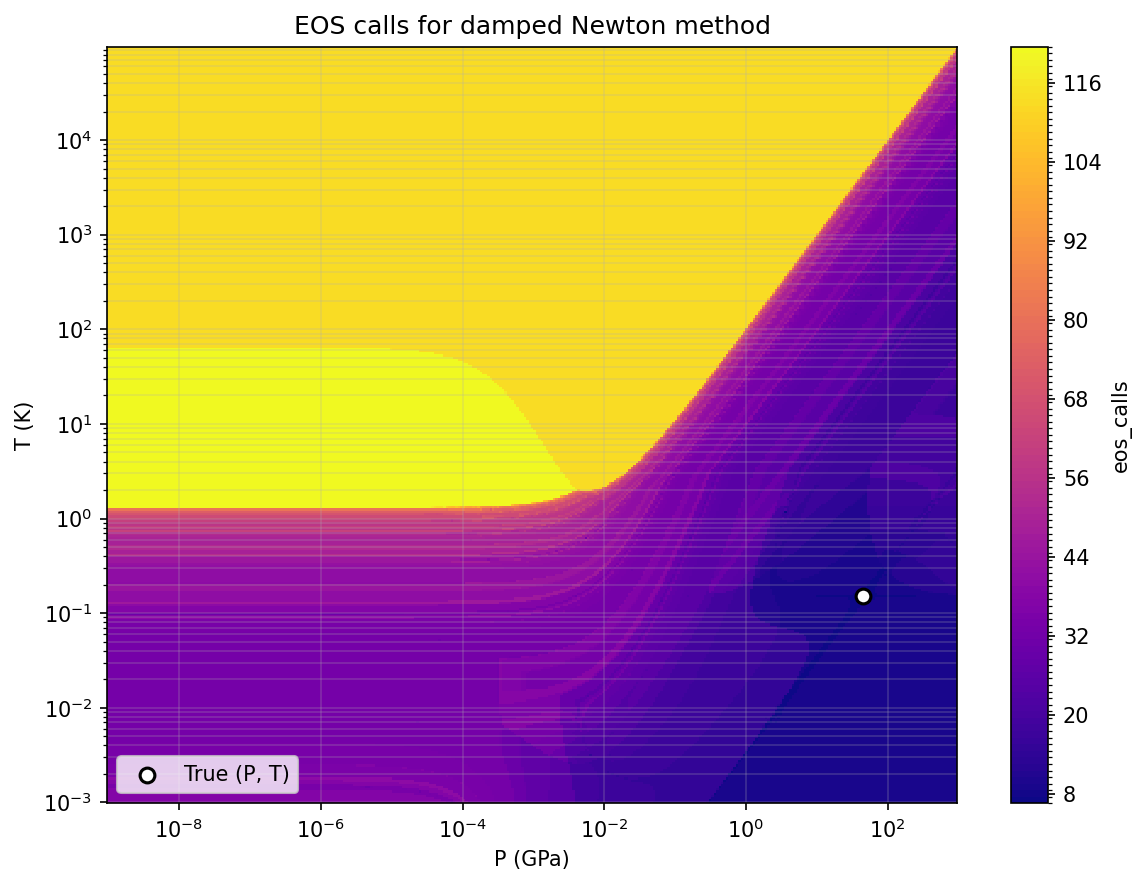}
    \includegraphics[width=0.48\linewidth]{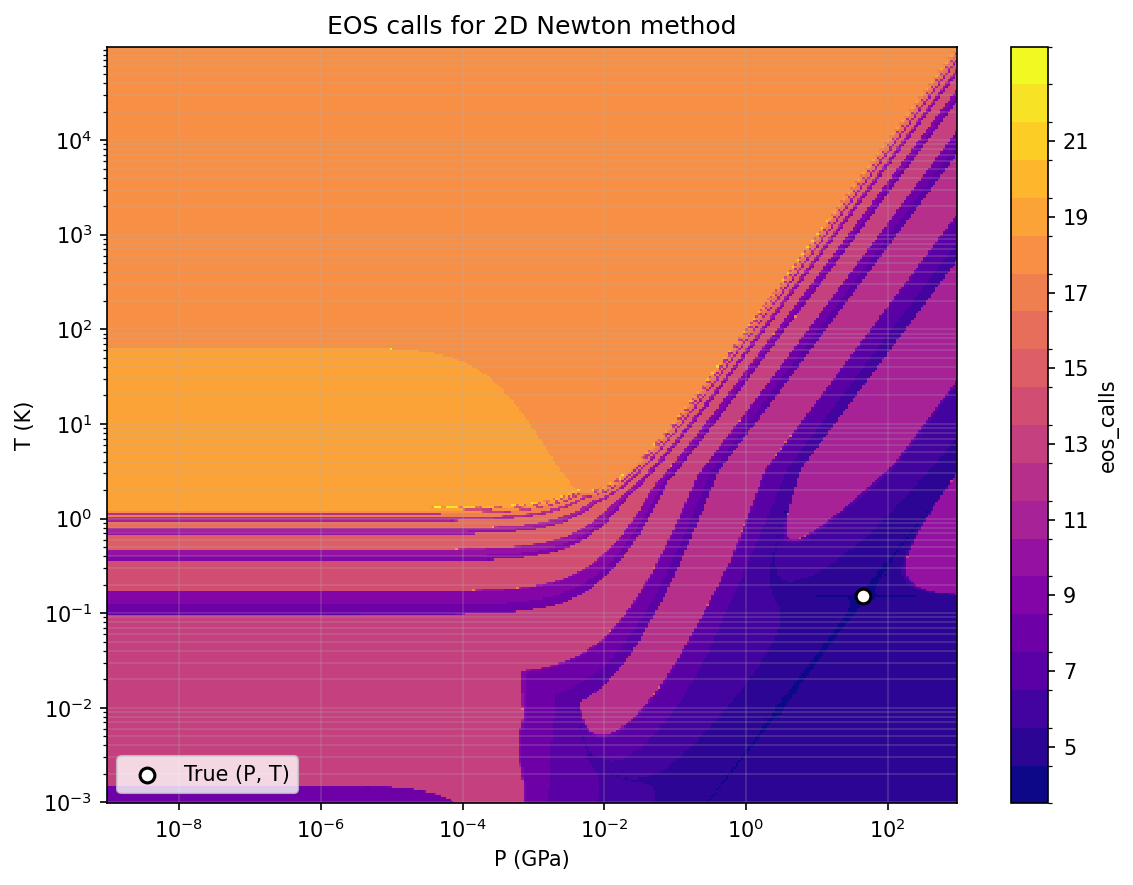}
    \caption{Plots of the number of EOS calls for each starting point for the 50/50 mixture of ideal and stiffened gas EOS.}
    \label{fig:grid_test_macaw_product}
\end{figure}

\subsubsection{Test 2} \label{sec:grid_damped_failure}
We now consider a similar test case as the one provided in Section~\ref{sec:grid_ideal_stiff}; however, we demonstrate that the damped Newton can inhibit the convergence ability of the method.
As before, we use an ideal ($Y_0 = 0.5$) and stiffened gas ($Y_1 = 0.5$) mixture with the true pressure and temperature provided as $P_0 = \SI{e-4}{\giga\pascal}$ and $T_0 = \SI{2e4}{\kelvin}$, respectively.
As in the previous section, we use the same domain, $\sfR = [\SI{e-9}{\giga\pascal}, \SI{e3}{\giga\pascal}] \times [\SI{0.001}{\kelvin}, \SI{e5}{\kelvin}]$.
The EOS call plots are presented in Figure~\ref{fig:grid_test_davis_test1}.
One sees that there is a large region of phase space for which the damped Newton is unable to converge.
Even the cyclic method has a small band at the low temperatures where it is unable to converge from, while the 2D Newton method is robust to all starting points for this problem.
\begin{figure}
    \centering
    \includegraphics[width=0.48\linewidth]{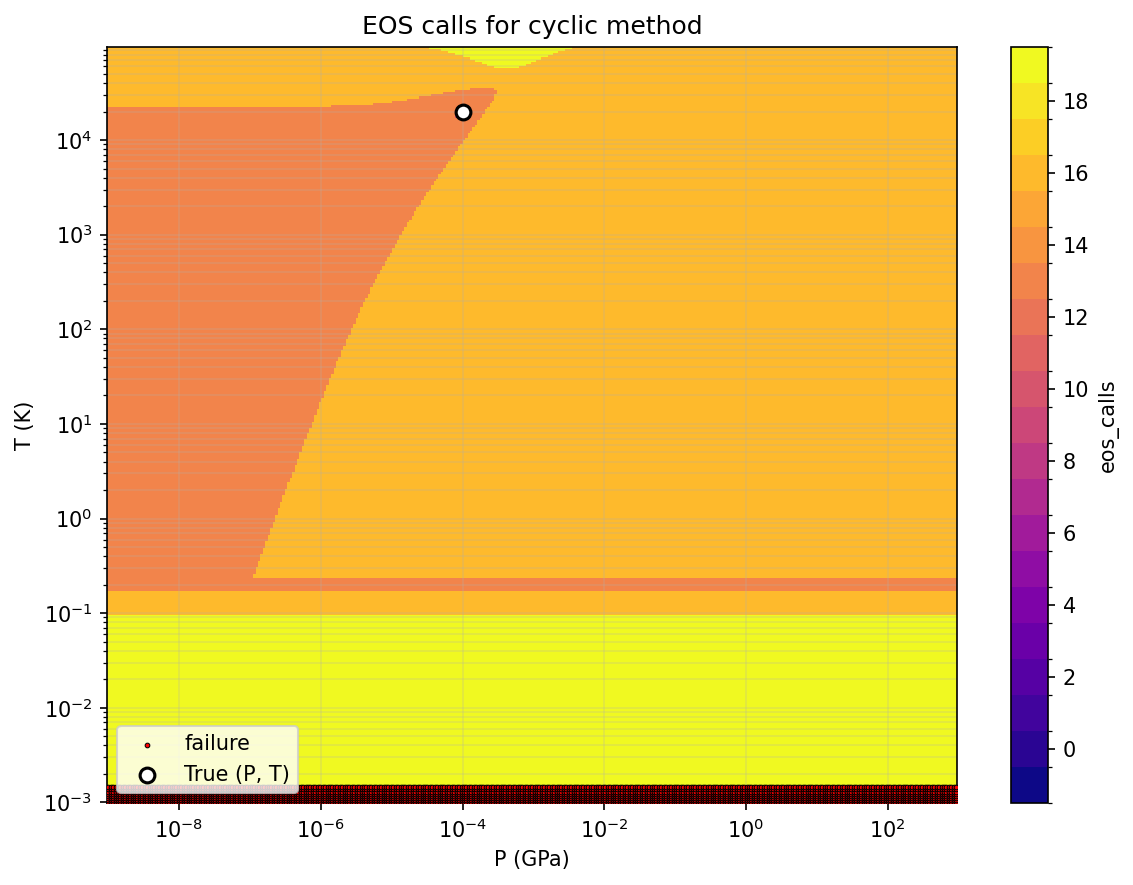}
    \includegraphics[width=0.48\linewidth]{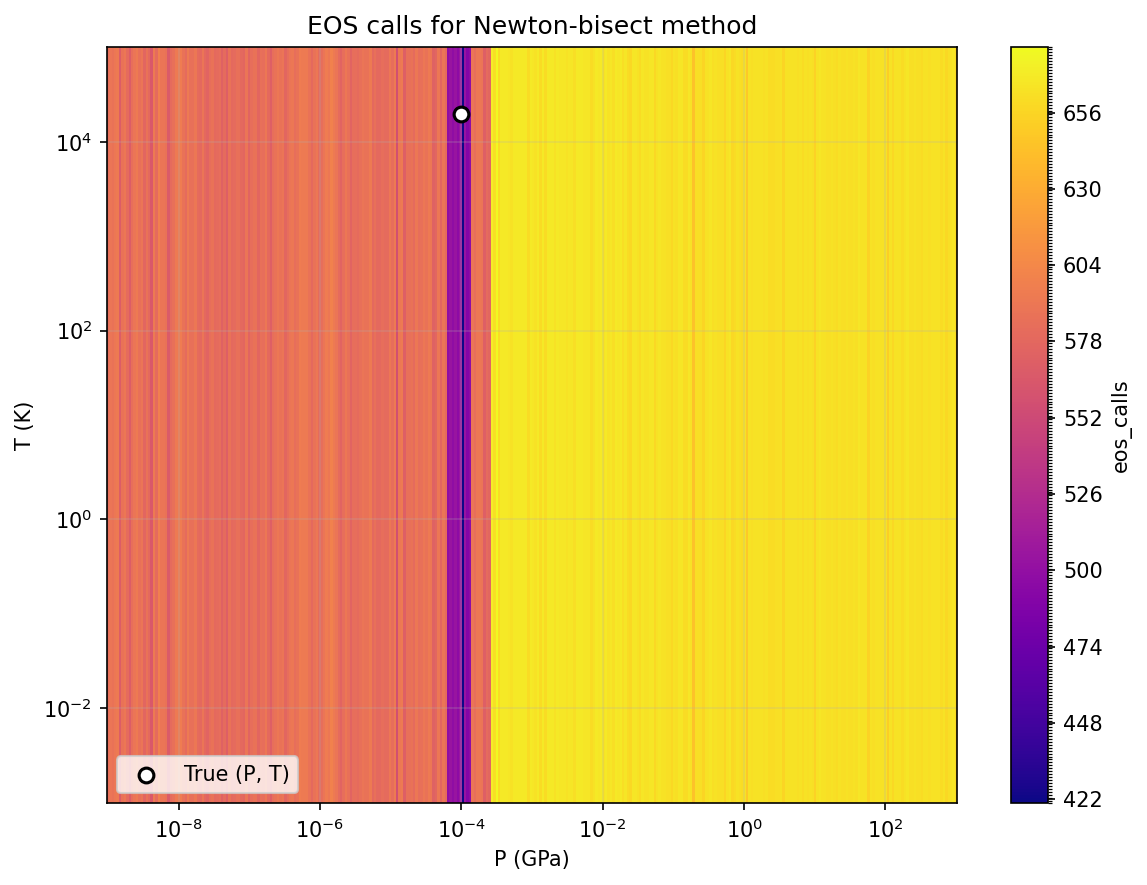}
    \includegraphics[width=0.48\linewidth]{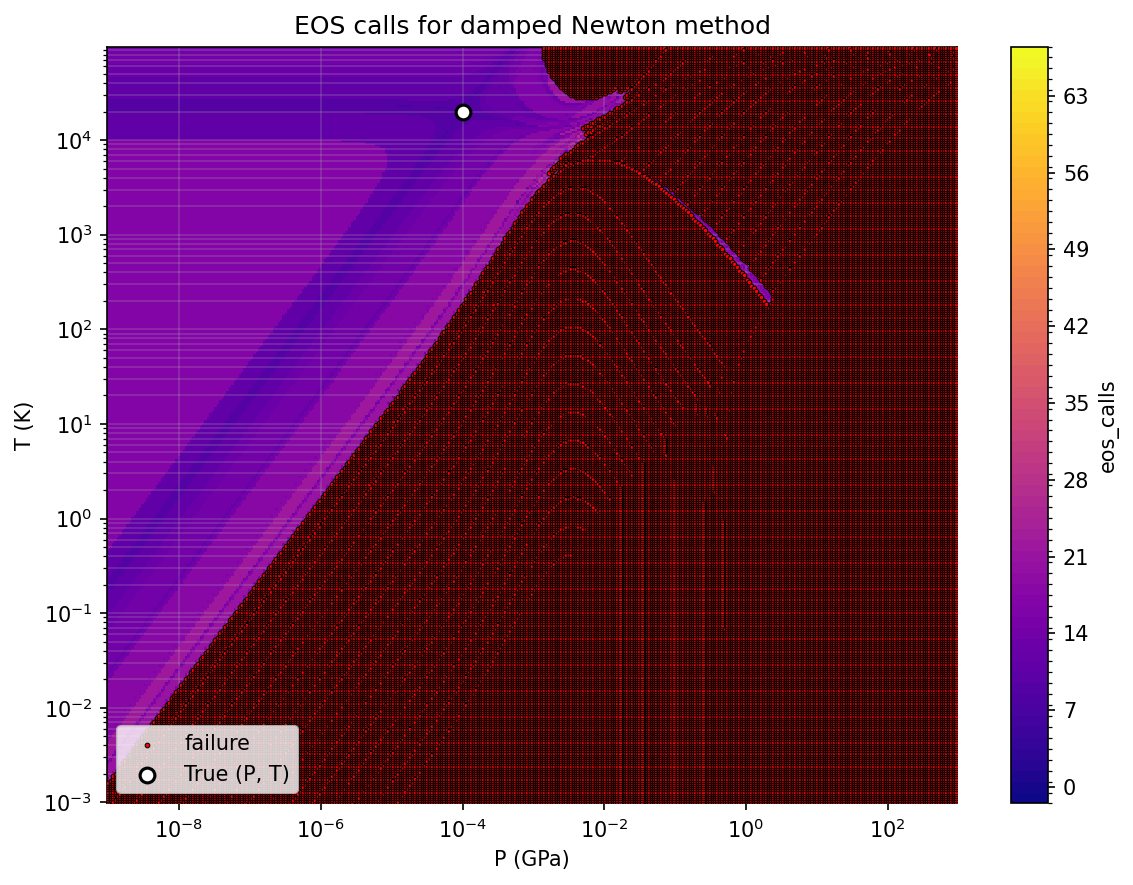}
    \includegraphics[width=0.48\linewidth]{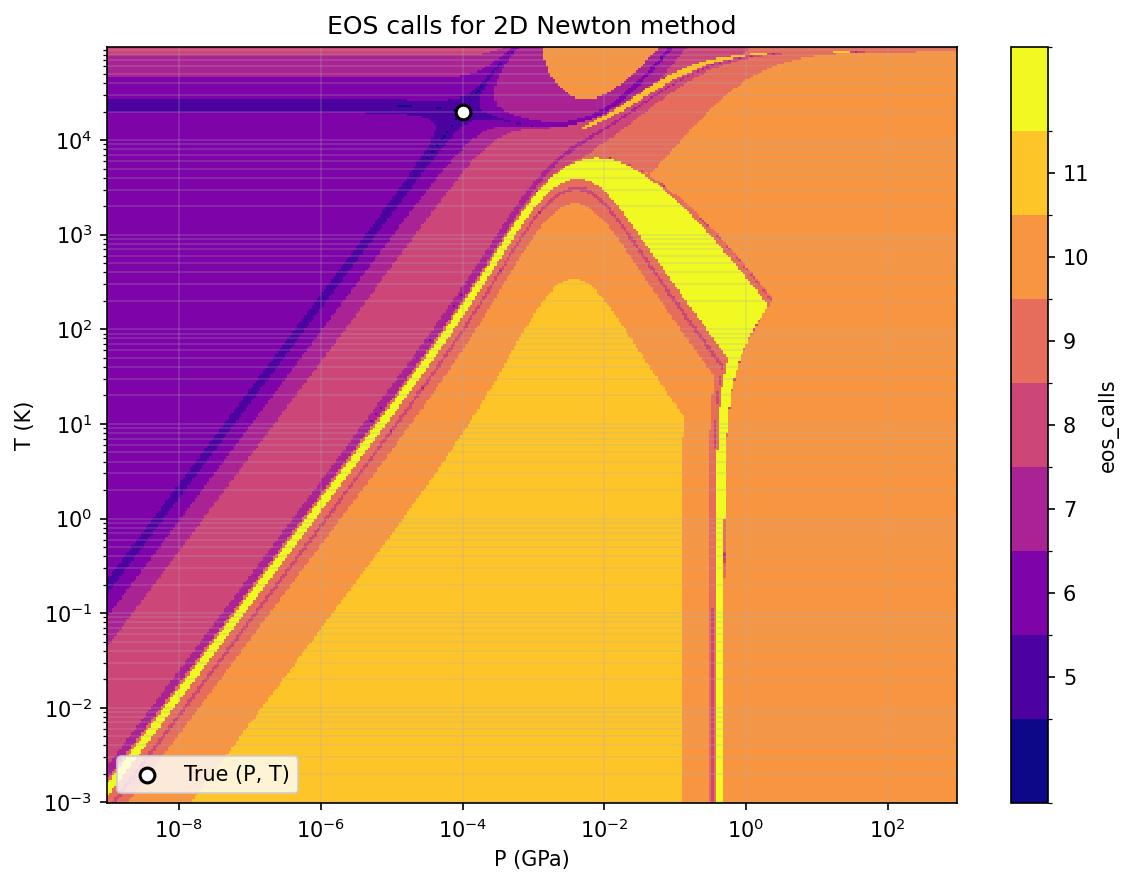}
    \caption{Iteration plots for a 50/50 mixture of ideal and stiffened gas equations of state from Section~\ref{sec:grid_damped_failure}.
    Note the failures in the damped Newton method.}
    \label{fig:grid_test_davis_test1}
\end{figure}

\subsubsection{Test 3} \label{sec:grid_5mat_test}
Next we run a test using ideal gas ($Y_0 = 0.05$), stiffened gas ($Y_1 = 0.05$), simple MACAW ($Y_2 = 0.2$), reactant Davis ($Y_3 = 0.05$), and product Davis ($Y_4 = 0.65$).
The true solution is chosen to be $P_0 = \SI{45}{\giga\pascal}$ and $T_0 = \SI{3000}{\kelvin}$ and the $P$-$T$ domain is provided by, $\sfR = [\SI{e-8}{\giga\pascal}, \SI{500}{\giga\pascal}] \times [\SI{0.001}{\kelvin}, \SI{3e4}{\kelvin}]$.
Results are presented in Figure~\ref{fig:grid_5mat_test}.
\begin{figure}
    \centering
    \includegraphics[width=0.48\linewidth]{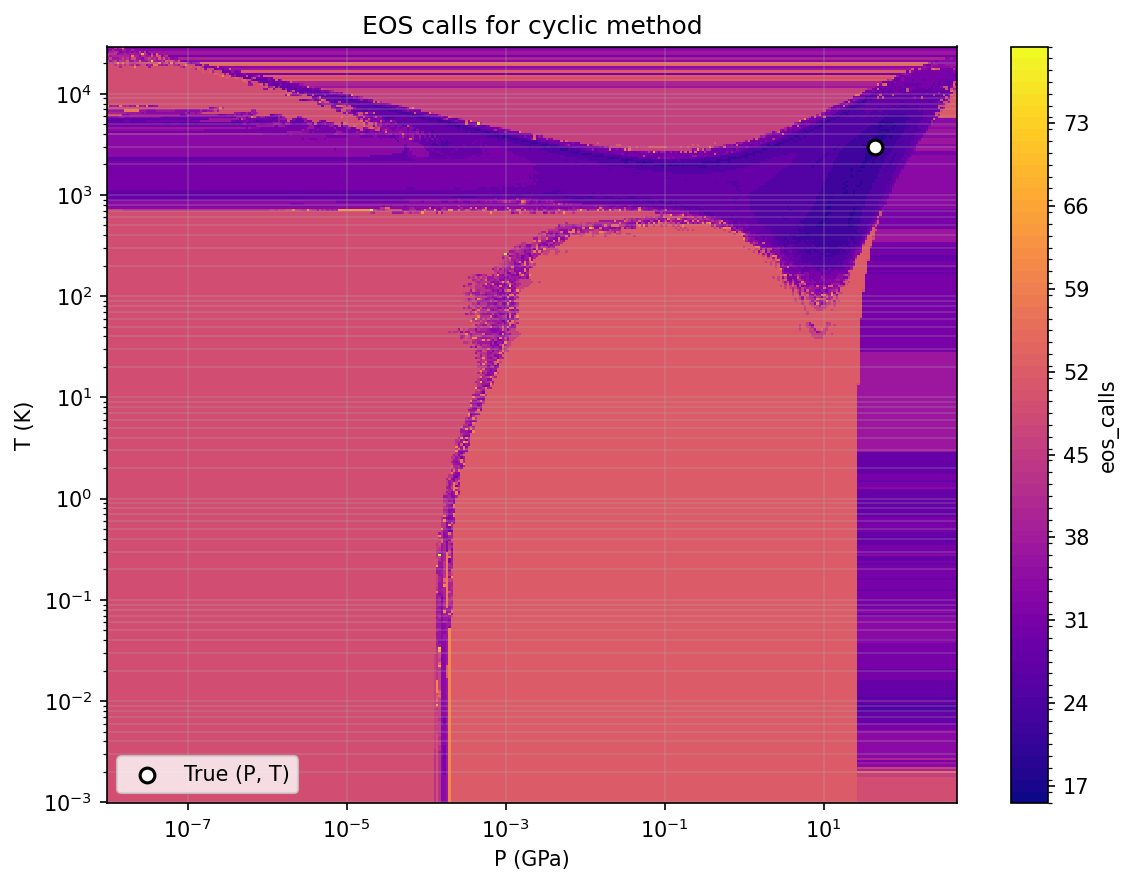}
    \includegraphics[width=0.48\linewidth]{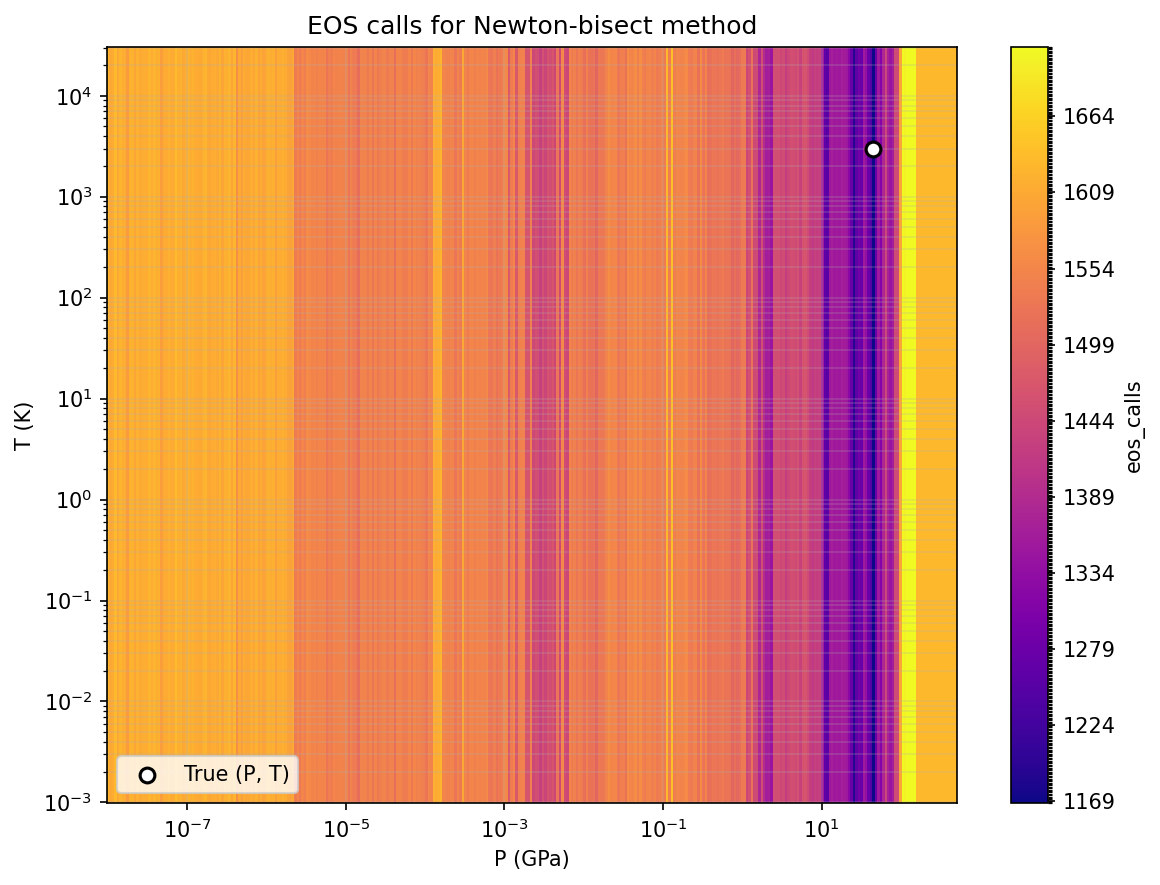}
    \includegraphics[width=0.48\linewidth]{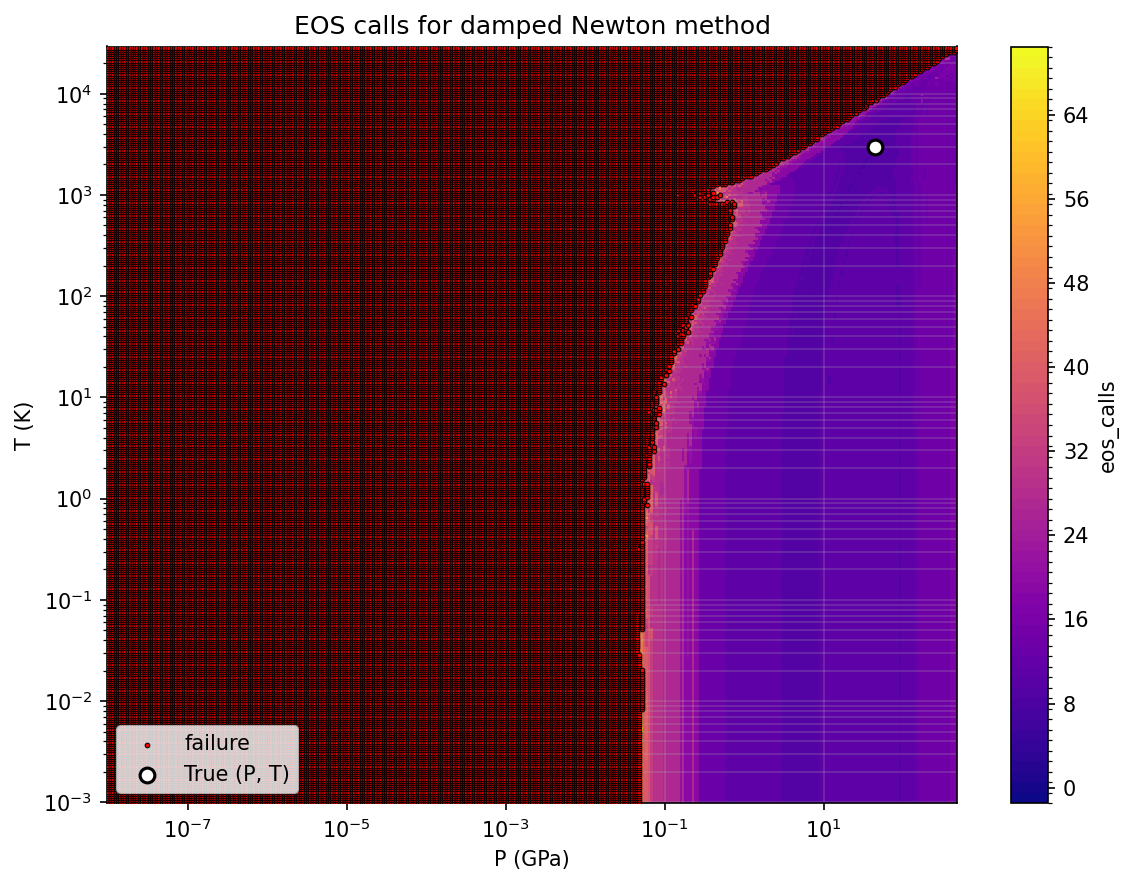}
    \includegraphics[width=0.48\linewidth]{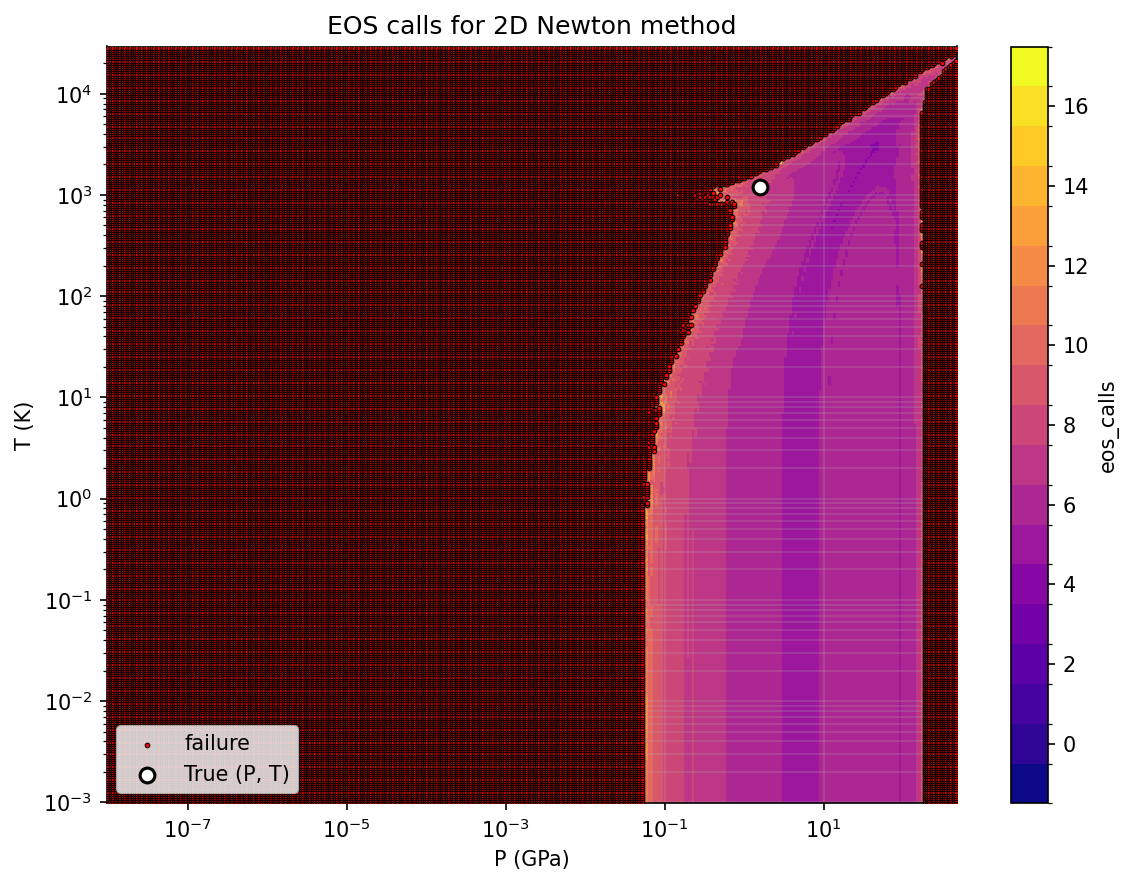}
    \caption{Iteration plots for five material problem in Section~\ref{sec:grid_5mat_test}.
    Results for 2D Newton method omitted due to high failure rate.}
    \label{fig:grid_5mat_test}
\end{figure}

\subsubsection{Test 4} \label{sec:grid_davis_test}
Next we run a test using reactant Davis ($Y_1 = 0.3$) and product Davis ($Y_2 = 0.7$).
The pressure-temperature domain is the same as the one provided in Section~\ref{sec:grid_5mat_test}.
The true solution is chosen to be $P_0 = \SI{1.6}{\giga\pascal}$ and $T_0 = \SI{1200}{\kelvin}$.
Results are presented in Figure~\ref{fig:grid_davis_test}.
We see that all methods except for the Newton-bisection method are not robust to the initial guess for this problem.
We see that the damped Newton appears to be more robust than the regular 2D Newton method, while the cyclic method and the damped Newton method have different regions for which the initial guess will converge.
\begin{figure}
    \centering
    \includegraphics[width=0.48\linewidth]{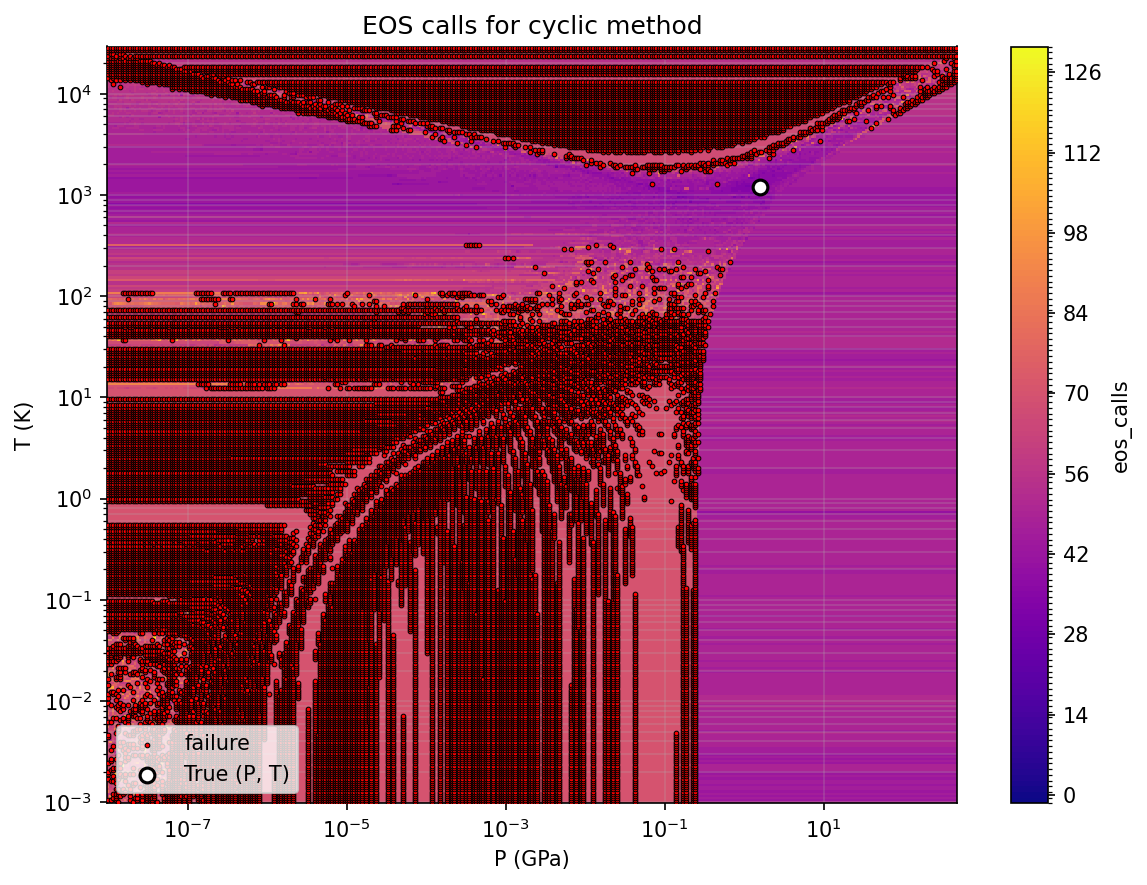}
    \includegraphics[width=0.48\linewidth]{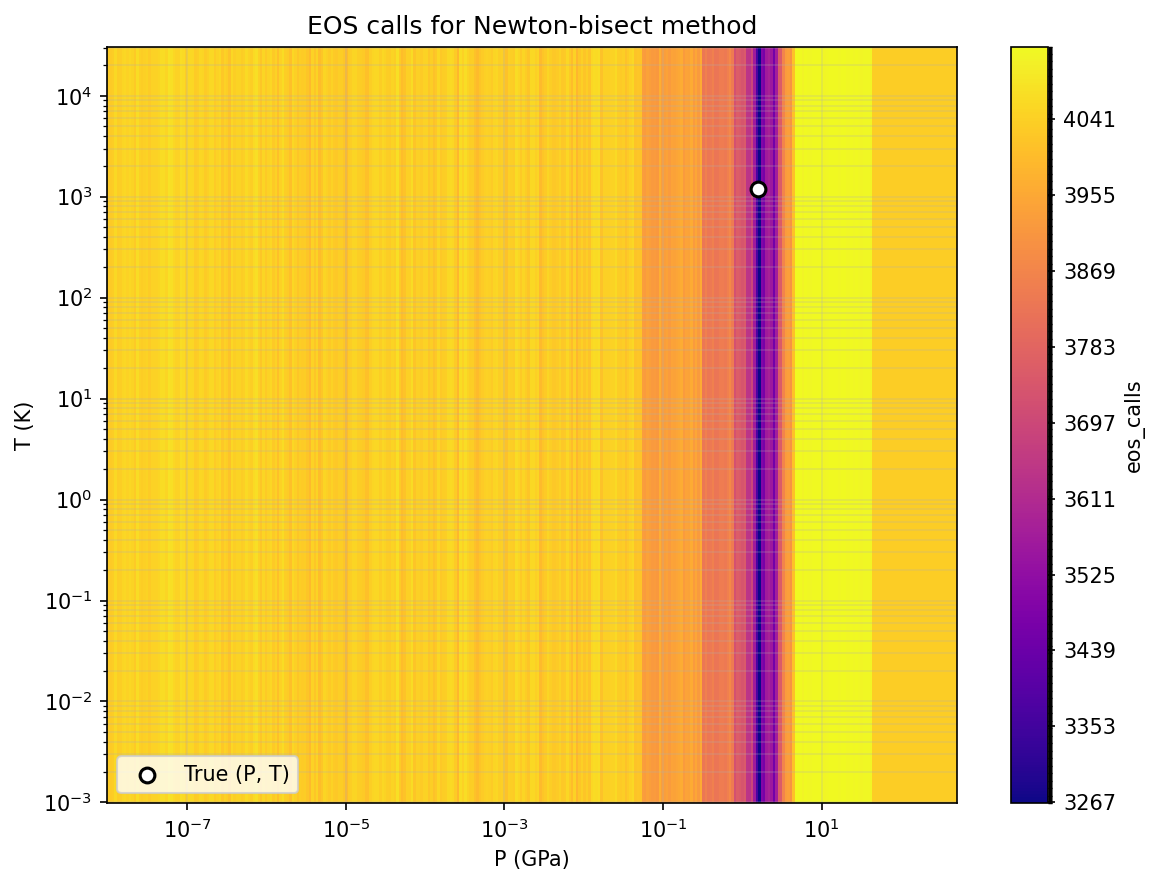}
    \includegraphics[width=0.48\linewidth]{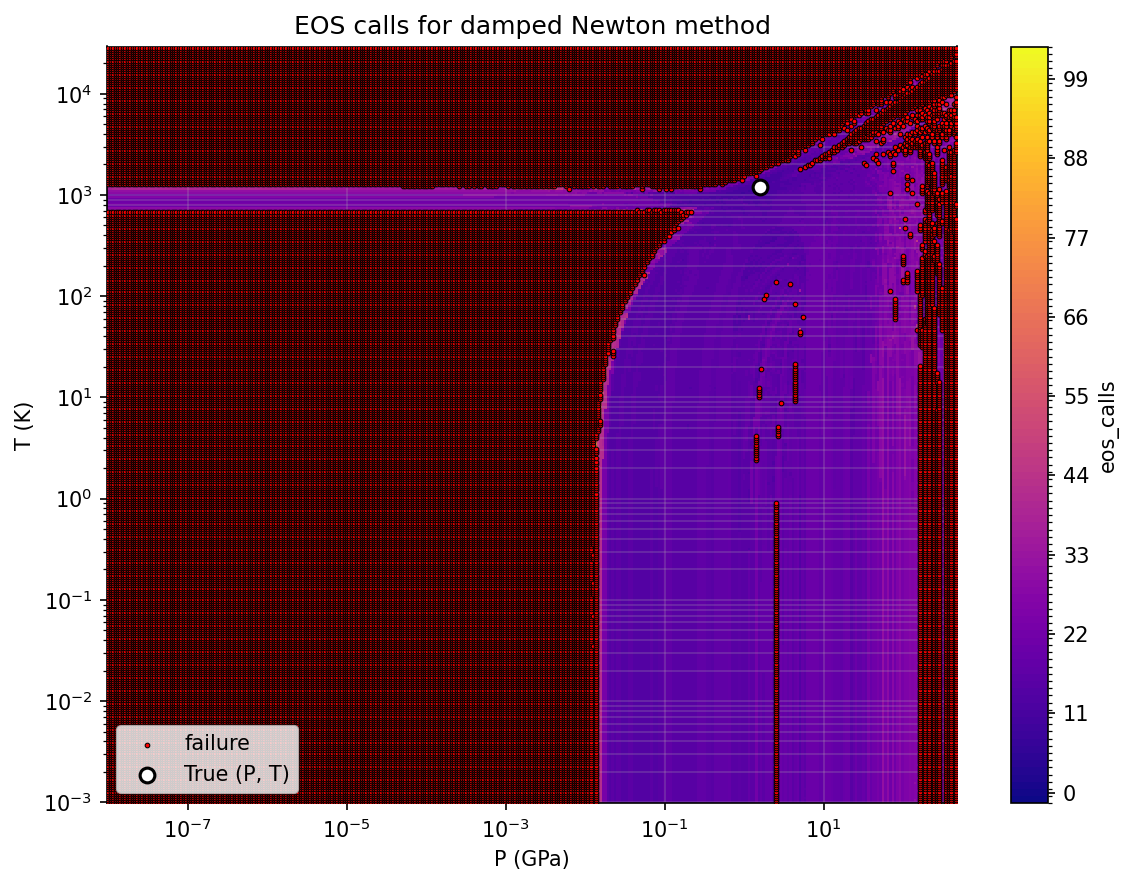}
    \includegraphics[width=0.48\linewidth]{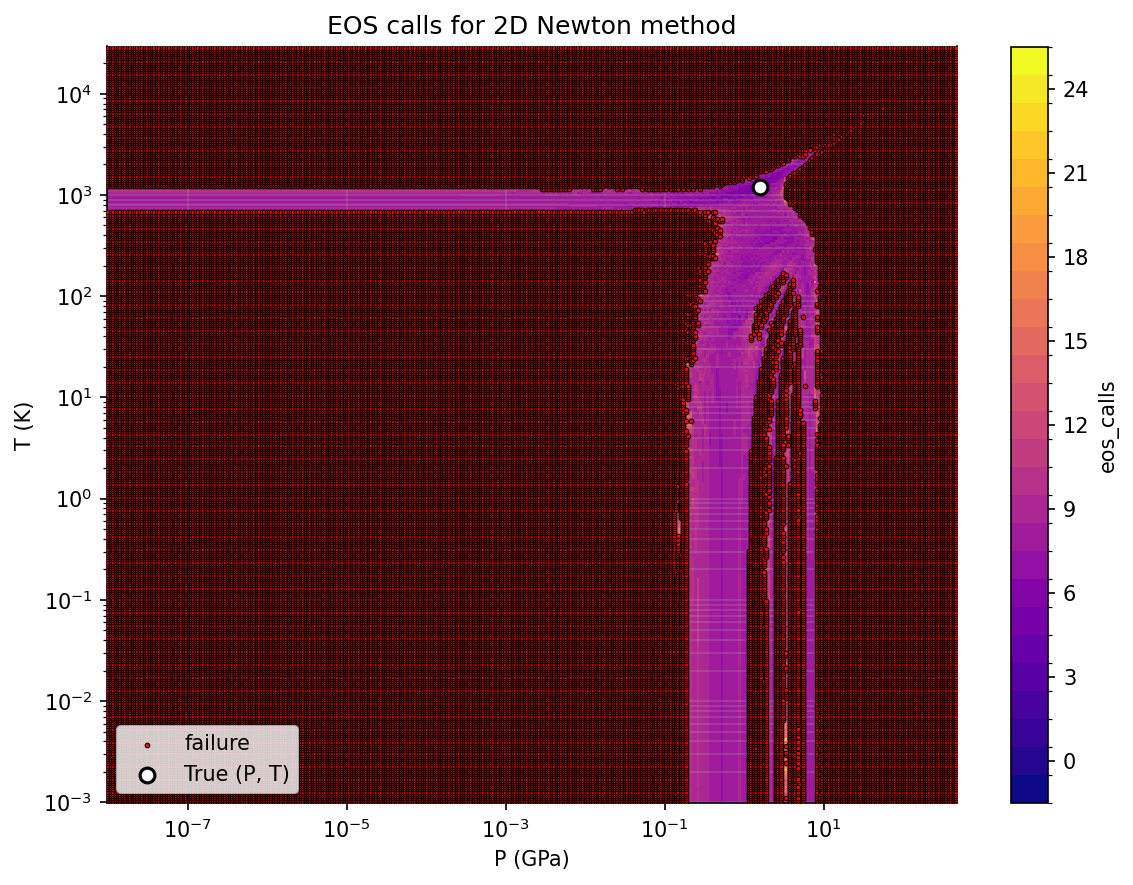}
    \caption{Iteration plots for reactant and product Davis mixture problem from Section~\ref{sec:grid_davis_test}.}
    \label{fig:grid_davis_test}
\end{figure}

%% file: conclusion.tex
\section{Conclusion}
We have analyzed the pressure-temperature equilibrium closure model imposed on the four-equation model for general equations of state that are thermodynamically stable.
We covered the mixture thermodynamic derivatives and various other thermodynamic properties.
In particular, the admissible set was identified in Definition~\ref{def:pte_admissible_set} and was found to be a convex set owing to Theorem~\ref{thm:pte_convex_set}.
Assuming the state vector belongs to the admissible set, the existence and uniqueness of solutions to the PTE system was proven in Theorem~\ref{thm:existence}.
These theoretical results establish the foundation for development of structure-preserving/invariant-domain preserving numerical methods.

We then introduced a possible method for constructing a tabular approximation of an equation of state in Section~\ref{sec:tabular_approximation} in order to avoid expensive solves for $\rho(P,T)$.
A transformation of the PTE system is provided which aides the nonlinear solvers.
Then a novel general nonlinear solver was introduced, the \textit{cyclic method}, which was applied to the PTE system.

The robustness and computational efficiency of the cyclic method was demonstrated against other standard iterative solvers.
While the cyclic method is more robust to the starting location, the 2D Newton method is computationally more efficient and is still robust when providing the previous solution to PTE as the initial guess as shown in Section~\ref{sec:transition_problems}.
The closing strategy to solve the PTE system is to first attempt the 2D Newton solve and switch to the cyclic method if the first attempt fails; then as a last resort, one can use a complete bisection method if both methods fail.
Since if the standard 2D Newton method fails, it is able to fail very quickly.
While the cyclic method may not outshine the traditional Newton method, it offers a new perspective on solving nonlinear systems from $\Real^2$ to $\Real^2$ and could be applied to other problems where good initial guess is unavailable.

%% file: appendix.tex
\section{Calculus}
For easy reference, we provide some standard Calculus results in this appendix.
\begin{theorem}[Clairaut-Schwarz] \label{thm:clairaut_shwarz}
    Let $\Omega \subset \Real^2$ be open and $f : \Omega \to \Real$.
    Let $(x_0, y_0) \in \Omega$, if $\partial_{x} \partial_{y} f$ and $\partial_y \partial_x f$ exist and are continuous in an open neighborhood, $\calU \subset \Omega$, about $(x_0, y_0)$. 
    Then the mixed partial derivatives commute at $(x_0, y_0)$; that is,
    \begin{equation}
        \frac{\partial}{\partial x} \Big( \pdv{f}{y} \Big)_x(x_0, y_0) = \frac{\partial}{\partial y} \Big( \pdv{f}{x} \Big)_y(x_0, y_0).
    \end{equation}
\end{theorem}
The application of Theorem~\ref{thm:clairaut_shwarz} to the free energies is often referred to as the \textit{Maxwell relations} in the thermodynamics literature.
\begin{theorem}[Implicit function theorem] \label{thm:implicit_func}
    Let $\Omega \subset \Real^{n+m}$ be open and $f : \Omega \to \Real^m$ is some function.
    For $(\bx_0, \by_0) \in \Omega$, assume that $\det(D_\by f(\bx_0, \by_0)) \neq 0$, $f(\bx_0, \by_0) = 0$, and that there exists an open neighborhood, $\calU \subset \Omega$ about $(\bx_0, \by_0)$ such that $f \in C^1(\calU)$.
    Then there exists open neighborhoods $\calV \subset \Real^{n}$ and $\calW \subset \Real^m$ with $\bx_0 \in \calV$ and $\by_0 \in \calW$ and a unique function, $g \in C^1(\calV; \calW)$ such that, $g(\bx_0) = \by_0$ and $f(\bx, g(\bx)) = 0$ for all $\bx \in \calV$.
\end{theorem}
A consequence of the implicit function theorem is the \textit{cyclic rule} (also referred to as the \textit{triple product rule}).
\begin{corollary}[Cyclic rule] \label{cor:cyclic_rule}
    Under the assumptions of Theorem~\ref{thm:implicit_func} for $n = m = 1$ and assuming that $\partial_x f(x_0,y_0) \neq 0$,
    there exists $\widetilde{\calV} \subset \calV$ such that following holds
    \begin{equation} \label{eq:cyclic_rule}
        \frac{\big( \pdv{f}{y} \big)_x \big( \pdv{y}{x} \big)_f}{\big( \pdv{f}{x} \big)_y } = -1,
    \end{equation}
    where $y = g(x)$ for all $x \in \widetilde{\calV}$ and $y \in g(\widetilde{\calV})$.
    Alternatively, if $\partial_x f(x_0, y_0) = 0$, then $\big(\pdv{y}{x}\big)_f = g'(x_0) = 0$.
\end{corollary}

\begin{remark}[Cyclic rule style]
    Typically, the cyclic rule is written as 
    \begin{equation}
        \Big( \pdv{f}{y} \Big)_x \Big( \pdv{x}{f} \Big)_y \Big( \pdv{y}{x} \Big)_f = -1.
    \end{equation}
    We have written it in the form \eqref{eq:cyclic_rule} so as to avoid some confusion and abuse of notation with $x$.
\end{remark}